\documentclass[10pt,a4paper]{article}
\usepackage{amstext}
\usepackage{array}
\usepackage{amsfonts}
\usepackage{amssymb}
\usepackage{graphicx}
\usepackage{epstopdf}
\usepackage{algorithm}     
\usepackage{algpseudocode} 
\usepackage{url}
\usepackage{caption}
\usepackage{subfig}
\usepackage{csquotes}
\usepackage{mathtools}
\usepackage{bbm}
\usepackage{amsthm}
\usepackage{stmaryrd}

\usepackage{tikz}
\usepackage{pgfplots}
\usetikzlibrary{patterns}
\usetikzlibrary{arrows}
\usepgfplotslibrary{external}
\DeclareRobustCommand{\tikzcaption}[1]{\tikzset{external/export next=false}#1}
\DeclareRobustCommand{\tikzref}[1]{\tikzcaption{\resizebox{!}{\refsize}{\ref{#1}}}}

\newcommand{\vertiii}[1]{{\left\vert\kern-0.25ex\left\vert\kern-0.25ex\left\vert #1 
   \right\vert\kern-0.25ex\right\vert\kern-0.25ex\right\vert}}

\newtheorem{Definition}{Definition}[section]
\newtheorem{theorem}{Theorem}[section]
\newtheorem{theorem1}{Theorem}[section]
\newtheorem{theorem2}{Theorem}[section]
\newtheorem{proposition}[theorem]{Proposition}
\newtheorem{lemma}[theorem1]{Lemma}
\newtheorem{remark}[theorem2]{Remark}
\newtheorem{corollary}[theorem]{Corollary}

\newtheorem{hypothesis}{Hypothesis}  

\usepackage{geometry}
\makeatletter
\def\blfootnote{\xdef\@thefnmark{$\star$}\@footnotetext}
\makeatother
\newenvironment{Authors}%
  {\begin{center}\begin{bfseries}}%
  {\end{bfseries}\end{center}}
\newenvironment{Addresses}%
  {\begin{flushleft}\begin{itshape}}%
  {\end{itshape}\end{flushleft}}
  \newcommand{\email}[1]{\hspace*{\stretch{1}}\emph{\texttt{#1}}}

 \usepackage{fancyhdr}  
  \fancypagestyle{plain}{
\fancyhead{}
\fancyhead[C]{\hfill Submitted to SIAM Journal on Scientific Computing, August 2023}
 }
\begin{document}

\thispagestyle{plain}

\title{Compositional maps for registration in complex geometries}
 \date{}
 \maketitle

 \maketitle
\vspace{-50pt} 
 
\begin{Authors}
Tommaso Taddei$^{1}$
\end{Authors}

\begin{Addresses}
$^1$
Univ. Bordeaux, CNRS, Bordeaux INP, IMB, UMR 5251, F-33400 Talence, France\\ Inria Bordeaux Sud-Ouest, Team MEMPHIS, 33400 Talence, France, \email{tommaso.taddei@inria.fr} \\
\end{Addresses}


\begin{abstract}
We develop and analyze a parametric registration procedure for manifolds associated with the solutions to parametric partial differential equations in two-dimensional domains.
Given the domain $\Omega \subset \mathbb{R}^2$ and the manifold $\mathcal{M}=\{ u_{\mu} : \mu\in \mathcal{P}\}$ associated with the parameter domain $\mathcal{P} \subset \mathbb{R}^P$ and the parametric field $\mu\mapsto u_{\mu} \in L^2(\Omega)$, our approach takes as input a set of snapshots from $\mathcal{M}$ and returns a parameter-dependent mapping $\Phi: \Omega \times \mathcal{P} \to \Omega$, which tracks coherent features (e.g., shocks, shear layers) of the solution field and ultimately simplifies the task of model reduction.
We consider mappings of the form $\Phi=\texttt{N}(\mathbf{a})$  where $\texttt{N}:\mathbb{R}^M \to {\rm Lip}(\Omega; \mathbb{R}^2)$ is a suitable linear or nonlinear operator; then, we state the registration problem as an unconstrained optimization statement for the coefficients $\mathbf{a}$. We identify minimal requirements for the operator $\texttt{N}$ to ensure the satisfaction of the bijectivity constraint; we propose a class of compositional maps that satisfy the desired requirements and enable non-trivial deformations over curved (non-straight) boundaries of $\Omega$; 
we develop a thorough analysis of the proposed ansatz for polytopal domains and we discuss the approximation properties for general curved domains. We perform numerical experiments for 
a parametric inviscid transonic compressible flow past a cascade of turbine blades to illustrate the many features of the method.
\end{abstract}

\medskip

 \section{Introduction}
\label{sec:intro}

\subsection{Registration in bounded domains}
Numerical methods based on 
Lagrangian (registration-based) approximations have proven to be a promising research direction in scientific computing, both for the  numerical discretization of partial differential equations (PDEs)
\cite{shubin1981steady,zahr2018optimization}
and also for model order reduction  of parametric systems
(MOR,  \cite{iollo2014advection,mirhoseini2022model,mojgani2017arbitrary,sarna2021data,taddei2020registration}). 
Given the field of interest $u$ defined over the domain $\Omega\subset \mathbb{R}^2$, Lagrangian methods seek approximations of the solution $u$ of  the form $\tilde{u} \circ \Phi^{-1}$ where $\tilde{u}$ belongs to a linear approximation space and $\Phi$ is a bijection  $\Omega$.
This class of methods is designed for problems with sharp features such as shocks, contact discontinuities or shear layers in fluid mechanics and fractures in solid mechanics. The problem of finding the  mapping $\Phi$ is dubbed as registration problem and shares important features with geometry registration 
\cite{ma2018nonrigid,myronenko2010point}
 mesh morphing  \cite{staten2012comparison,tonon2021linear} techniques, and also optimal transportation \cite{blickhan2023registration,jacobs2020fast,peyre2019computational}.
The goal of this paper is to develop a general registration procedure for parametric problems with emphasis on parametric MOR applications, and to provide a  rigorous mathematical analysis of the problem of registration in bounded two-dimensional domains.

We pursue an optimization-based approach to the problem of registration. We denote by $\mu$ a vector of $P$ parameters in 
 the parameter domain $\mathcal{P} \subset \mathbb{R}^P$; given the domain $\Omega\subset \mathbb{R}^2$, we denote by $\mathfrak{B}$ the space of Lipschitz bijections from $\Omega$ in itself, 
 and by $\mathfrak{D}$ the space of diffeomorphisms from $\Omega$ in itself.
Given the mapping $\Phi$, we denote by $J(\Phi)$ the Jacobian determinant, and we denote by $\texttt{id}:\mathbb{R}^2\to \mathbb{R}^2$ the identity map in $\mathbb{R}^2$, $\texttt{id}(x)=x$ for all $x\in \mathbb{R}^2$.
If we fix the value of $\mu\in \mathcal{P}$, our goal is to minimize a target function 
 $\mathfrak{f}_{\mu}^{\rm tg}$ over all possible diffeomorphisms of $\Omega$,
\begin{equation}
\label{eq:optimization_based_registration_ideal}
\min_{\Phi \in \mathfrak{D}} \mathfrak{f}_{\mu}^{\rm tg}(\Phi).
\end{equation}
Problem \eqref{eq:optimization_based_registration_ideal} is computationally intractable due to the fact that  $\mathfrak{D}$ is  an highly non-convex subset of $C^1(\Omega; \mathbb{R}^2)$: research on registration should hence focus on the development of computational strategies to devise tractable counterparts of the statement \eqref{eq:optimization_based_registration_ideal}.

\subsection{Optimization-based registration}
To devise a tractable registration procedure, we introduce an operator $\texttt{N}: \mathbb{R}^M \to {\rm Lip}(\Omega; \mathbb{R}^2)$ and a    penalty function $\mathfrak{f}_{\rm pen}:\mathbb{R}^M \to \mathbb{R}_+$ such that
\begin{equation}
\label{eq:desiderata_mapping_space}
\left\{
\begin{array}{l}
\displaystyle{
\mathcal{B}_{\texttt{N}}
=\left\{
\texttt{N}(\mathbf{a})  \; :  \;
\mathfrak{f}_{\rm pen}(\mathbf{a}) \leq C
\right\}
\subset \mathcal{B},
\quad
{\rm for \; some} \; C>0;
}
\\[3mm]
\displaystyle{
\texttt{N}(\mathbf{a}= 0) = \texttt{id},
\;\;
\mathfrak{f}_{\rm pen}(\mathbf{a}= 0)  <  C;
}
\\[3mm]
\displaystyle{
\texttt{N}, \; \mathfrak{f}_{\rm pen}
\; {\rm are \; Lipschitz \; continuous}.
}
\\
\end{array}
\right.
\end{equation}
Equations \eqref{eq:desiderata_mapping_space}$_1$ and
\eqref{eq:desiderata_mapping_space}$_2$ imply that the set $\mathcal{B}_{\texttt{N}}$ is not empty; furthermore, exploiting \eqref{eq:desiderata_mapping_space}$_3$ we find that the interior of $\mathcal{B}_{\texttt{N}}$ is not empty --- that is, bijectivity of  the mapping $\texttt{N}(\mathbf{a})$  is preserved for small perturbations of the mapping coefficients $\mathbf{a}$.
We further observe that, given the full rank matrix $\mathbf{W}\in \mathbb{R}^{M\times m}$, the pair
$(\widetilde{\texttt{N}}, \widetilde{\mathfrak{f}_{\rm pen}})$ such that 
$\widetilde{\texttt{N}} (\cdot )= \texttt{N}(\mathbf{W} \cdot)$ and
$\widetilde{\mathfrak{f}_{\rm pen}} (\cdot )= \mathfrak{f}_{\rm pen}(\mathbf{W} \cdot)$ satisfies \eqref{eq:desiderata_mapping_space}: 
we can hence
apply linear compression methods
such as proper orthogonal decomposition
(POD, \cite{sirovich1987turbulence,volkwein2011model})
 to the  mapping coefficients $\mathbf{a}$ without fundamentally changing the properties  of our ansatz; as discussed in sections \ref{sec:affine_maps_polytopes} and \ref{sec:methods}, this feature greatly simplifies the task of dimensionality reduction for parametric problems.

Exploiting the previous definitions, we introduce the  surrogate of \eqref{eq:optimization_based_registration_ideal}:
\begin{equation}
\label{eq:tractable_optimization_based_registration}
\min_{   \mathbf{a}  \in \mathbb{R}^M}
 \mathfrak{f}_{\mu}^{\rm obj}(\mathbf{a} ):=
 \mathfrak{f}_{\mu}^{\rm tg}(\texttt{N}( \mathbf{a} ))
\; +\xi  \; 
 \mathfrak{f}_{\rm pen}(\mathbf{a} ).
\end{equation}
Provided that $\xi$ is sufficiently large, solutions $\mathbf{a}^{\star}$ to \eqref{eq:tractable_optimization_based_registration}   satisfy the condition $ \mathfrak{f}^{\rm pen}(\mathbf{a}^{\star} )$ $\leq C$. Note that \eqref{eq:tractable_optimization_based_registration} reads as a nonlinear non-convex unconstrained optimization problem that can be tackled using standard gradient-descent optimization algorithms.

The previous discussion highlights the two major questions in registration methods: 
(i) how to construct 
$(\texttt{N}, \mathfrak{f}_{\rm pen})$ that satisfy \eqref{eq:desiderata_mapping_space} for a given domain $\Omega$;
(ii) how to establish a rigorous relation between the solutions to \eqref{eq:optimization_based_registration_ideal} and to \eqref{eq:tractable_optimization_based_registration}
in the limit $M\to \infty$.
Note that the second question concerns the ability of approximating arbitrary elements of $\mathfrak{D}$ using operators $\texttt{N}$ that satisfy \eqref{eq:desiderata_mapping_space}. We notice that even if we focus on the approximation of diffeomorphisms we allow ourselves to consider approximations in a less regular space: this choice is justified by the particular discretization method (the finite element (FE) method) employed in this work to represent the operator $\texttt{N}$ and, more fundamentally,  by the  strategy  proposed here to define $\texttt{N}$. 

\subsection{Compositional maps for registration}
The objective of this paper is
to devise pairs  $(\texttt{N}, \mathfrak{f}_{\rm pen})$  that satisfy \eqref{eq:desiderata_mapping_space} and study the approximation properties in $\mathfrak{D}$ for arbitrary curved Lipschitz domains $\Omega$.
In this work, 
we say that the boundary  of $\Omega$ is curved
 if it cannot be represented as the   union of a finite number of straight lines; if $\partial \Omega$ is curved, 
  we refer to  $\Omega$ as to ``curved domain''.

Towards this end, we first show that
for polytopal domains
affine maps  
$\texttt{N}_{\rm p}( \mathbf{a}) = 
 \texttt{id} + \sum_{i=1}^M (\mathbf{a})_i \varphi_i$ 
for suitably chosen functions $\{  \varphi_i 
\colon\mathbb{R}^2\to\mathbb{R}^2
\}_i$ and a suitable penalty function  $\mathfrak{f}_{\rm pen}$
satisfy \eqref{eq:desiderata_mapping_space} 
(cf. Proposition \ref{th:bijectivity} 
 and Corollary
\ref{th:easy_bijectivity_consequence});
the function $\{  \varphi_i \}_i$  should be chosen so that 
their normal components vanish on $\partial \Omega$,
 $\varphi_i \cdot \mathbf{n} |_{\partial \Omega} = 0$ for $i=1,\ldots,M$,
 where $\mathbf{n}:\partial \Omega \to \mathbb{S}^1=\{x\in \mathbb{R}^2: \|x\|_2 = 1\}$ is the outward normal to $ \Omega$.
We   also prove that affine maps  of this form
are dense for $M\to \infty$ in a meaningful subspace of diffeomorphisms 
(cf. Proposition \ref{th:vertices_plus_density}  and 
Corollary \ref{th:easy_density_consequence}).

As rigorously shown in Lemma \ref{th:inadequacy_linear_maps}, 
 affine maps are fundamentally ill-suited to approximate non-trivial  diffeomorphisms in curved domains;
 to address this issue,
we propose and analyze compositional maps of the form
\begin{equation}
\label{eq:nonlinear_ansatz}
\texttt{N}(\mathbf{a}) = \Psi \circ  \texttt{N}_{\rm p}(\mathbf{a}) \circ \Psi^{-1},
\end{equation}
where $\Psi: \Omega_{\rm p} \to \Omega$ is a bijection from the polytope $\Omega_{\rm p}$ to $\Omega$ and 
$\texttt{N}_{\rm p}(\mathbf{a}): \Omega_{\rm p} \to    \Omega_{\rm p}$  satisfies
 $\texttt{N}_{\rm p}( \mathbf{a}) = 
 \texttt{id} + \sum_{i=1}^M (\mathbf{a})_i \varphi_i$  with
 $\varphi_i \cdot \mathbf{n} |_{\partial \Omega_{\rm p}} = 0$ for $i=1,\ldots,M$.
We propose an actionable strategy to define 
the polytope
$\Omega_{\rm p}$ and the mapping $\Psi$:
our approach relies on the definition 
of a coarse-grained curved high-order FE mesh $\mathcal{T}$ of $\Omega$, and to a 
 fully-automated procedure to define the polytope $\Omega_{\rm p}$ and the map $\Psi$ based on the curved mesh $\mathcal{T}$.
 
The bijection $\Psi$ is independent of the mapping coefficients $\mathbf{a}$ and is designed to recast the registration problem of interest into a polytope $\Omega_{\rm p}$ where affine maps can be employed.  Figure \ref{fig:geometric_interpretation_compositional_maps} provides a graphic representation of the sought mapping $\Psi$.
Consider the problem of finding a bijection of the 
semicircular domain $\Omega$ such that $\Phi(a) = b$:
first, we introduce the bijection $\Psi$ from the polytope (triangle) 
$\Omega_{\rm p}$ to the curved domain $\Omega$;
then, we seek a bijection $\Phi_{\rm p}$ of $\Omega_{\rm p}$ such that 
$\Phi_{\rm p}(a_{\rm p}) = b_{\rm p}$ with $a_{\rm p}=\Psi^{-1}(a)$ and $b_{\rm p}=\Psi^{-1}(b)$.
By construction, the resulting mapping
$\Phi = \Psi\circ \Phi_{\rm p} \circ \Psi^{-1}$ is a Lipschitz bijection of $\Omega$ that maps $a$ into $b$.
 
\begin{figure}[h!]
\centering
\begin{tikzpicture}[scale=0.75]
\linethickness{0.3 mm}

  \def\Radius{3}
  \path
    (-\Radius, 0) coordinate (A)
    -- coordinate (M)
    (\Radius, 0) coordinate (B)
    (M) +(60:\Radius) coordinate (C)
    +(120:\Radius) coordinate (D)
  ;
  \draw[ultra thick]
    (B) arc(0:180:\Radius) -- cycle
  ;
    
\draw[ultra thick]  (5,0)--(8,0)--(6.5,3)--(5,0);

\draw[fill] (2.6,1.5) circle (5pt);
\coordinate [label={below:  {\Large{$a$}}}] (E) at (2.5,1.4) ;
\draw[fill] (1.83,2.4) circle (5pt);
\coordinate [label={below:  {\Large{$b$}}}] (E) at (1.83,2.35) ;

\coordinate [label={center:  {\LARGE {$\Omega$}}}] (E) at (0.5, 1.5) ;

\coordinate [label={center:  {\LARGE {$\Omega_{\rm p}$}}}] (E) at (7, 1) ;

\draw[<-,ultra thick]  (3.05,1.5)--(4.95,1.5);
\coordinate [label={above:  {\LARGE {${\Psi}$}}}] (E) at (4, 1.8) ;

\draw[fill] (7.3,1.5) circle (5pt);
\draw[fill] (6.8,2.4) circle (5pt);
\coordinate [label={right:  {\Large{$a_{\rm p}$}}}] (E) at (7.3,1.5) ;
\coordinate [label={right:  {\Large{$b_{\rm p}$}}}] (E) at (6.8,2.4) ;

\end{tikzpicture}
 
\caption{compositional maps. The bijection $\Psi$ is designed to   recast the registration task from the curved domain $\Omega$ to the polytope $\Omega_{\rm p}$. 
}
\label{fig:geometric_interpretation_compositional_maps}
\end{figure}
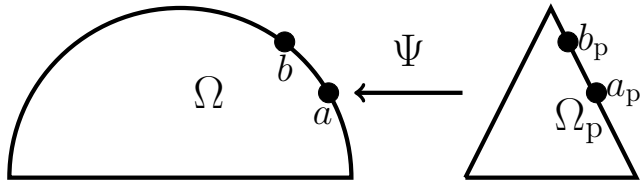 

We  exploit the analysis of registration in polytopes to study the properties of  maps of the form \eqref{eq:nonlinear_ansatz} and we discuss the approximation power.
We show (cf. Lemma \ref{th:density_composition})  that single-layer compositional maps cannot approximate arbitrary diffeomorphisms in curved domains.
 To improve the approximation power,
we  consider 
 the  multi-layer  generalization of \eqref{eq:nonlinear_ansatz} such that
\begin{equation}
\label{eq:nonlinear_ansatz_generalized}
\texttt{N}
\left(\mathbf{a}
=[\mathbf{a}_1,\ldots,\mathbf{a}_{\ell}]
\right)
 =
 \texttt{N}_1
\left( \mathbf{a}_1 \right)
\circ
\ldots
\circ
 \texttt{N}_{\ell} 
\left( \mathbf{a}_{\ell} \right);
\end{equation}
here,
$ \texttt{N}_{i} 
\left( \mathbf{a}_{i} \right)
=
 \Psi_i \circ  \texttt{N}_{{\rm p},i}(\mathbf{a}_i) \circ \Psi_i^{-1}
$ where
$ \Psi_i : \Omega_{{\rm p},i} \to \Omega$
is a Lipschitz bijection and
$\Omega_{{\rm p},i}$ is a suitable polytope, for $i=1,\ldots,\ell$.
In Lemma \ref{th:density_multilayer_composition}, we show the superior properties of the ansatz 
 \eqref{eq:nonlinear_ansatz_generalized}.

Our method is related to several previous works.
A first extension of the registration procedure in \cite{taddei2020registration} to arbitrary domains was proposed  in \cite{taddei2021registration}: the work of \cite{taddei2021registration}  relies on the introduction of a coarse-grained partition of the domain $\Omega$ and on Gordon-Hall maps to morph each element of the partition into the unit square.
The approach in \cite{taddei2021registration} requires that each element of the partition is mapped in itself (\emph{local bijectivity}) and relies on Gordon-Hall maps:
it is hence very sensitive to the choice of the coarse-grained partition,
it is not dense in any meaningful subset of diffeomorphisms 
 and it cannot be  extended to three-dimensional domains.
We also recall the work by Zahr and Persson 
\cite{zahr2020r} for high-order implicit shock tracking methods: the approach in \cite{zahr2020r} relies on a local parameterization of boundary degrees of freedom; for this reason, it cannot be readily combined with linear dimensionality reduction techniques to identify low-rank mapping spaces for  parametric systems.
We finally remark that the multi-layer ansatz 
\eqref{eq:nonlinear_ansatz_generalized} is 
closely related to registration  methods
appeared in the image processing literature
\cite{camion2001geodesic,cao2005large}.
 
We here rely on a standard $H^1$-conforming  FE discretization to represent the mapping; on the other hand, as in our previous works
(e.g., \cite{taddei2020registration,taddei2021registration}), we aim to  include an $H^2$ penalization term in the objective $\mathfrak{f}^{\rm obj}$.
Towards this end, we propose a discrete $H^2$ broken norm that is inspired by the work by Mozolevski  and coauthors \cite{mozolevski2007hp} on discontinuous Galerkin (DG) discretizations of the  biharmonic equation.
Following the seminal work by Argyris and coauthors 
\cite{argyris1968tuba},  several researchers have considered
 $H^2$-conforming FE spaces
 for the discretization of  fourth-order operators:
 the use of standard FE discretizations simplifies the implementation and also enables the application  of  polynomial bases of arbitrary  order.

The outline of the paper is as follows.
Section \ref{sec:affine_maps_polytopes} provides a complete analysis of  registration in polytopes;
section  \ref{sec:nonlinear_maps}  addresses the extension to curved domains;
section \ref{sec:methods} discusses the construction of   the polytope
$\Omega_{\rm p}$ and the mapping $\Psi$ and reviews the parametric registration method 
proposed in \cite{taddei2021registration} and employed in the numerical experiments;
section \ref{sec:numerics} contains   numerical investigations to illustrate the performance of the registration method  and its implications for model reduction.
We here couple our registration procedure with a non-intrusive (POD+regression) MOR procedure for state estimation; the integration of registration in the offline-online computational paradigm of projection-based MOR is the subject  of ongoing research (see \cite{barral2023registration}).
Section 
\ref{sec:conclusions} concludes the paper.

\section{Affine maps in polytopes}
\label{sec:affine_maps_polytopes} 
In order to find $\texttt{N}$ and  $\mathfrak{f}^{\rm pen}$
that satisfy \eqref{eq:desiderata_mapping_space}, we introduce the sets
\begin{equation}
\label{eq:admissible_maps_jacobian}
A:= \{  \mathbf{a}\in \mathbb{R}^M \, : \, \texttt{N}(\mathbf{a}) \in  \mathfrak{B} \},
\quad
A_{\rm jac} : = \{
\mathbf{a}\in \mathbb{R}^M : \inf_{x\in \Omega} J( \texttt{N}(\mathbf{a})    )>0
\}.
\end{equation}
Then, we require that the condition 
$ \inf_{x\in \Omega} J(\Phi)>0$ implies bijectivity in $\Omega$ for any mapping $\Phi$ spanned by $\texttt{N}$, that is 
$A_{\rm jac} \subset A$.
It is possible to construct maps $\Phi$ that are bijective in $\Omega$ for which
$\inf_{x\in \Omega} J(\Phi) =  0$  (cf. 
\cite[Theorem 1.1]{ruzhansky2015global}); 
however, bijections $\Phi$ which satisfy 
$\inf_{x\in \Omega} J(\Phi) =  0$ are of little practical interest for scientific computing applications and in particular for (projection-based) MOR.

Note that the pointwise condition 
$J( \texttt{N}(\mathbf{a}) )(x)>0$ for all $x\in \Omega$ cannot be directly translated into a penalty term for  \eqref{eq:tractable_optimization_based_registration}: 
we address the construction of the  penalty term in section  \ref{sec:penalty_function}.
In section \ref{sec:basic_theorems}, we present two technical results 
for affine maps in polytopes;
{in sections  \ref{sec:ansatz_explained_polytopes}
and  \ref{sec:approx_polytopes}, we construct operators 
$\texttt{N}$ that satisfy $A_{\rm jac} \subset A$ and are dense in a meaningful subspace of diffeomorphisms};
finally, in section \ref{sec:lipschitz_maps} we comment on the approximation of Lipschitz maps.

\subsection{Definition of the penalty function}
\label{sec:penalty_function}
We denote by $\|  \cdot \|_{\rm F}$ Frobenius norm, and by 
$H(w)$ the Hessian of the field $w$. Given the polytope $\Omega_{\rm p}$, the constants
$\epsilon>0$ and  $C_{\rm exp}$ such that $C_{\rm exp} \ll \epsilon$, we introduce the function $\mathfrak{f}_{\rm pen}^{(1)}(  \Phi ) = 
\mathfrak{f}_{\rm jac}(  \Phi ) \: + \; \| \nabla J(\Phi)  \|_{L^{\infty}(\Omega_{\rm p})}$ with 
\begin{equation} 
\label{eq:f_jac_tmp}
\mathfrak{f}_{\rm jac}(  \Phi )
=
\frac{1}{|\Omega_{\rm p}|}
\int_{\Omega_{\rm p}}
{\rm exp} \left(
\frac{\epsilon - J(\Phi)}{C_{\rm exp}} 
\right)
\, dx.
\end{equation}
We can show  that  there exist constants $C,C_{\rm exp}$ such that the condition $\mathfrak{f}_{\rm pen}^{(1)}(  \Phi )
\leq C$ implies that $J(\Phi) \geq \epsilon/2$ for all $x\in \Omega_{\rm p}$ (cf. \cite[section 2.2]{taddei2020registration}), that is $\{ \mathbf{a}:
\mathfrak{f}_{\rm pen}^{(1)}(  \texttt{N}(\mathbf{a})  ) \leq C \} \subset A_{\rm jac}$.
Recalling Jacobi's formula, we find that 
$\| \nabla J(\Phi)   \|_2 \leq C \|  \nabla \Phi    \|_{\rm F}
\| H(\Phi)   \|_{\rm F}$, for 
some constant $C$ that is independent of $\nabla \Phi$;
 therefore, we can replace  $\| \nabla J(\Phi)  \|_{L^{\infty}(\Omega_{\rm p})}$
  with 
 $\| H(\Phi)  \|_{L^{\infty}(\Omega_{\rm p})}$, to obtain
 \begin{equation}
\label{eq:f_pen_theory}
\mathfrak{f}_{\rm pen}^{\rm th}(  \Phi )
= \mathfrak{f}_{\rm jac}(  \Phi ) + \| H(\Phi)  \|_{L^{\infty}(\Omega_{\rm p})}^2.
\end{equation}
The penalty \eqref{eq:f_pen_theory} is defined for elements of the Sobolev space
$W^{2,\infty}(\Omega) = \{ v\in L^{\infty}(\Omega): \nabla v, H(v) \in  L^{\infty}(\Omega) \}$.
We anticipate that the penalty \eqref{eq:f_pen_theory}   is still not fully  amenable for 
computations: we address the problem of constructing an actionable penalty for \eqref{eq:tractable_optimization_based_registration} in section \ref{sec:methods}.

\subsection{Mathematical background}
\label{sec:basic_theorems}
We say that $\Omega_{\rm p} \subset \mathbb{R}^2$ is a polytope if the boundary of   $\Omega_{\rm p}$,
   $\partial \Omega_{\rm p}$,  consists of a finite number of flat sides (\emph{faces}).  
 The boundary of two-dimensional polytopes  is described by a finite number of straight segments (\emph{edges}) connected to form (possibly several disjoint)  closed polygonal chains; the points of  $\partial \Omega_{\rm p}$ where two consecutive 
non-parallel
 edges meet are dubbed \emph{vertices}; the union of all vertices is here denoted by $V = \{ x_i^{\rm v} \}_{i=1}^{N_{\rm v}}$.  We   consider polytopes that satisfy the  condition below:
note that Definition \ref{def:regular_polytopes} implies that each   vertex $ x^{\rm v}$ 
of a regular polytope 
 is the intersection of two edges 
  (cf. Figure \ref{fig:geometric_definition}).
 
\begin{Definition}
\label{def:regular_polytopes}
The bounded polytope $\Omega_{\rm p}$  is said to be regular if there exist
$N+1$ bounded polytopes
$\Omega_{{\rm int},1}, \ldots,$ 
$\Omega_{{\rm int},N}$, 
$\Omega_{\rm ext}$
such that
\textbf{(i)}
 $\Omega_{\rm p} = \Omega_{\rm ext} \setminus \bigcup_{i=1}^N \Omega_{{\rm int},i}$, 
 \textbf{(ii)}
$\Omega_{{\rm int},1}, \ldots,\Omega_{{\rm int},N}$ are pairwise disjoint polytopes that are compactly embedded in 
$\Omega_{\rm ext}$ and 
\textbf{(iii)}
$\Omega_{\rm ext},\Omega_{{\rm int},1},\ldots, \Omega_{{\rm int},N}$ are isomorphic to the unit ball.
 \end{Definition}
 
\begin{figure}[h!]
\centering
\subfloat[]{
\centering
\begin{tikzpicture}[scale=0.8]
\linethickness{0.3 mm}
\linethickness{0.3 mm}

\filldraw[pattern=north west lines, pattern color=blue] 
(0,0)--(3,1)--(4,2)--(2,1.5)--(3,2.5)--(0,2)--(-1,0)--(0,0);

\filldraw[fill=white]  (-0.25,0.5)--(2,0.9)--(0,1)--(-0.25,0.5);



\coordinate [label={right:  {\LARGE {$\Omega_{\rm p}$}}}] (E) at (4, 2) ;
\end{tikzpicture}
}
~~~
\subfloat[]{
\begin{tikzpicture}[scale=0.8]
\linethickness{0.3 mm}
\linethickness{0.3 mm}

\filldraw[pattern=north west lines, pattern color=blue] 
(0,0)--(3,1)--(4,2)--(2,1.5)--(3,2.5)--(0,2)--(-1,0)--(0,0);

\filldraw[fill=white]  (-0.25,0.5)--(2,1.5)--(0,1)--(-0.25,0.5);

\filldraw (2,1.5) circle (3pt);


\coordinate [label={right:  {\LARGE {$\Omega_{\rm p}$}}}] (E) at (4, 2) ;
\end{tikzpicture}

}
\caption{interpretation of Definition
\ref{def:regular_polytopes}.
(a) regular polytope.
(b) irregular polytope.
}
\label{fig:geometric_definition}
\end{figure}
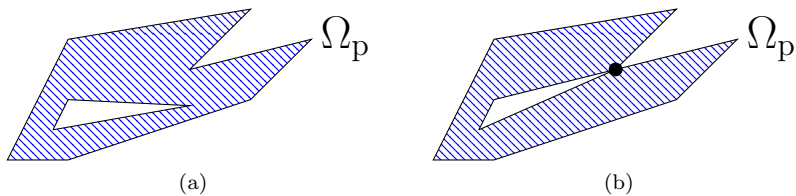

We denote by $\mathbf{n}(x)$ the outward normal to $\Omega_{\rm p}$ at $x\in \partial \Omega_{\rm p}$; 
we also denote by $\overline{\Omega}_{\rm p}$ the closure of $\Omega_{\rm p}$ in $\mathbb{R}^2$.
Propositions \ref{th:bijectivity} and \ref{th:vertices_plus_density} 
contain two important results for  mappings  in polytopes.
The proof of Proposition  \ref{th:bijectivity} is technical and is postponed to Appendix \ref{sec:proofs}.

\begin{proposition}
\label{th:bijectivity}
Let   $\Omega_{\rm p}$ be a regular bounded polytope.
Define the space
${\mathfrak{U}}_0={\mathfrak{U}}_0(\Omega_{\rm p}) = \{ \varphi \in C^1(\overline{\Omega}_{\rm p}; \mathbb{R}^2) \,: \, \varphi \cdot \mathbf{n} |_{\partial \Omega_{\rm p}} = 0 \}$ and consider the vector-valued function
$\Phi = \texttt{id} + \varphi$ with $\varphi\in {\mathfrak{U}}_0$. Then, $\Phi$ is a bijection in $\Omega_{\rm p}$ if $\min_{x\in \overline{\Omega}_{\rm p}} J(\Phi) > 0$. 
\end{proposition}
 
\begin{proposition}
\label{th:vertices_plus_density}
Let $\Phi$ be a  diffeomorphism   in $\Omega_{\rm p}$. Then, $\Phi(V) = V$. Furthermore, if $\Phi(x^{\rm v}) = x^{\rm v}$ for all $x^{\rm v}\in V$, then $\Phi = \texttt{id} + \varphi$ with $\varphi\in \mathfrak{U}_0$. 
\end{proposition}

\begin{proof}
Let $x^{\rm v}\in V$ and define $y = \Phi(x^{\rm v})$; recalling Nanson's  formula \cite[page 7]{marsden1994mathematical}, we find that the normal to $\Phi(\Omega_{\rm p})$ satisfies 
$\mathbf{n}_{\Phi}(y)
\propto
J( \Phi(x)   ) (\nabla \Phi(x))^{-T} \mathbf{n}(x);
$
therefore, if $x\in \partial \Omega_{\rm p} \mapsto \mathbf{n}(x)$ is discontinuous at  $x^{\rm v}$, we must have that $x \mapsto  \mathbf{n}_{\Phi}(x) $
is discontinuous at  $y$. Recalling the definition of vertices, we conclude that $y\in V$ for any $x^{\rm v}\in V$, that is $\Phi(V) \subset V$. Since $\Phi$ is a bijection in $\Omega_{\rm p}$, the image points $\{ \Phi(x^{\rm v}): x^{\rm v}\in V  \}$ should all be distinct and thus
${\rm card} (\Phi(V)) = {\rm card} (V)$, which implies
 $\Phi(V)=V$.
 
 Let $\Phi$ satisfy the condition $x^{\rm v} = \Phi(x^{\rm v})$
for all $x^{\rm v}\in V$ and define the displacement $\varphi = \Phi-\texttt{id}$. Consider the edge $F \subset \partial \Omega_{\rm p}$ and consider the parameterization $\gamma_{\rm f}:[0,1]\to F$ such that
$\gamma_{\rm f}(s) = x_1^{\rm v} + s \|x_2^{\rm v} -x_1^{\rm v}   \|_2 \mathbf{t}_{\rm f}$. Since $\partial \Omega_{\rm p}$ is closed and $\Phi(\partial \Omega_{\rm p}) = \partial \Omega_{\rm p}$, we either have
$\Phi(F) = \partial \Omega_{\rm p}\setminus F$ or 
$\Phi(F) =   F$. Recalling the expression for the normal 
$\mathbf{n}_{\Phi}$, we find that $\mathbf{n}_{\Phi}$ is continuous in $\Phi(F)$: since $ \partial \Omega_{\rm p}\setminus F$ contains at least one vertex in addition to $x_1^{\rm v},x_2^{\rm v}$ (we here exploit the fact that two-dimensional polytopes have at least three vertices), we must have 
$\Phi(F) =   F$.

The condition $\Phi(F) =   F$ implies that $\Phi(\gamma_{\rm f}(s)) = x_1^{\rm v} + \alpha(s) \mathbf{t}_{\rm f}$ for all $s\in [0,1]$ and for some injective function $\alpha:[0,1] \to [0, \|x_2^{\rm v} -x_1^{\rm v}   \|_2]$. Given $x\in F$, 
since $\mathbf{n}(x) = \mathbf{n}_{\rm f}$ and 
$\mathbf{n}_{\rm f}\perp \mathbf{t}_{\rm f}$, we find
$$
\mathbf{n}(x) \cdot \varphi(x)
=
\mathbf{n}_{\rm f} \cdot
\left(
\Phi(x) - x
\right)
=
\mathbf{n}_{\rm f} \cdot
\left(
x_1^{\rm v} + \alpha(s) \mathbf{t}_{\rm f}
-
x_1^{\rm v} -  s \|x_2^{\rm v} -x_1^{\rm v}   \|_2  \mathbf{t}_{\rm f}
\right)
= 0,
$$
which is the desired result.
\end{proof}

\subsection{Construction of finite-dimensional operators for  registration}
\label{sec:ansatz_explained_polytopes}
We introduce the finite-dimensional space 
$\mathcal{U}$ spanned by
$\{ \varphi_i  \}_{i=1}^{M} \subset {\mathfrak{U}}_0$, and the affine space
$\texttt{id} +\mathcal{U}$.
We introduce the operator
$\texttt{N}_{\rm p}: \mathbb{R}^M \to \texttt{id} +\mathcal{U}$ such that
$\texttt{N}_{\rm p}(\mathbf{a}) = \texttt{id} + \sum_{i=1}^M (\mathbf{a})_i \varphi_i$. 
Exploiting Proposition \ref{th:bijectivity}  and the discussion in section  \ref{sec:penalty_function}, we   show 
in Corollary \ref{th:easy_bijectivity_consequence}
that the operator $\texttt{N}_{\rm p}$ and the penalty \eqref{eq:f_pen_theory}
 satisfy \eqref{eq:desiderata_mapping_space}:
the stronger regularity assumption is required by  the choice of the penalty.
The proof is straightforward and is here omitted. 

\begin{corollary}
\label{th:easy_bijectivity_consequence}
Let 
$\mathcal{U} = {\rm span}
\{  \varphi_i \}_{i=1}^M$ be an $M$-dimensional subspace of  ${\mathfrak{U}}_0 \cap W^{2,\infty}(\Omega_{\rm p})$. Then, the affine operator $\texttt{N}_{\rm p}$ and the penalty 
$\mathbf{a} \mapsto \mathfrak{f}_{\rm pen}^{\rm th}(\texttt{N}_{\rm p}(\mathbf{a}))$ 
(cf.  \eqref{eq:f_pen_theory})  satisfy \eqref{eq:desiderata_mapping_space}.
\end{corollary}

\subsection{Approximation properties of affine mappings}
\label{sec:approx_polytopes}
Proposition \ref{th:vertices_plus_density} shows that bijections are of the form
$\Phi  = \texttt{id} + \varphi$ with $\varphi\in {\mathfrak{U}}_0$ if  $\Phi(x^{\rm v}) =x^{\rm v}$ for all $x^{\rm v}\in  V$  (i.e., $\Phi|_V = \texttt{id}$).  Since $\varphi|_V =0$ for all $\varphi\in {\mathfrak{U}}_0$, diffeomorphisms  $\Phi$ that do not satisfy   $\Phi|_V = \texttt{id}$ do not belong to  
$\texttt{id} + {\mathfrak{U}}_0$. It is easy to construct bijections that do not fulfill the requirement
$\Phi|_V = \texttt{id}$: to provide a concrete example, consider the map $\Phi(x) = - x$ for the polytope $\Omega_{\rm p} = (-1,1)^2$. 
However, in the setting of MOR and also geometry reduction, we are interested in parametric maps that are smooth deformations of the identity map:
since registration is applied with respect to an element of the solution manifold
$\mathcal{M}=\{u_{\mu} : \mu\in \mathcal{P} \}$,
we can indeed assume that there exists $\mu\in \mathcal{P}$ such that $\Phi_{\mu}=\texttt{id}$. 
We hence have the following result, which follows from 
Proposition \ref{th:vertices_plus_density}.

\begin{corollary}
\label{th:easy_density_consequence}
Let $\Phi: \Omega_{\rm p} \times \mathcal{P} \to \Omega_{\rm p}$ be a continuous function of the parameter $\mu\in \mathcal{P}$  and let 
$\Phi_{\mu}$ be a 
   diffeomorphism for all 
$\mu\in \mathcal{P}$.
Assume that   $\Phi_{\mu'}=\texttt{id}$ for some $\mu'\in \mathcal{P}$ and that $\mathcal{P}$ is simply connected. 
Then, 
$\Phi_{\mu}  =\texttt{id} +\varphi_{\mu}$ with 
$\varphi_{\mu}\in {\mathfrak{U}}_0$ for all
$\mu \in \mathcal{P}$.
\end{corollary}

\begin{proof}
Exploiting the first statement of Proposition \ref{th:vertices_plus_density}, we find that $\Phi_{\mu}(V)=V$ for all $\mu\in \mathcal{P}$.
Since the set of vertices $V$ is a discrete closed set, 
  if the function $\Phi: \Omega_{\rm p} \times \mathcal{P} \to \Omega_{\rm p}$ is continuous with respect to the parameter $\mu\in \mathcal{P}$ and
is bijective in $ \Omega_{\rm p}$   for all
$\mu\in \mathcal{P}$, then the hypothesis 
$\Phi_{\mu'}|_V=\texttt{id}$  
 \emph{for some} $\mu'\in \mathcal{P}$ implies 
 $\Phi_{\mu}|_V=\texttt{id}$  
 \emph{for all} $\mu\in \mathcal{P}$. 
 Therefore, exploiting the second statement of Proposition \ref{th:vertices_plus_density}, we conclude that
 $\Phi_{\mu}  =\texttt{id} +\varphi_{\mu}$ with 
$\varphi_{\mu}\in {\mathfrak{U}}_0$ for all
$\mu \in \mathcal{P}$.
\end{proof}

Given the parametric registration problem
 \eqref{eq:tractable_optimization_based_registration} with $\texttt{N} =\texttt{N}_{\rm p}$, we
 denote by 
$\mathbf{a}_{\mu}^{\rm opt}$ a solution to 
 \eqref{eq:tractable_optimization_based_registration}  for any $\mu\in \mathcal{P}$ and we define
 the corresponding manifold   $\mathcal{M}_{\Phi}: = \{ \mathbf{a}_{\mu}^{\rm opt} : \mu\in \mathcal{P} \} \subset \mathbb{R}^M$.  
 If  $\mathcal{M}_{\Phi}$ is reducible --- 
 that is, there exists an orthogonal matrix 
$\mathbf{W}\in \mathbb{R}^{M\times m}$
whose columns approximate the elements of 
$\mathcal{M}_{\Phi}$
for   $m \ll M$ 
 ---
 we can restrict the search space in  \eqref{eq:tractable_optimization_based_registration} to the image of $\mathbf{W}$, 
 ${\rm col}(\mathbf{W})$. 
As discussed in the introduction, since the pair
$(\texttt{N} ( \cdot) = \texttt{N}_{\rm p} (\mathbf{W} \cdot),
\mathfrak{f}_{\rm pen}^{\rm th}(\texttt{N}_{\rm p}( 
\mathbf{W} \cdot  ))$ satisfies \eqref{eq:desiderata_mapping_space} for any full-rank matrix  
$\mathbf{W}$, the choice of the latter  can be made solely based on approximation considerations.

\subsection{Approximation of Lipschitz maps}
\label{sec:lipschitz_maps}
The proof of Proposition \ref{th:bijectivity} 
exploits   the 
Hadamard's global inverse function theorem (cf. \cite[Theorem  6.2.8] {krantz2002implicit});
it hence relies on the assumption that $\Phi$ is of class $C^1$.
Our numerical investigations suggest that a similar result --- possibly with further conditions --- might hold for piecewise-smooth maps such as FE fields.   

On the other hand, 
Proposition \ref{th:vertices_plus_density} does not hold for general Lipschitz maps, as shown in the next example.
First, we introduce  the square domain $\Omega=(0,1)^2$,
and  we define 
the  map $\Phi_1: \Omega=(0,1)^2 \to \mathbb{R}^2$ such that
\begin{equation}
\label{eq:silly_example}
\Phi_1(x) = \left[
\begin{array}{cc}
1/2 & 0 \\ -1/2 & 1 \\
\end{array}
\right] x \mathbbm{1}_{x_1<x_2}(x)
+
\left[
\begin{array}{cc}
1 & -1/2 \\  0 & 1/2 \\
\end{array}
\right] x \mathbbm{1}_{x_1\geq x_2}(x).
\end{equation}
It is easy to verify that
$\Phi_1$ is a Lipschitz map 
from $\Omega$ in the unit triangle $D = \{x\in \Omega : x_1+x_2<1\}$
with Lipschitz inverse
(cf. 
Figure \ref{fig:geometric_interpretation_Phi}(a)).
Then, we introduce a smooth bijection $\Phi_2: D\to D$ that maps $x^{\star} = [1/2,1/2]$ into 
$y^{\star} = [1/4,3/4]$ 
(cf.  Figure \ref{fig:geometric_interpretation_Phi}(b)):
$\Phi_2$ can be constructed using an expansion of  quadratic polynomials; we omit the explicit expression.
Finally, we define the map
$\Phi = \Phi_1^{-1} \circ \Phi_2 \circ \Phi_1$:
clearly, $\Phi$ is a bijection in $\Omega$; furthermore, 
$\Phi([1,1]) = 
\Phi_1^{-1}\left(\Phi_2 ([1/2,1/2]) \right)
= 
\Phi_1^{-1}\left([1/4,3/4] \right)
=[1,1/2].$ We hence found a bijection  $\Phi$ in $\Omega$ such that  $\Phi(V) \neq V$.

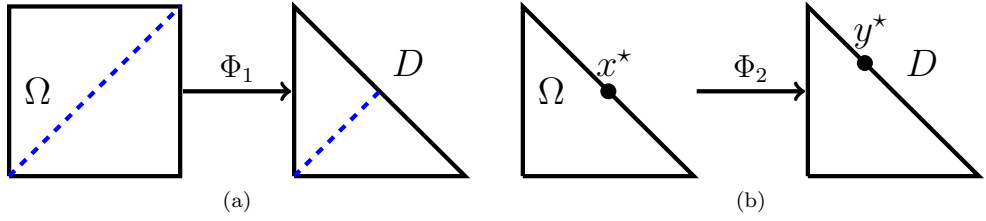
\begin{figure}[h!]
\centering
\subfloat[]{
\begin{tikzpicture}[scale=0.75]
\linethickness{0.3 mm}
\draw[ultra thick]  (0,0)--(3,0)--(3,3)--(0,3)--(0,0);
\draw[ultra thick]  (5,0)--(8,0)--(5,3)--(5,0);

\draw[ultra thick,dashed,blue] (0,0)--(3,3);

\draw[ultra thick,dashed,blue] (5,0)--(6.5,1.5);

\coordinate [label={center:  {\Large {$\Omega$}}}] (E) at (0.5, 1.5) ;

\coordinate [label={center:  {\Large {$D$}}}] (E) at (7, 2) ;

\draw[->,ultra thick]  (3.05,1.5)--(4.95,1.5);

\coordinate [label={above:  {\large {${\Phi_1}$}}}] (E) at (4, 1.5) ;
 
\end{tikzpicture}
}
~~~
\subfloat[]{
\begin{tikzpicture}[scale=0.75]
\linethickness{0.3 mm}
\draw[ultra thick]  (0,0)--(3,0)--(0,3)--(0,0);
\draw[ultra thick]  (5,0)--(8,0)--(5,3)--(5,0);

\coordinate [label={above:  {\Large {$x^{\star}$}}}] (E)   at (1.6, 1.6) ;
\coordinate (E) at (1.5, 1.5) ;
\fill (E) circle[radius=4pt];

\coordinate [label={above:  {\Large {$y^{\star}$}}}] (E)   at (6.1, 2.1) ;
 \coordinate (E) at (6,2) ;
\fill (E) circle[radius=4pt];

\coordinate [label={center:  {\Large {$\Omega$}}}] (E) at (0.5, 1.5) ;

\coordinate [label={center:  {\Large {$D$}}}] (E) at (7, 2) ;

\draw[->,ultra thick]  (3.05,1.5)--(4.95,1.5);

\coordinate [label={above:  {\large {${\Phi_2}$}}}] (E) at (4, 1.5) ;
 
\end{tikzpicture}
}
\caption{example of Lipschitz map
with $\Phi(V)\neq V$.
(a) Lipschitz map 
 from the unit square to the unit triangle (cf \eqref{eq:silly_example}).
 (b) smooth map in the unit triangle $D$.
}
\label{fig:geometric_interpretation_Phi}
\end{figure}

\section{Nonlinear  maps for general domains}
\label{sec:nonlinear_maps} 

We address the extension to non-polytopal domains. First, we   justify the need for nonlinear ansätze;
then, we investigate the approximation properties of the nonlinear ansätze \eqref{eq:nonlinear_ansatz} and \eqref{eq:nonlinear_ansatz_generalized}.

\subsection{Inadequacy of affine maps for domains with curved boundaries}
For many problems of interest, including the one considered in the numerical results of section \ref{sec:numerics}, it is of paramount importance to deform points that lie on curved edges or faces; we hence seek an operator $\texttt{N}$ and a penalty $\mathfrak{f}_{\rm pen}$ that satisfy 
\eqref{eq:desiderata_mapping_space} and enable non-trivial deformations of points on curved boundaries.
Next result, 
which generalizes \cite[Lemma 2.1]{taddei2021registration}, 
 shows that non-trivial affine mappings for curved domains lead to admissible sets $A$ with empty interior: therefore, 
affine maps 
in combination with the penalty \eqref{eq:f_pen_theory} do not satisfy   \eqref{eq:desiderata_mapping_space}.
 This observation shows the inadequacy of affine maps for domains with curved boundaries and ultimately justifies the need for a nonlinear ansatz.

\begin{lemma}
\label{th:inadequacy_linear_maps}
Let $\Omega\subset \mathbb{R}^2$ be a domain  with curved boundary $\Gamma \subset \partial \Omega$; let $\texttt{N}:\mathbb{R}^M \to {\rm Lip}(\overline{\Omega}; \mathbb{R}^2)$ be an affine operator of the form
$\texttt{N}(\mathbf{a}) = \texttt{id} + \sum_{i=1}^M(\mathbf{a})_i \varphi_i$ and let 
$A\subset \mathbb{R}^M$ be the set of admissible (i.e., bijective in $\Omega$) maps. Suppose that there exists $i\in \{1,\ldots,M\}$ and $x\in \Gamma$ such that $\varphi_i(x)\neq 0$.
Then, $0\in \mathbb{R}^M$ does not belong to the interior of $A$. Furthermore, if $\Gamma  =  \partial \Omega$, then $A$ has empty interior.
\end{lemma}
\begin{proof}
By contradiction, there exists $r>0$ such that $\mathcal{B}_r(0) \subset A$; therefore, there exists $\delta>0$ such that $\Phi_t = \texttt{id} + t \varphi_i$ is a bijection in $\Omega$ for all $t\in (-\delta,\delta)$. 
Since  $\Phi_t$ is bijective in $\Omega$, we must have $\Phi_t(\partial \Omega) = \partial \Omega$, which implies that the linear segment $S_{\delta} = \{ x + t  \varphi_i(x) : t\in(-\delta,\delta) \}$ is  contained in $\partial \Omega$. 
We conclude that $\partial \Omega$ is flat in the neighborhood of $x\in \Gamma$. Contradiction.

We remark that the argument of the proof  solely relies on  the fact that
 $\Phi_{t=0}(x)\in \Gamma$: if 
$\partial \Omega$ is entirely curved (that is, if $\Gamma  =  \partial \Omega$) we can replicate the same argument   for  any $\mathbf{a}\in A$  to prove that 
$A$ has empty interior.
\end{proof}

\subsection{Approximation properties of compositional maps}
\label{sec:density_composition}
Given the domain $\Omega$, we introduce the polytope $\Omega_{\rm p}$ and the isomorphism $\Psi$ from $\Omega_{\rm p}$ to $\Omega$;
we denote by 
$V=\{x_i^{\rm v} \}_{i=1}^{N_{\rm v}}$ the vertices of $\Omega_{\rm p}$ and we denote by
$\{F_j^{\rm p} \}_{j=1}^{N_{\rm f,p}}$ the facets of $\Omega_{\rm p}$; we also define the mapped facets
 $\{F_j = \Psi(F_j^{\rm p}) \}_{j=1}^{N_{\rm f,p}}$.
In the remainder, we assume that the pair $(\Omega_{\rm p}, \Psi)$ satisfies the following assumption;
the set  $V_{\rm ang}$  
of angular points of $\Omega$ 
is the subset of $\partial \Omega$ where 
the normal $\mathbf{n}$ is discontinuous.
 
\begin{hypothesis}
\label{hyp:regular_polytopes}
The pair $(\Omega_{\rm p}, \Psi)$ satisfies the following conditions:
(i)
the vertices $V$ of $\Omega_{\rm p}$ belong to $\partial \Omega$;
(ii)
the set  $V_{\rm ang}$ of angular points of $\Omega$  is contained in $V$;
(iii)
 the linear facets of $\partial \Omega$ belong to $\partial  \Omega_{\rm p}$;
(iv)
the field $\Psi$ is a Lipschitz bijection from $\Omega_{\rm p}$ to $\Omega$ that satisfies
$\Psi|_V= \texttt{id}$ and  $J(\Psi)>0$ a.e..
 \end{hypothesis}

By construction,  any bijection   $\Phi$ of the form \eqref{eq:nonlinear_ansatz}
--- that is, $\Phi = \Psi \circ \Phi_{\rm p} \circ \Psi^{-1}$ with  
 $\Phi_{\rm p} \in \texttt{id} + {\mathfrak{U}}_0(\Omega_{\rm p})$---
  satisfies 
 $\Phi|_V= \texttt{id}$ and
 $\Phi(F_j) = F_j $ for $j=1,\ldots,N_{\rm f}$. Conversely, we have the following result, which signifies that nonlinear maps of the form \eqref{eq:nonlinear_ansatz} are dense in the space of diffeomorphisms that satisfy 
 $\Phi|_V= \texttt{id}$.

\begin{lemma}
\label{th:density_composition}
Let 
 $(\Omega_{\rm p},\Psi)$ satisfy 
 Hypothesis \ref{hyp:regular_polytopes}.
Let $\Phi$ be a diffeomorphism  in $\Omega$   such that
$\Phi|_V= \texttt{id}$. 
Then, $\widetilde{\Phi} = \Psi^{-1} \circ \Phi \circ \Psi = \texttt{id} + \varphi$ for some   $\varphi \in {\rm Lip}(\Omega_{\rm p}; \mathbb{R}^2)$ such that $\varphi \cdot \mathbf{n}|_{\partial \Omega_{\rm p}} = 0$.
\end{lemma}
\begin{proof}
Since $\Phi$ is a diffeomorphism, we must have $\Phi(\partial \Omega) = \partial \Omega$; furthermore, 
the condition $\Phi|_V= \texttt{id}$ implies that $\Phi$ preserves the orientation of the boundary. In conclusion, we find
${\Phi}(F_j) = F_j$ for $j=1,\ldots,N_{\rm f}$.
Therefore, we have that  
$\widetilde{\Phi} = \Psi^{-1} \circ \Phi \circ \Psi$ is a bijection in $\Omega_{\rm p}$ that satisfies 
$\widetilde{\Phi}(F_j^{\rm p}) = F_j^{\rm p}$ for $j=1,\ldots,N_{\rm f}$.
Then, exploiting the same argument of the proof of Proposition \ref{th:vertices_plus_density}, we find that 
$\widetilde{\Phi} = \texttt{id} + \varphi$ with 
$\varphi \cdot \mathbf{n}|_{\partial \Omega_{\rm p}} = 0$. We omit the details.
\end{proof}

\subsection{Multi-layer compositional maps}
\label{sec:PTC_link}
Lemma \ref{th:density_composition} shows that single-layer compositional maps  cannot approximate arbitrary diffeomorphisms in curved domains.
To address this issue, we might consider more general multi-layer maps of the form \eqref{eq:nonlinear_ansatz_generalized}. We here study the approximation properties of 
\eqref{eq:nonlinear_ansatz_generalized} for  $\ell=2$ layers
$$
\texttt{N}(\mathbf{a}_1,\mathbf{a}_2) = 
\texttt{N}_1(\mathbf{a}_1) \circ \texttt{N}_2(\mathbf{a}_2),
\quad
\texttt{N}_i(\mathbf{a}_i) 
=
\Psi_i \circ \texttt{N}_{{\rm p},i}(\mathbf{a}_i) \circ \Psi_i^{-1},
\; \; i=1,2.
$$
The extension of the analysis to arbitrary number of layers is the subject of ongoing investigations.

We denote by $\Omega_{{\rm p}, 1}$ and $\Omega_{{\rm p}, 2}$ the polytopes associated to the bijections $\Psi_1,\Psi_2$;
we assume that 
$(\Omega_{{\rm p}, 1}, \Psi_1)$ and 
$(\Omega_{{\rm p}, 2}, \Psi_2)$ satisfy
Hypothesis \ref{hyp:regular_polytopes}.
We  also denote by $V_i^{\rm fict} = \{  x_{j,i}^{\rm v}  \}_{j=1}^{N_{{\rm v},i}}$ the \emph{fictitious}  
vertices of $\Omega_{{\rm p}, i}$, 
that is, the vertices of $\Omega_{{\rm p}, i}$ that do not correspond to
angular points of $\partial \Omega$.
Finally, given $x,y\in\partial \Omega$, we define the set of curves
\begin{subequations}
\label{eq:geodesics_distance}
\begin{equation}
\label{eq:geodesics_distance_a}
A_{\partial \Omega}(x,y) = \left\{
\gamma\in C^1([0,1]; \partial \Omega) :
\gamma(0)=x, 
\; 
\gamma(1)=y,
\; 
\|  \dot{\gamma} \|_2 \equiv  {\rm const} \;
\right\}
\end{equation}
and the geodesic distance
\begin{equation}
\label{eq:geodesics_distance_b}
{\rm dist}_{\partial \Omega}(x,y)
=
\left\{
\begin{array}{ll}
\displaystyle{
\inf \left\{
\int_0^1 \|  \dot{\gamma} \|_2 \, dt :
\; \;
\gamma \in A_{\partial \Omega}(x,y)
\right\}
}
&
{\rm if} \; A_{\partial \Omega}(x,y) \neq \emptyset
\\[3mm]
+\infty
&
{\rm if} \; A_{\partial \Omega}(x,y)  =  \emptyset
\\
\end{array}
\right.
\end{equation}
\end{subequations}
Note that $C^1$ curves cannot pass through a non-smooth point of the boundary (angular point): therefore, if $x,y\in \partial \Omega$ are separated by angular points (e.g., points on two separate faces of a polytope), we have  
$ A_{\partial \Omega}(x,y)  =  \emptyset$ and ${\rm dist}_{\partial \Omega}(x,y)=+\infty$.

Next Lemma shows that two-layer maps can approximate  
 arbitrary diffeomorphisms under the assumption of small deformations.
Note that boundary deformations should be smaller than 
  (i) the minimum distance between the fictitious vertices of 
$\Omega_{{\rm p}, 1}$ and
$\Omega_{{\rm p}, 2}$
---
 $\min_{x\in V_1^{\rm fict} , y\in V_2^{\rm fict}} 
{\rm dist}_{\partial \Omega}(x ,y)$;
(ii) the minimum distance between two vertices of 
 $\Omega_{{\rm p}, 1}$ and  $\Omega_{{\rm p}, 2}$
 ---
$\min_{x,y\in V_i, x\neq y  } 
{\rm dist}_{\partial \Omega}(x ,y)$ with 
$i=1,2$.

\begin{lemma}
\label{th:density_multilayer_composition}
Let 
 $(\Omega_{{\rm p},1},\Psi_1)$
 and
  $(\Omega_{{\rm p},2},\Psi_2)$ satisfy 
  Hypothesis \ref{hyp:regular_polytopes}.
 Let $\Phi$ be a diffeomorphism in $\Omega$   with positive Jacobian determinant  such that
$\max_{x\in \partial \Omega} {\rm dist}_{\partial \Omega} $
$(x,\Phi(x))$ $< C$ where $C=C(\Omega_{{\rm p}, 1},\Omega_{{\rm p}, 2}   )$
is given by  
\begin{equation}
\label{eq:constantC}
\begin{array}{rl}
C:= \min  \Big\{
&
\displaystyle{
\min_{x\in V_1^{\rm fict} , y\in V_2^{\rm fict}} 
{\rm dist}_{\partial \Omega}(x ,y),
\;\;
\min_{x,y\in V_1, x\neq y  } 
{\rm dist}_{\partial \Omega}(x ,y),
} \\[3mm]
&
\displaystyle{
\min_{x,y\in V_2, x\neq y  } 
{\rm dist}_{\partial \Omega}(x ,y)\Big\}.}
\\
\end{array}
\end{equation}
where $V_1^{\rm fict}, V_2^{\rm fict}$ are the vertices of $\Omega_{{\rm p},1},\Omega_{{\rm p},2}$ that do not correspond to angular points of $\partial \Omega$.
Then, there exist $\Phi_1 = \texttt{id}+\varphi_1$ and 
$\Phi_2 = \texttt{id}+\varphi_2$  such that
$\Phi = 
\Psi_1 \circ \Phi_1 \circ \Psi_1^{-1}
\circ
\Psi_2 \circ \Phi_2 \circ \Psi_2^{-1}
$ and 
$\varphi_i\cdot \mathbf{n}|_{\partial \Omega_{{\rm p},i}}=0$ for $i=1,2$.
\end{lemma}

\begin{proof}
We assume that $\Omega$
has only one curved boundary $\Gamma\subset \partial \Omega$ with parameterization $\boldsymbol{\gamma}:[0,1]\to \Gamma$. We define 
$N_{\rm v}= N_{{\rm v},1}+N_{{\rm v},2}$,
$\{y_{\ell} \}_{\ell=1}^{N_{\rm v}} \subset \Gamma$ and 
$\{t_{\ell} \}_{\ell=1}^{N_{\rm v}} \subset (0,1)$ such that
$$
\{ y_{\ell}  =\boldsymbol{\gamma}(t_\ell)   \}_{\ell=1}^{N_{\rm v}} 
=
V_1^{\rm fict} \cup 
V_2^{\rm fict} ,
\quad
{\rm with} \;\;
0< t_1 < \ldots < t_{N_{\rm v}}< 1.
$$
Definition \eqref{eq:constantC} implies that 
${\rm dist}_{\partial \Omega} (y_{\ell},y_{{\ell+1}}) \geq C$, while the hypothesis
$\max_{x\in \partial \Omega}$ 
$ {\rm dist}_{\partial \Omega} $
$(x,\Phi(x))$ $< C$ implies that
\begin{equation}
\label{eq:messy_proof_multilayer}
\Phi(y_{\ell}) \in \left\{
\boldsymbol{\gamma}(t) : 
t\in (t_{\ell-1}, t_{\ell+1}) 
\right\},
\quad
{\rm for} \;
\ell=1,\ldots, N_{\rm v},
\end{equation}
with $t_{\ell=0}=0$, 
$t_{\ell=N_{\rm v}+1}=1$.

Exploiting \eqref{eq:messy_proof_multilayer} and the fact that $\Phi$ does not deform angular points of $\partial \Omega$ (cf. Proposition \ref{th:vertices_plus_density}), we find that there exists a diffeomorphism $\widetilde{\Phi}_1$  such that
$\widetilde{\Phi}_1(y)=y$ for all $y\in V_1$ and
$\widetilde{\Phi}_1(y)=\Phi(y)$ for all $y\in V_2$. Exploiting Lemma \ref{th:density_composition}, we find that 
$\widetilde{\Phi}_1 = 
\Psi_1 \circ \Phi_1 \circ \Psi_1^{-1}$ for some $\Phi_1\in  \texttt{id} + {\mathfrak{U}}_0(\Omega_{{\rm p}, 1})$.
Similarly, since 
$\widetilde{\Phi}_1^{-1}\circ \Phi (y) = y$ for all $y\in V_2$, we find that
$\widetilde{\Phi}_1^{-1}\circ \Phi
=:\widetilde{\Phi}_2 = \Psi_2 \circ \Phi_2 \circ \Psi_2^{-1}$ for some 
$\Phi_2\in  \texttt{id} + {\mathfrak{U}}_0(\Omega_{{\rm p}, 2})$. In conclusion, we find that 
$\Phi=\widetilde{\Phi}_1 \circ \widetilde{\Phi}_2$ which is the desired result.
\end{proof}

The extension to multiple layers enables the approximation of  diffeomorphisms that involve  larger deformations over curved edges. 
  Figure \ref{fig:multilayer_map} illustrates the approximation power of multi-layer maps for deformations over an airfoil:
Figure \ref{fig:multilayer_map}(a)  shows the  two polytopes
$\Omega_{\rm p,1}, \Omega_{\rm p,2}$
 that are used to define the nonlinear ansatz   and two points $x,y$ on the profile such that $\Phi(x)=y$;
Figures \ref{fig:multilayer_map}(b)-(c)-(d)  show the  action of  the maps
$\texttt{N}_3$, $\texttt{N}_2$ and 
$\texttt{N}_1$, respectively.
In more detail, the map $\texttt{N}_3$
--- which is associated to $\Omega_{\rm p,1}$ ---
 maps the point $x$  into the point $x_1$ (cf. Figure  \ref{fig:multilayer_map}(b));
the map $\texttt{N}_2$
--- which is associated to $\Omega_{\rm p,2}$ --- deforms the point $x_1$  into the point $x_2$ (cf. Figure  \ref{fig:multilayer_map}(c));
the map $\texttt{N}_1$
--- which is associated to $\Omega_{\rm p,1}$ ---
 deforms the point $x_2$   into  the point $y$ (cf. Figure  \ref{fig:multilayer_map}(d)).
In conclusion, the considered map satisfies
$\texttt{N}_1(\texttt{N}_2(\texttt{N}_3(x ))) = y$.

\begin{figure}[h!]
\centering
\subfloat[]{
\begin{tikzpicture}[scale=0.53]
\draw[fill] (8.7,3.1) circle (5pt);
\coordinate [label={above:  {\Large{$x$}}}] (E) at (8.7,3.1) ;

\coordinate [label={above:  {\Large{$\textcolor{blue}{\Omega_{\rm p,2}}$}}}] (E) at (7.7,3.5) ;

\draw[fill] (2.05,1.23) circle (5pt);
\coordinate [label={left:  {\Large{$y$}}}] (E) at (2.05,1.23) ;

\coordinate [label={left:  {\Large{$\textcolor{red}{\Omega_{\rm p,1}}$}}}] (E) at (2.4,4) ;

\begin{axis}[
axis lines=none,
width=\textwidth,
ymax=0.5,
xmax=1.2,
xmin=-0.2,
ymin=-0.5,
height=0.5\textwidth]

{\addplot [thick, black, solid, forget plot]
coordinates {
(0, 0) 
(0.005, 0.081421) 
(0.01, 0.11358) 
(0.015, 0.137578) 
(0.02, 0.157318) 
(0.025, 0.174315) 
(0.03, 0.189343) 
(0.035, 0.202865) 
(0.04, 0.215181) 
(0.045, 0.226502) 
(0.05, 0.236979) 
(0.055, 0.246728) 
(0.06, 0.255839) 
(0.065, 0.264383) 
(0.07, 0.272418) 
(0.075, 0.279993) 
(0.08, 0.287148) 
(0.085, 0.293917) 
(0.09, 0.30033) 
(0.095, 0.306411) 
(0.1, 0.312184) 
(0.105, 0.317667) 
(0.11, 0.322878) 
(0.115, 0.327832) 
(0.12, 0.332544) 
(0.125, 0.337025) 
(0.13, 0.341287) 
(0.135, 0.345341) 
(0.14, 0.349195) 
(0.145, 0.352859) 
(0.15, 0.35634) 
(0.155, 0.359645) 
(0.16, 0.362783) 
(0.165, 0.365757) 
(0.17, 0.368576) 
(0.175, 0.371244) 
(0.18, 0.373766) 
(0.185, 0.376147) 
(0.19, 0.378392) 
(0.195, 0.380505) 
(0.2, 0.382489) 
(0.205, 0.38435) 
(0.21, 0.38609) 
(0.215, 0.387713) 
(0.22, 0.389222) 
(0.225, 0.39062) 
(0.23, 0.391911) 
(0.235, 0.393097) 
(0.24, 0.39418) 
(0.245, 0.395164) 
(0.25, 0.39605) 
(0.255, 0.396842) 
(0.26, 0.397541) 
(0.265, 0.398149) 
(0.27, 0.39867) 
(0.275, 0.399104) 
(0.28, 0.399454) 
(0.285, 0.399721) 
(0.29, 0.399908) 
(0.295, 0.400016) 
(0.3, 0.400047) 
(0.305, 0.400002) 
(0.31, 0.399884) 
(0.315, 0.399693) 
(0.32, 0.399431) 
(0.325, 0.3991) 
(0.33, 0.3987) 
(0.335, 0.398234) 
(0.34, 0.397703) 
(0.345, 0.397107) 
(0.35, 0.396448) 
(0.355, 0.395728) 
(0.36, 0.394947) 
(0.365, 0.394107) 
(0.37, 0.393209) 
(0.375, 0.392253) 
(0.38, 0.391241) 
(0.385, 0.390174) 
(0.39, 0.389053) 
(0.395, 0.387879) 
(0.4, 0.386652) 
(0.405, 0.385374) 
(0.41, 0.384046) 
(0.415, 0.382668) 
(0.42, 0.381242) 
(0.425, 0.379767) 
(0.43, 0.378246) 
(0.435, 0.376678) 
(0.44, 0.375065) 
(0.445, 0.373407) 
(0.45, 0.371705) 
(0.455, 0.369959) 
(0.46, 0.368171) 
(0.465, 0.366341) 
(0.47, 0.36447) 
(0.475, 0.362558) 
(0.48, 0.360606) 
(0.485, 0.358615) 
(0.49, 0.356584) 
(0.495, 0.354516) 
(0.5, 0.35241) 
(0.505, 0.350267) 
(0.51, 0.348087) 
(0.515, 0.345871) 
(0.52, 0.34362) 
(0.525, 0.341334) 
(0.53, 0.339013) 
(0.535, 0.336658) 
(0.54, 0.334269) 
(0.545, 0.331848) 
(0.55, 0.329393) 
(0.555, 0.326906) 
(0.56, 0.324388) 
(0.565, 0.321838) 
(0.57, 0.319256) 
(0.575, 0.316644) 
(0.58, 0.314002) 
(0.585, 0.31133) 
(0.59, 0.308628) 
(0.595, 0.305896) 
(0.6, 0.303136) 
(0.605, 0.300347) 
(0.61, 0.297529) 
(0.615, 0.294684) 
(0.62, 0.29181) 
(0.625, 0.28891) 
(0.63, 0.285981) 
(0.635, 0.283026) 
(0.64, 0.280044) 
(0.645, 0.277035) 
(0.65, 0.274) 
(0.655, 0.270939) 
(0.66, 0.267852) 
(0.665, 0.264739) 
(0.67, 0.261601) 
(0.675, 0.258437) 
(0.68, 0.255248) 
(0.685, 0.252034) 
(0.69, 0.248795) 
(0.695, 0.245532) 
(0.7, 0.242244) 
(0.705, 0.238931) 
(0.71, 0.235594) 
(0.715, 0.232232) 
(0.72, 0.228847) 
(0.725, 0.225437) 
(0.73, 0.222003) 
(0.735, 0.218545) 
(0.74, 0.215064) 
(0.745, 0.211558) 
(0.75, 0.208029) 
(0.755, 0.204476) 
(0.76, 0.2009) 
(0.765, 0.197299) 
(0.77, 0.193675) 
(0.775, 0.190028) 
(0.78, 0.186356) 
(0.785, 0.182661) 
(0.79, 0.178943) 
(0.795, 0.1752) 
(0.8, 0.171434) 
(0.805, 0.167644) 
(0.81, 0.16383) 
(0.815, 0.159992) 
(0.82, 0.156131) 
(0.825, 0.152245) 
(0.83, 0.148335) 
(0.835, 0.144401) 
(0.84, 0.140443) 
(0.845, 0.13646) 
(0.85, 0.132453) 
(0.855, 0.128421) 
(0.86, 0.124364) 
(0.865, 0.120283) 
(0.87, 0.116176) 
(0.875, 0.112045) 
(0.88, 0.107888) 
(0.885, 0.103705) 
(0.89, 0.0994971) 
(0.895, 0.0952632) 
(0.9, 0.0910032) 
(0.905, 0.0867171) 
(0.91, 0.0824047) 
(0.915, 0.0780656) 
(0.92, 0.0736997) 
(0.925, 0.0693068) 
(0.93, 0.0648866) 
(0.935, 0.060439) 
(0.94, 0.0559636) 
(0.945, 0.0514603) 
(0.95, 0.0469288) 
(0.955, 0.0423687) 
(0.96, 0.0377799) 
(0.965, 0.0331621) 
(0.97, 0.028515) 
(0.975, 0.0238383) 
(0.98, 0.0191316) 
(0.985, 0.0143948) 
(0.99, 0.00962748) 
(0.995, 0.00482931) 
(1, -1.11022e-16) 
 };}

{\addplot [thick, black, solid, forget plot]
coordinates {
(0, -0) 
(0.005, -0.081421) 
(0.01, -0.11358) 
(0.015, -0.137578) 
(0.02, -0.157318) 
(0.025, -0.174315) 
(0.03, -0.189343) 
(0.035, -0.202865) 
(0.04, -0.215181) 
(0.045, -0.226502) 
(0.05, -0.236979) 
(0.055, -0.246728) 
(0.06, -0.255839) 
(0.065, -0.264383) 
(0.07, -0.272418) 
(0.075, -0.279993) 
(0.08, -0.287148) 
(0.085, -0.293917) 
(0.09, -0.30033) 
(0.095, -0.306411) 
(0.1, -0.312184) 
(0.105, -0.317667) 
(0.11, -0.322878) 
(0.115, -0.327832) 
(0.12, -0.332544) 
(0.125, -0.337025) 
(0.13, -0.341287) 
(0.135, -0.345341) 
(0.14, -0.349195) 
(0.145, -0.352859) 
(0.15, -0.35634) 
(0.155, -0.359645) 
(0.16, -0.362783) 
(0.165, -0.365757) 
(0.17, -0.368576) 
(0.175, -0.371244) 
(0.18, -0.373766) 
(0.185, -0.376147) 
(0.19, -0.378392) 
(0.195, -0.380505) 
(0.2, -0.382489) 
(0.205, -0.38435) 
(0.21, -0.38609) 
(0.215, -0.387713) 
(0.22, -0.389222) 
(0.225, -0.39062) 
(0.23, -0.391911) 
(0.235, -0.393097) 
(0.24, -0.39418) 
(0.245, -0.395164) 
(0.25, -0.39605) 
(0.255, -0.396842) 
(0.26, -0.397541) 
(0.265, -0.398149) 
(0.27, -0.39867) 
(0.275, -0.399104) 
(0.28, -0.399454) 
(0.285, -0.399721) 
(0.29, -0.399908) 
(0.295, -0.400016) 
(0.3, -0.400047) 
(0.305, -0.400002) 
(0.31, -0.399884) 
(0.315, -0.399693) 
(0.32, -0.399431) 
(0.325, -0.3991) 
(0.33, -0.3987) 
(0.335, -0.398234) 
(0.34, -0.397703) 
(0.345, -0.397107) 
(0.35, -0.396448) 
(0.355, -0.395728) 
(0.36, -0.394947) 
(0.365, -0.394107) 
(0.37, -0.393209) 
(0.375, -0.392253) 
(0.38, -0.391241) 
(0.385, -0.390174) 
(0.39, -0.389053) 
(0.395, -0.387879) 
(0.4, -0.386652) 
(0.405, -0.385374) 
(0.41, -0.384046) 
(0.415, -0.382668) 
(0.42, -0.381242) 
(0.425, -0.379767) 
(0.43, -0.378246) 
(0.435, -0.376678) 
(0.44, -0.375065) 
(0.445, -0.373407) 
(0.45, -0.371705) 
(0.455, -0.369959) 
(0.46, -0.368171) 
(0.465, -0.366341) 
(0.47, -0.36447) 
(0.475, -0.362558) 
(0.48, -0.360606) 
(0.485, -0.358615) 
(0.49, -0.356584) 
(0.495, -0.354516) 
(0.5, -0.35241) 
(0.505, -0.350267) 
(0.51, -0.348087) 
(0.515, -0.345871) 
(0.52, -0.34362) 
(0.525, -0.341334) 
(0.53, -0.339013) 
(0.535, -0.336658) 
(0.54, -0.334269) 
(0.545, -0.331848) 
(0.55, -0.329393) 
(0.555, -0.326906) 
(0.56, -0.324388) 
(0.565, -0.321838) 
(0.57, -0.319256) 
(0.575, -0.316644) 
(0.58, -0.314002) 
(0.585, -0.31133) 
(0.59, -0.308628) 
(0.595, -0.305896) 
(0.6, -0.303136) 
(0.605, -0.300347) 
(0.61, -0.297529) 
(0.615, -0.294684) 
(0.62, -0.29181) 
(0.625, -0.28891) 
(0.63, -0.285981) 
(0.635, -0.283026) 
(0.64, -0.280044) 
(0.645, -0.277035) 
(0.65, -0.274) 
(0.655, -0.270939) 
(0.66, -0.267852) 
(0.665, -0.264739) 
(0.67, -0.261601) 
(0.675, -0.258437) 
(0.68, -0.255248) 
(0.685, -0.252034) 
(0.69, -0.248795) 
(0.695, -0.245532) 
(0.7, -0.242244) 
(0.705, -0.238931) 
(0.71, -0.235594) 
(0.715, -0.232232) 
(0.72, -0.228847) 
(0.725, -0.225437) 
(0.73, -0.222003) 
(0.735, -0.218545) 
(0.74, -0.215064) 
(0.745, -0.211558) 
(0.75, -0.208029) 
(0.755, -0.204476) 
(0.76, -0.2009) 
(0.765, -0.197299) 
(0.77, -0.193675) 
(0.775, -0.190028) 
(0.78, -0.186356) 
(0.785, -0.182661) 
(0.79, -0.178943) 
(0.795, -0.1752) 
(0.8, -0.171434) 
(0.805, -0.167644) 
(0.81, -0.16383) 
(0.815, -0.159992) 
(0.82, -0.156131) 
(0.825, -0.152245) 
(0.83, -0.148335) 
(0.835, -0.144401) 
(0.84, -0.140443) 
(0.845, -0.13646) 
(0.85, -0.132453) 
(0.855, -0.128421) 
(0.86, -0.124364) 
(0.865, -0.120283) 
(0.87, -0.116176) 
(0.875, -0.112045) 
(0.88, -0.107888) 
(0.885, -0.103705) 
(0.89, -0.0994971) 
(0.895, -0.0952632) 
(0.9, -0.0910032) 
(0.905, -0.0867171) 
(0.91, -0.0824047) 
(0.915, -0.0780656) 
(0.92, -0.0736997) 
(0.925, -0.0693068) 
(0.93, -0.0648866) 
(0.935, -0.060439) 
(0.94, -0.0559636) 
(0.945, -0.0514603) 
(0.95, -0.0469288) 
(0.955, -0.0423687) 
(0.96, -0.0377799) 
(0.965, -0.0331621) 
(0.97, -0.028515) 
(0.975, -0.0238383) 
(0.98, -0.0191316) 
(0.985, -0.0143948) 
(0.99, -0.00962748) 
(0.995, -0.00482931) 
(1, 0) 
};}

{\addplot [ultra thick, red, dashed, forget plot]
coordinates {
(0.1, 0.312184) 
(1, 0) 
(0.1, -0.312184) 
(0.1, 0.312184) 
};}

{\addplot [ultra thick, blue, dashed, forget plot]
coordinates {
(0, 0) 
(0.5, 0.35241) 
(1, 0) 
(0.5, -0.35241) 
(0, 0)
};}

\end{axis}
\end{tikzpicture}
}
~~
\subfloat[$\texttt{N}_3$]{ 
\begin{tikzpicture}[scale=0.53]
\draw[fill] (8.7,3.1) circle (5pt);
\coordinate [label={above:  {\Large{$x$}}}] (E) at (8.7,3.1) ;

\draw[fill] (2.05,1.23) circle (5pt);
\coordinate [label={left:  {\Large{$y$}}}] (E) at (2.05,1.23) ;

\draw[fill] (3.3,4.33) circle (5pt);
\coordinate [label={above:  {\Large{$x_1$}}}] (E) at (3.3,4.33) ;

\draw[->,dashed,red] (8.6,3.15) -- (3.63,4.3);

\begin{axis}[
axis lines=none,
width=\textwidth,
ymax=0.5,
xmax=1.2,
xmin=-0.2,
ymin=-0.5,
height=0.5\textwidth]

{\addplot [thick, black, solid, forget plot]
coordinates {
(0, 0) 
(0.005, 0.081421) 
(0.01, 0.11358) 
(0.015, 0.137578) 
(0.02, 0.157318) 
(0.025, 0.174315) 
(0.03, 0.189343) 
(0.035, 0.202865) 
(0.04, 0.215181) 
(0.045, 0.226502) 
(0.05, 0.236979) 
(0.055, 0.246728) 
(0.06, 0.255839) 
(0.065, 0.264383) 
(0.07, 0.272418) 
(0.075, 0.279993) 
(0.08, 0.287148) 
(0.085, 0.293917) 
(0.09, 0.30033) 
(0.095, 0.306411) 
(0.1, 0.312184) 
(0.105, 0.317667) 
(0.11, 0.322878) 
(0.115, 0.327832) 
(0.12, 0.332544) 
(0.125, 0.337025) 
(0.13, 0.341287) 
(0.135, 0.345341) 
(0.14, 0.349195) 
(0.145, 0.352859) 
(0.15, 0.35634) 
(0.155, 0.359645) 
(0.16, 0.362783) 
(0.165, 0.365757) 
(0.17, 0.368576) 
(0.175, 0.371244) 
(0.18, 0.373766) 
(0.185, 0.376147) 
(0.19, 0.378392) 
(0.195, 0.380505) 
(0.2, 0.382489) 
(0.205, 0.38435) 
(0.21, 0.38609) 
(0.215, 0.387713) 
(0.22, 0.389222) 
(0.225, 0.39062) 
(0.23, 0.391911) 
(0.235, 0.393097) 
(0.24, 0.39418) 
(0.245, 0.395164) 
(0.25, 0.39605) 
(0.255, 0.396842) 
(0.26, 0.397541) 
(0.265, 0.398149) 
(0.27, 0.39867) 
(0.275, 0.399104) 
(0.28, 0.399454) 
(0.285, 0.399721) 
(0.29, 0.399908) 
(0.295, 0.400016) 
(0.3, 0.400047) 
(0.305, 0.400002) 
(0.31, 0.399884) 
(0.315, 0.399693) 
(0.32, 0.399431) 
(0.325, 0.3991) 
(0.33, 0.3987) 
(0.335, 0.398234) 
(0.34, 0.397703) 
(0.345, 0.397107) 
(0.35, 0.396448) 
(0.355, 0.395728) 
(0.36, 0.394947) 
(0.365, 0.394107) 
(0.37, 0.393209) 
(0.375, 0.392253) 
(0.38, 0.391241) 
(0.385, 0.390174) 
(0.39, 0.389053) 
(0.395, 0.387879) 
(0.4, 0.386652) 
(0.405, 0.385374) 
(0.41, 0.384046) 
(0.415, 0.382668) 
(0.42, 0.381242) 
(0.425, 0.379767) 
(0.43, 0.378246) 
(0.435, 0.376678) 
(0.44, 0.375065) 
(0.445, 0.373407) 
(0.45, 0.371705) 
(0.455, 0.369959) 
(0.46, 0.368171) 
(0.465, 0.366341) 
(0.47, 0.36447) 
(0.475, 0.362558) 
(0.48, 0.360606) 
(0.485, 0.358615) 
(0.49, 0.356584) 
(0.495, 0.354516) 
(0.5, 0.35241) 
(0.505, 0.350267) 
(0.51, 0.348087) 
(0.515, 0.345871) 
(0.52, 0.34362) 
(0.525, 0.341334) 
(0.53, 0.339013) 
(0.535, 0.336658) 
(0.54, 0.334269) 
(0.545, 0.331848) 
(0.55, 0.329393) 
(0.555, 0.326906) 
(0.56, 0.324388) 
(0.565, 0.321838) 
(0.57, 0.319256) 
(0.575, 0.316644) 
(0.58, 0.314002) 
(0.585, 0.31133) 
(0.59, 0.308628) 
(0.595, 0.305896) 
(0.6, 0.303136) 
(0.605, 0.300347) 
(0.61, 0.297529) 
(0.615, 0.294684) 
(0.62, 0.29181) 
(0.625, 0.28891) 
(0.63, 0.285981) 
(0.635, 0.283026) 
(0.64, 0.280044) 
(0.645, 0.277035) 
(0.65, 0.274) 
(0.655, 0.270939) 
(0.66, 0.267852) 
(0.665, 0.264739) 
(0.67, 0.261601) 
(0.675, 0.258437) 
(0.68, 0.255248) 
(0.685, 0.252034) 
(0.69, 0.248795) 
(0.695, 0.245532) 
(0.7, 0.242244) 
(0.705, 0.238931) 
(0.71, 0.235594) 
(0.715, 0.232232) 
(0.72, 0.228847) 
(0.725, 0.225437) 
(0.73, 0.222003) 
(0.735, 0.218545) 
(0.74, 0.215064) 
(0.745, 0.211558) 
(0.75, 0.208029) 
(0.755, 0.204476) 
(0.76, 0.2009) 
(0.765, 0.197299) 
(0.77, 0.193675) 
(0.775, 0.190028) 
(0.78, 0.186356) 
(0.785, 0.182661) 
(0.79, 0.178943) 
(0.795, 0.1752) 
(0.8, 0.171434) 
(0.805, 0.167644) 
(0.81, 0.16383) 
(0.815, 0.159992) 
(0.82, 0.156131) 
(0.825, 0.152245) 
(0.83, 0.148335) 
(0.835, 0.144401) 
(0.84, 0.140443) 
(0.845, 0.13646) 
(0.85, 0.132453) 
(0.855, 0.128421) 
(0.86, 0.124364) 
(0.865, 0.120283) 
(0.87, 0.116176) 
(0.875, 0.112045) 
(0.88, 0.107888) 
(0.885, 0.103705) 
(0.89, 0.0994971) 
(0.895, 0.0952632) 
(0.9, 0.0910032) 
(0.905, 0.0867171) 
(0.91, 0.0824047) 
(0.915, 0.0780656) 
(0.92, 0.0736997) 
(0.925, 0.0693068) 
(0.93, 0.0648866) 
(0.935, 0.060439) 
(0.94, 0.0559636) 
(0.945, 0.0514603) 
(0.95, 0.0469288) 
(0.955, 0.0423687) 
(0.96, 0.0377799) 
(0.965, 0.0331621) 
(0.97, 0.028515) 
(0.975, 0.0238383) 
(0.98, 0.0191316) 
(0.985, 0.0143948) 
(0.99, 0.00962748) 
(0.995, 0.00482931) 
(1, -1.11022e-16) 
 };}

{\addplot [thick, black, solid, forget plot]
coordinates {
(0, -0) 
(0.005, -0.081421) 
(0.01, -0.11358) 
(0.015, -0.137578) 
(0.02, -0.157318) 
(0.025, -0.174315) 
(0.03, -0.189343) 
(0.035, -0.202865) 
(0.04, -0.215181) 
(0.045, -0.226502) 
(0.05, -0.236979) 
(0.055, -0.246728) 
(0.06, -0.255839) 
(0.065, -0.264383) 
(0.07, -0.272418) 
(0.075, -0.279993) 
(0.08, -0.287148) 
(0.085, -0.293917) 
(0.09, -0.30033) 
(0.095, -0.306411) 
(0.1, -0.312184) 
(0.105, -0.317667) 
(0.11, -0.322878) 
(0.115, -0.327832) 
(0.12, -0.332544) 
(0.125, -0.337025) 
(0.13, -0.341287) 
(0.135, -0.345341) 
(0.14, -0.349195) 
(0.145, -0.352859) 
(0.15, -0.35634) 
(0.155, -0.359645) 
(0.16, -0.362783) 
(0.165, -0.365757) 
(0.17, -0.368576) 
(0.175, -0.371244) 
(0.18, -0.373766) 
(0.185, -0.376147) 
(0.19, -0.378392) 
(0.195, -0.380505) 
(0.2, -0.382489) 
(0.205, -0.38435) 
(0.21, -0.38609) 
(0.215, -0.387713) 
(0.22, -0.389222) 
(0.225, -0.39062) 
(0.23, -0.391911) 
(0.235, -0.393097) 
(0.24, -0.39418) 
(0.245, -0.395164) 
(0.25, -0.39605) 
(0.255, -0.396842) 
(0.26, -0.397541) 
(0.265, -0.398149) 
(0.27, -0.39867) 
(0.275, -0.399104) 
(0.28, -0.399454) 
(0.285, -0.399721) 
(0.29, -0.399908) 
(0.295, -0.400016) 
(0.3, -0.400047) 
(0.305, -0.400002) 
(0.31, -0.399884) 
(0.315, -0.399693) 
(0.32, -0.399431) 
(0.325, -0.3991) 
(0.33, -0.3987) 
(0.335, -0.398234) 
(0.34, -0.397703) 
(0.345, -0.397107) 
(0.35, -0.396448) 
(0.355, -0.395728) 
(0.36, -0.394947) 
(0.365, -0.394107) 
(0.37, -0.393209) 
(0.375, -0.392253) 
(0.38, -0.391241) 
(0.385, -0.390174) 
(0.39, -0.389053) 
(0.395, -0.387879) 
(0.4, -0.386652) 
(0.405, -0.385374) 
(0.41, -0.384046) 
(0.415, -0.382668) 
(0.42, -0.381242) 
(0.425, -0.379767) 
(0.43, -0.378246) 
(0.435, -0.376678) 
(0.44, -0.375065) 
(0.445, -0.373407) 
(0.45, -0.371705) 
(0.455, -0.369959) 
(0.46, -0.368171) 
(0.465, -0.366341) 
(0.47, -0.36447) 
(0.475, -0.362558) 
(0.48, -0.360606) 
(0.485, -0.358615) 
(0.49, -0.356584) 
(0.495, -0.354516) 
(0.5, -0.35241) 
(0.505, -0.350267) 
(0.51, -0.348087) 
(0.515, -0.345871) 
(0.52, -0.34362) 
(0.525, -0.341334) 
(0.53, -0.339013) 
(0.535, -0.336658) 
(0.54, -0.334269) 
(0.545, -0.331848) 
(0.55, -0.329393) 
(0.555, -0.326906) 
(0.56, -0.324388) 
(0.565, -0.321838) 
(0.57, -0.319256) 
(0.575, -0.316644) 
(0.58, -0.314002) 
(0.585, -0.31133) 
(0.59, -0.308628) 
(0.595, -0.305896) 
(0.6, -0.303136) 
(0.605, -0.300347) 
(0.61, -0.297529) 
(0.615, -0.294684) 
(0.62, -0.29181) 
(0.625, -0.28891) 
(0.63, -0.285981) 
(0.635, -0.283026) 
(0.64, -0.280044) 
(0.645, -0.277035) 
(0.65, -0.274) 
(0.655, -0.270939) 
(0.66, -0.267852) 
(0.665, -0.264739) 
(0.67, -0.261601) 
(0.675, -0.258437) 
(0.68, -0.255248) 
(0.685, -0.252034) 
(0.69, -0.248795) 
(0.695, -0.245532) 
(0.7, -0.242244) 
(0.705, -0.238931) 
(0.71, -0.235594) 
(0.715, -0.232232) 
(0.72, -0.228847) 
(0.725, -0.225437) 
(0.73, -0.222003) 
(0.735, -0.218545) 
(0.74, -0.215064) 
(0.745, -0.211558) 
(0.75, -0.208029) 
(0.755, -0.204476) 
(0.76, -0.2009) 
(0.765, -0.197299) 
(0.77, -0.193675) 
(0.775, -0.190028) 
(0.78, -0.186356) 
(0.785, -0.182661) 
(0.79, -0.178943) 
(0.795, -0.1752) 
(0.8, -0.171434) 
(0.805, -0.167644) 
(0.81, -0.16383) 
(0.815, -0.159992) 
(0.82, -0.156131) 
(0.825, -0.152245) 
(0.83, -0.148335) 
(0.835, -0.144401) 
(0.84, -0.140443) 
(0.845, -0.13646) 
(0.85, -0.132453) 
(0.855, -0.128421) 
(0.86, -0.124364) 
(0.865, -0.120283) 
(0.87, -0.116176) 
(0.875, -0.112045) 
(0.88, -0.107888) 
(0.885, -0.103705) 
(0.89, -0.0994971) 
(0.895, -0.0952632) 
(0.9, -0.0910032) 
(0.905, -0.0867171) 
(0.91, -0.0824047) 
(0.915, -0.0780656) 
(0.92, -0.0736997) 
(0.925, -0.0693068) 
(0.93, -0.0648866) 
(0.935, -0.060439) 
(0.94, -0.0559636) 
(0.945, -0.0514603) 
(0.95, -0.0469288) 
(0.955, -0.0423687) 
(0.96, -0.0377799) 
(0.965, -0.0331621) 
(0.97, -0.028515) 
(0.975, -0.0238383) 
(0.98, -0.0191316) 
(0.985, -0.0143948) 
(0.99, -0.00962748) 
(0.995, -0.00482931) 
(1, 0) 
};}

{\addplot [ultra thick, red, dashed, forget plot]
coordinates {
(0.1, 0.312184) 
(1, 0) 
(0.1, -0.312184) 
(0.1, 0.312184) 
};}

\end{axis}
\end{tikzpicture}
}

\subfloat[$\texttt{N}_2$]{
\begin{tikzpicture}[scale=0.53]
\draw[fill] (8.7,3.1) circle (5pt);
\coordinate [label={above:  {\Large{$x$}}}] (E) at (8.7,3.1) ;

\draw[fill] (2.05,1.23) circle (5pt);
\coordinate [label={left:  {\Large{$y$}}}] (E) at (2.05,1.23) ;

\draw[fill] (3.3,4.33) circle (5pt);
\coordinate [label={above:  {\Large{$x_1$}}}] (E) at (3.3,4.33) ;

\draw[fill] (1.74,3) circle (5pt);
\coordinate [label={left:  {\Large{$x_2$}}}] (E) at (1.74,3) ;

\draw[->,dashed,blue] (2.95,4.1) -- (1.85,3.2);

%
%
%
%
%

\begin{axis}[
axis lines=none,
width=\textwidth,
ymax=0.5,
xmax=1.2,
xmin=-0.2,
ymin=-0.5,
height=0.5\textwidth]

{\addplot [thick, black, solid, forget plot]
coordinates {
(0, 0) 
(0.005, 0.081421) 
(0.01, 0.11358) 
(0.015, 0.137578) 
(0.02, 0.157318) 
(0.025, 0.174315) 
(0.03, 0.189343) 
(0.035, 0.202865) 
(0.04, 0.215181) 
(0.045, 0.226502) 
(0.05, 0.236979) 
(0.055, 0.246728) 
(0.06, 0.255839) 
(0.065, 0.264383) 
(0.07, 0.272418) 
(0.075, 0.279993) 
(0.08, 0.287148) 
(0.085, 0.293917) 
(0.09, 0.30033) 
(0.095, 0.306411) 
(0.1, 0.312184) 
(0.105, 0.317667) 
(0.11, 0.322878) 
(0.115, 0.327832) 
(0.12, 0.332544) 
(0.125, 0.337025) 
(0.13, 0.341287) 
(0.135, 0.345341) 
(0.14, 0.349195) 
(0.145, 0.352859) 
(0.15, 0.35634) 
(0.155, 0.359645) 
(0.16, 0.362783) 
(0.165, 0.365757) 
(0.17, 0.368576) 
(0.175, 0.371244) 
(0.18, 0.373766) 
(0.185, 0.376147) 
(0.19, 0.378392) 
(0.195, 0.380505) 
(0.2, 0.382489) 
(0.205, 0.38435) 
(0.21, 0.38609) 
(0.215, 0.387713) 
(0.22, 0.389222) 
(0.225, 0.39062) 
(0.23, 0.391911) 
(0.235, 0.393097) 
(0.24, 0.39418) 
(0.245, 0.395164) 
(0.25, 0.39605) 
(0.255, 0.396842) 
(0.26, 0.397541) 
(0.265, 0.398149) 
(0.27, 0.39867) 
(0.275, 0.399104) 
(0.28, 0.399454) 
(0.285, 0.399721) 
(0.29, 0.399908) 
(0.295, 0.400016) 
(0.3, 0.400047) 
(0.305, 0.400002) 
(0.31, 0.399884) 
(0.315, 0.399693) 
(0.32, 0.399431) 
(0.325, 0.3991) 
(0.33, 0.3987) 
(0.335, 0.398234) 
(0.34, 0.397703) 
(0.345, 0.397107) 
(0.35, 0.396448) 
(0.355, 0.395728) 
(0.36, 0.394947) 
(0.365, 0.394107) 
(0.37, 0.393209) 
(0.375, 0.392253) 
(0.38, 0.391241) 
(0.385, 0.390174) 
(0.39, 0.389053) 
(0.395, 0.387879) 
(0.4, 0.386652) 
(0.405, 0.385374) 
(0.41, 0.384046) 
(0.415, 0.382668) 
(0.42, 0.381242) 
(0.425, 0.379767) 
(0.43, 0.378246) 
(0.435, 0.376678) 
(0.44, 0.375065) 
(0.445, 0.373407) 
(0.45, 0.371705) 
(0.455, 0.369959) 
(0.46, 0.368171) 
(0.465, 0.366341) 
(0.47, 0.36447) 
(0.475, 0.362558) 
(0.48, 0.360606) 
(0.485, 0.358615) 
(0.49, 0.356584) 
(0.495, 0.354516) 
(0.5, 0.35241) 
(0.505, 0.350267) 
(0.51, 0.348087) 
(0.515, 0.345871) 
(0.52, 0.34362) 
(0.525, 0.341334) 
(0.53, 0.339013) 
(0.535, 0.336658) 
(0.54, 0.334269) 
(0.545, 0.331848) 
(0.55, 0.329393) 
(0.555, 0.326906) 
(0.56, 0.324388) 
(0.565, 0.321838) 
(0.57, 0.319256) 
(0.575, 0.316644) 
(0.58, 0.314002) 
(0.585, 0.31133) 
(0.59, 0.308628) 
(0.595, 0.305896) 
(0.6, 0.303136) 
(0.605, 0.300347) 
(0.61, 0.297529) 
(0.615, 0.294684) 
(0.62, 0.29181) 
(0.625, 0.28891) 
(0.63, 0.285981) 
(0.635, 0.283026) 
(0.64, 0.280044) 
(0.645, 0.277035) 
(0.65, 0.274) 
(0.655, 0.270939) 
(0.66, 0.267852) 
(0.665, 0.264739) 
(0.67, 0.261601) 
(0.675, 0.258437) 
(0.68, 0.255248) 
(0.685, 0.252034) 
(0.69, 0.248795) 
(0.695, 0.245532) 
(0.7, 0.242244) 
(0.705, 0.238931) 
(0.71, 0.235594) 
(0.715, 0.232232) 
(0.72, 0.228847) 
(0.725, 0.225437) 
(0.73, 0.222003) 
(0.735, 0.218545) 
(0.74, 0.215064) 
(0.745, 0.211558) 
(0.75, 0.208029) 
(0.755, 0.204476) 
(0.76, 0.2009) 
(0.765, 0.197299) 
(0.77, 0.193675) 
(0.775, 0.190028) 
(0.78, 0.186356) 
(0.785, 0.182661) 
(0.79, 0.178943) 
(0.795, 0.1752) 
(0.8, 0.171434) 
(0.805, 0.167644) 
(0.81, 0.16383) 
(0.815, 0.159992) 
(0.82, 0.156131) 
(0.825, 0.152245) 
(0.83, 0.148335) 
(0.835, 0.144401) 
(0.84, 0.140443) 
(0.845, 0.13646) 
(0.85, 0.132453) 
(0.855, 0.128421) 
(0.86, 0.124364) 
(0.865, 0.120283) 
(0.87, 0.116176) 
(0.875, 0.112045) 
(0.88, 0.107888) 
(0.885, 0.103705) 
(0.89, 0.0994971) 
(0.895, 0.0952632) 
(0.9, 0.0910032) 
(0.905, 0.0867171) 
(0.91, 0.0824047) 
(0.915, 0.0780656) 
(0.92, 0.0736997) 
(0.925, 0.0693068) 
(0.93, 0.0648866) 
(0.935, 0.060439) 
(0.94, 0.0559636) 
(0.945, 0.0514603) 
(0.95, 0.0469288) 
(0.955, 0.0423687) 
(0.96, 0.0377799) 
(0.965, 0.0331621) 
(0.97, 0.028515) 
(0.975, 0.0238383) 
(0.98, 0.0191316) 
(0.985, 0.0143948) 
(0.99, 0.00962748) 
(0.995, 0.00482931) 
(1, -1.11022e-16) 
 };}

{\addplot [thick, black, solid, forget plot]
coordinates {
(0, -0) 
(0.005, -0.081421) 
(0.01, -0.11358) 
(0.015, -0.137578) 
(0.02, -0.157318) 
(0.025, -0.174315) 
(0.03, -0.189343) 
(0.035, -0.202865) 
(0.04, -0.215181) 
(0.045, -0.226502) 
(0.05, -0.236979) 
(0.055, -0.246728) 
(0.06, -0.255839) 
(0.065, -0.264383) 
(0.07, -0.272418) 
(0.075, -0.279993) 
(0.08, -0.287148) 
(0.085, -0.293917) 
(0.09, -0.30033) 
(0.095, -0.306411) 
(0.1, -0.312184) 
(0.105, -0.317667) 
(0.11, -0.322878) 
(0.115, -0.327832) 
(0.12, -0.332544) 
(0.125, -0.337025) 
(0.13, -0.341287) 
(0.135, -0.345341) 
(0.14, -0.349195) 
(0.145, -0.352859) 
(0.15, -0.35634) 
(0.155, -0.359645) 
(0.16, -0.362783) 
(0.165, -0.365757) 
(0.17, -0.368576) 
(0.175, -0.371244) 
(0.18, -0.373766) 
(0.185, -0.376147) 
(0.19, -0.378392) 
(0.195, -0.380505) 
(0.2, -0.382489) 
(0.205, -0.38435) 
(0.21, -0.38609) 
(0.215, -0.387713) 
(0.22, -0.389222) 
(0.225, -0.39062) 
(0.23, -0.391911) 
(0.235, -0.393097) 
(0.24, -0.39418) 
(0.245, -0.395164) 
(0.25, -0.39605) 
(0.255, -0.396842) 
(0.26, -0.397541) 
(0.265, -0.398149) 
(0.27, -0.39867) 
(0.275, -0.399104) 
(0.28, -0.399454) 
(0.285, -0.399721) 
(0.29, -0.399908) 
(0.295, -0.400016) 
(0.3, -0.400047) 
(0.305, -0.400002) 
(0.31, -0.399884) 
(0.315, -0.399693) 
(0.32, -0.399431) 
(0.325, -0.3991) 
(0.33, -0.3987) 
(0.335, -0.398234) 
(0.34, -0.397703) 
(0.345, -0.397107) 
(0.35, -0.396448) 
(0.355, -0.395728) 
(0.36, -0.394947) 
(0.365, -0.394107) 
(0.37, -0.393209) 
(0.375, -0.392253) 
(0.38, -0.391241) 
(0.385, -0.390174) 
(0.39, -0.389053) 
(0.395, -0.387879) 
(0.4, -0.386652) 
(0.405, -0.385374) 
(0.41, -0.384046) 
(0.415, -0.382668) 
(0.42, -0.381242) 
(0.425, -0.379767) 
(0.43, -0.378246) 
(0.435, -0.376678) 
(0.44, -0.375065) 
(0.445, -0.373407) 
(0.45, -0.371705) 
(0.455, -0.369959) 
(0.46, -0.368171) 
(0.465, -0.366341) 
(0.47, -0.36447) 
(0.475, -0.362558) 
(0.48, -0.360606) 
(0.485, -0.358615) 
(0.49, -0.356584) 
(0.495, -0.354516) 
(0.5, -0.35241) 
(0.505, -0.350267) 
(0.51, -0.348087) 
(0.515, -0.345871) 
(0.52, -0.34362) 
(0.525, -0.341334) 
(0.53, -0.339013) 
(0.535, -0.336658) 
(0.54, -0.334269) 
(0.545, -0.331848) 
(0.55, -0.329393) 
(0.555, -0.326906) 
(0.56, -0.324388) 
(0.565, -0.321838) 
(0.57, -0.319256) 
(0.575, -0.316644) 
(0.58, -0.314002) 
(0.585, -0.31133) 
(0.59, -0.308628) 
(0.595, -0.305896) 
(0.6, -0.303136) 
(0.605, -0.300347) 
(0.61, -0.297529) 
(0.615, -0.294684) 
(0.62, -0.29181) 
(0.625, -0.28891) 
(0.63, -0.285981) 
(0.635, -0.283026) 
(0.64, -0.280044) 
(0.645, -0.277035) 
(0.65, -0.274) 
(0.655, -0.270939) 
(0.66, -0.267852) 
(0.665, -0.264739) 
(0.67, -0.261601) 
(0.675, -0.258437) 
(0.68, -0.255248) 
(0.685, -0.252034) 
(0.69, -0.248795) 
(0.695, -0.245532) 
(0.7, -0.242244) 
(0.705, -0.238931) 
(0.71, -0.235594) 
(0.715, -0.232232) 
(0.72, -0.228847) 
(0.725, -0.225437) 
(0.73, -0.222003) 
(0.735, -0.218545) 
(0.74, -0.215064) 
(0.745, -0.211558) 
(0.75, -0.208029) 
(0.755, -0.204476) 
(0.76, -0.2009) 
(0.765, -0.197299) 
(0.77, -0.193675) 
(0.775, -0.190028) 
(0.78, -0.186356) 
(0.785, -0.182661) 
(0.79, -0.178943) 
(0.795, -0.1752) 
(0.8, -0.171434) 
(0.805, -0.167644) 
(0.81, -0.16383) 
(0.815, -0.159992) 
(0.82, -0.156131) 
(0.825, -0.152245) 
(0.83, -0.148335) 
(0.835, -0.144401) 
(0.84, -0.140443) 
(0.845, -0.13646) 
(0.85, -0.132453) 
(0.855, -0.128421) 
(0.86, -0.124364) 
(0.865, -0.120283) 
(0.87, -0.116176) 
(0.875, -0.112045) 
(0.88, -0.107888) 
(0.885, -0.103705) 
(0.89, -0.0994971) 
(0.895, -0.0952632) 
(0.9, -0.0910032) 
(0.905, -0.0867171) 
(0.91, -0.0824047) 
(0.915, -0.0780656) 
(0.92, -0.0736997) 
(0.925, -0.0693068) 
(0.93, -0.0648866) 
(0.935, -0.060439) 
(0.94, -0.0559636) 
(0.945, -0.0514603) 
(0.95, -0.0469288) 
(0.955, -0.0423687) 
(0.96, -0.0377799) 
(0.965, -0.0331621) 
(0.97, -0.028515) 
(0.975, -0.0238383) 
(0.98, -0.0191316) 
(0.985, -0.0143948) 
(0.99, -0.00962748) 
(0.995, -0.00482931) 
(1, 0) 
};}

{\addplot [ultra thick, blue, dashed, forget plot]
coordinates {
(0, 0) 
(0.5, 0.35241) 
(1, 0) 
(0.5, -0.35241) 
(0, 0)
};}

\end{axis}
\end{tikzpicture}
}
~~
\subfloat[$\texttt{N}_1$]{
\begin{tikzpicture}[scale=0.53]
\draw[fill] (8.7,3.1) circle (5pt);
\coordinate [label={above:  {\Large{$x$}}}] (E) at (8.7,3.1) ;

\draw[fill] (2.05,1.23) circle (5pt);
\coordinate [label={left:  {\Large{$y$}}}] (E) at (2.05,1.23) ;

\draw[fill] (1.74,3) circle (5pt);
\coordinate [label={left:  {\Large{$x_2$}}}] (E) at (1.74,3) ;

\draw[->,dashed,red] (1.8,2.8) -- (2.05,1.43);

\begin{axis}[
axis lines=none,
width=\textwidth,
ymax=0.5,
xmax=1.2,
xmin=-0.2,
ymin=-0.5,
height=0.5\textwidth]

{\addplot [thick, black, solid, forget plot]
coordinates {
(0, 0) 
(0.005, 0.081421) 
(0.01, 0.11358) 
(0.015, 0.137578) 
(0.02, 0.157318) 
(0.025, 0.174315) 
(0.03, 0.189343) 
(0.035, 0.202865) 
(0.04, 0.215181) 
(0.045, 0.226502) 
(0.05, 0.236979) 
(0.055, 0.246728) 
(0.06, 0.255839) 
(0.065, 0.264383) 
(0.07, 0.272418) 
(0.075, 0.279993) 
(0.08, 0.287148) 
(0.085, 0.293917) 
(0.09, 0.30033) 
(0.095, 0.306411) 
(0.1, 0.312184) 
(0.105, 0.317667) 
(0.11, 0.322878) 
(0.115, 0.327832) 
(0.12, 0.332544) 
(0.125, 0.337025) 
(0.13, 0.341287) 
(0.135, 0.345341) 
(0.14, 0.349195) 
(0.145, 0.352859) 
(0.15, 0.35634) 
(0.155, 0.359645) 
(0.16, 0.362783) 
(0.165, 0.365757) 
(0.17, 0.368576) 
(0.175, 0.371244) 
(0.18, 0.373766) 
(0.185, 0.376147) 
(0.19, 0.378392) 
(0.195, 0.380505) 
(0.2, 0.382489) 
(0.205, 0.38435) 
(0.21, 0.38609) 
(0.215, 0.387713) 
(0.22, 0.389222) 
(0.225, 0.39062) 
(0.23, 0.391911) 
(0.235, 0.393097) 
(0.24, 0.39418) 
(0.245, 0.395164) 
(0.25, 0.39605) 
(0.255, 0.396842) 
(0.26, 0.397541) 
(0.265, 0.398149) 
(0.27, 0.39867) 
(0.275, 0.399104) 
(0.28, 0.399454) 
(0.285, 0.399721) 
(0.29, 0.399908) 
(0.295, 0.400016) 
(0.3, 0.400047) 
(0.305, 0.400002) 
(0.31, 0.399884) 
(0.315, 0.399693) 
(0.32, 0.399431) 
(0.325, 0.3991) 
(0.33, 0.3987) 
(0.335, 0.398234) 
(0.34, 0.397703) 
(0.345, 0.397107) 
(0.35, 0.396448) 
(0.355, 0.395728) 
(0.36, 0.394947) 
(0.365, 0.394107) 
(0.37, 0.393209) 
(0.375, 0.392253) 
(0.38, 0.391241) 
(0.385, 0.390174) 
(0.39, 0.389053) 
(0.395, 0.387879) 
(0.4, 0.386652) 
(0.405, 0.385374) 
(0.41, 0.384046) 
(0.415, 0.382668) 
(0.42, 0.381242) 
(0.425, 0.379767) 
(0.43, 0.378246) 
(0.435, 0.376678) 
(0.44, 0.375065) 
(0.445, 0.373407) 
(0.45, 0.371705) 
(0.455, 0.369959) 
(0.46, 0.368171) 
(0.465, 0.366341) 
(0.47, 0.36447) 
(0.475, 0.362558) 
(0.48, 0.360606) 
(0.485, 0.358615) 
(0.49, 0.356584) 
(0.495, 0.354516) 
(0.5, 0.35241) 
(0.505, 0.350267) 
(0.51, 0.348087) 
(0.515, 0.345871) 
(0.52, 0.34362) 
(0.525, 0.341334) 
(0.53, 0.339013) 
(0.535, 0.336658) 
(0.54, 0.334269) 
(0.545, 0.331848) 
(0.55, 0.329393) 
(0.555, 0.326906) 
(0.56, 0.324388) 
(0.565, 0.321838) 
(0.57, 0.319256) 
(0.575, 0.316644) 
(0.58, 0.314002) 
(0.585, 0.31133) 
(0.59, 0.308628) 
(0.595, 0.305896) 
(0.6, 0.303136) 
(0.605, 0.300347) 
(0.61, 0.297529) 
(0.615, 0.294684) 
(0.62, 0.29181) 
(0.625, 0.28891) 
(0.63, 0.285981) 
(0.635, 0.283026) 
(0.64, 0.280044) 
(0.645, 0.277035) 
(0.65, 0.274) 
(0.655, 0.270939) 
(0.66, 0.267852) 
(0.665, 0.264739) 
(0.67, 0.261601) 
(0.675, 0.258437) 
(0.68, 0.255248) 
(0.685, 0.252034) 
(0.69, 0.248795) 
(0.695, 0.245532) 
(0.7, 0.242244) 
(0.705, 0.238931) 
(0.71, 0.235594) 
(0.715, 0.232232) 
(0.72, 0.228847) 
(0.725, 0.225437) 
(0.73, 0.222003) 
(0.735, 0.218545) 
(0.74, 0.215064) 
(0.745, 0.211558) 
(0.75, 0.208029) 
(0.755, 0.204476) 
(0.76, 0.2009) 
(0.765, 0.197299) 
(0.77, 0.193675) 
(0.775, 0.190028) 
(0.78, 0.186356) 
(0.785, 0.182661) 
(0.79, 0.178943) 
(0.795, 0.1752) 
(0.8, 0.171434) 
(0.805, 0.167644) 
(0.81, 0.16383) 
(0.815, 0.159992) 
(0.82, 0.156131) 
(0.825, 0.152245) 
(0.83, 0.148335) 
(0.835, 0.144401) 
(0.84, 0.140443) 
(0.845, 0.13646) 
(0.85, 0.132453) 
(0.855, 0.128421) 
(0.86, 0.124364) 
(0.865, 0.120283) 
(0.87, 0.116176) 
(0.875, 0.112045) 
(0.88, 0.107888) 
(0.885, 0.103705) 
(0.89, 0.0994971) 
(0.895, 0.0952632) 
(0.9, 0.0910032) 
(0.905, 0.0867171) 
(0.91, 0.0824047) 
(0.915, 0.0780656) 
(0.92, 0.0736997) 
(0.925, 0.0693068) 
(0.93, 0.0648866) 
(0.935, 0.060439) 
(0.94, 0.0559636) 
(0.945, 0.0514603) 
(0.95, 0.0469288) 
(0.955, 0.0423687) 
(0.96, 0.0377799) 
(0.965, 0.0331621) 
(0.97, 0.028515) 
(0.975, 0.0238383) 
(0.98, 0.0191316) 
(0.985, 0.0143948) 
(0.99, 0.00962748) 
(0.995, 0.00482931) 
(1, -1.11022e-16) 
 };}

{\addplot [thick, black, solid, forget plot]
coordinates {
(0, -0) 
(0.005, -0.081421) 
(0.01, -0.11358) 
(0.015, -0.137578) 
(0.02, -0.157318) 
(0.025, -0.174315) 
(0.03, -0.189343) 
(0.035, -0.202865) 
(0.04, -0.215181) 
(0.045, -0.226502) 
(0.05, -0.236979) 
(0.055, -0.246728) 
(0.06, -0.255839) 
(0.065, -0.264383) 
(0.07, -0.272418) 
(0.075, -0.279993) 
(0.08, -0.287148) 
(0.085, -0.293917) 
(0.09, -0.30033) 
(0.095, -0.306411) 
(0.1, -0.312184) 
(0.105, -0.317667) 
(0.11, -0.322878) 
(0.115, -0.327832) 
(0.12, -0.332544) 
(0.125, -0.337025) 
(0.13, -0.341287) 
(0.135, -0.345341) 
(0.14, -0.349195) 
(0.145, -0.352859) 
(0.15, -0.35634) 
(0.155, -0.359645) 
(0.16, -0.362783) 
(0.165, -0.365757) 
(0.17, -0.368576) 
(0.175, -0.371244) 
(0.18, -0.373766) 
(0.185, -0.376147) 
(0.19, -0.378392) 
(0.195, -0.380505) 
(0.2, -0.382489) 
(0.205, -0.38435) 
(0.21, -0.38609) 
(0.215, -0.387713) 
(0.22, -0.389222) 
(0.225, -0.39062) 
(0.23, -0.391911) 
(0.235, -0.393097) 
(0.24, -0.39418) 
(0.245, -0.395164) 
(0.25, -0.39605) 
(0.255, -0.396842) 
(0.26, -0.397541) 
(0.265, -0.398149) 
(0.27, -0.39867) 
(0.275, -0.399104) 
(0.28, -0.399454) 
(0.285, -0.399721) 
(0.29, -0.399908) 
(0.295, -0.400016) 
(0.3, -0.400047) 
(0.305, -0.400002) 
(0.31, -0.399884) 
(0.315, -0.399693) 
(0.32, -0.399431) 
(0.325, -0.3991) 
(0.33, -0.3987) 
(0.335, -0.398234) 
(0.34, -0.397703) 
(0.345, -0.397107) 
(0.35, -0.396448) 
(0.355, -0.395728) 
(0.36, -0.394947) 
(0.365, -0.394107) 
(0.37, -0.393209) 
(0.375, -0.392253) 
(0.38, -0.391241) 
(0.385, -0.390174) 
(0.39, -0.389053) 
(0.395, -0.387879) 
(0.4, -0.386652) 
(0.405, -0.385374) 
(0.41, -0.384046) 
(0.415, -0.382668) 
(0.42, -0.381242) 
(0.425, -0.379767) 
(0.43, -0.378246) 
(0.435, -0.376678) 
(0.44, -0.375065) 
(0.445, -0.373407) 
(0.45, -0.371705) 
(0.455, -0.369959) 
(0.46, -0.368171) 
(0.465, -0.366341) 
(0.47, -0.36447) 
(0.475, -0.362558) 
(0.48, -0.360606) 
(0.485, -0.358615) 
(0.49, -0.356584) 
(0.495, -0.354516) 
(0.5, -0.35241) 
(0.505, -0.350267) 
(0.51, -0.348087) 
(0.515, -0.345871) 
(0.52, -0.34362) 
(0.525, -0.341334) 
(0.53, -0.339013) 
(0.535, -0.336658) 
(0.54, -0.334269) 
(0.545, -0.331848) 
(0.55, -0.329393) 
(0.555, -0.326906) 
(0.56, -0.324388) 
(0.565, -0.321838) 
(0.57, -0.319256) 
(0.575, -0.316644) 
(0.58, -0.314002) 
(0.585, -0.31133) 
(0.59, -0.308628) 
(0.595, -0.305896) 
(0.6, -0.303136) 
(0.605, -0.300347) 
(0.61, -0.297529) 
(0.615, -0.294684) 
(0.62, -0.29181) 
(0.625, -0.28891) 
(0.63, -0.285981) 
(0.635, -0.283026) 
(0.64, -0.280044) 
(0.645, -0.277035) 
(0.65, -0.274) 
(0.655, -0.270939) 
(0.66, -0.267852) 
(0.665, -0.264739) 
(0.67, -0.261601) 
(0.675, -0.258437) 
(0.68, -0.255248) 
(0.685, -0.252034) 
(0.69, -0.248795) 
(0.695, -0.245532) 
(0.7, -0.242244) 
(0.705, -0.238931) 
(0.71, -0.235594) 
(0.715, -0.232232) 
(0.72, -0.228847) 
(0.725, -0.225437) 
(0.73, -0.222003) 
(0.735, -0.218545) 
(0.74, -0.215064) 
(0.745, -0.211558) 
(0.75, -0.208029) 
(0.755, -0.204476) 
(0.76, -0.2009) 
(0.765, -0.197299) 
(0.77, -0.193675) 
(0.775, -0.190028) 
(0.78, -0.186356) 
(0.785, -0.182661) 
(0.79, -0.178943) 
(0.795, -0.1752) 
(0.8, -0.171434) 
(0.805, -0.167644) 
(0.81, -0.16383) 
(0.815, -0.159992) 
(0.82, -0.156131) 
(0.825, -0.152245) 
(0.83, -0.148335) 
(0.835, -0.144401) 
(0.84, -0.140443) 
(0.845, -0.13646) 
(0.85, -0.132453) 
(0.855, -0.128421) 
(0.86, -0.124364) 
(0.865, -0.120283) 
(0.87, -0.116176) 
(0.875, -0.112045) 
(0.88, -0.107888) 
(0.885, -0.103705) 
(0.89, -0.0994971) 
(0.895, -0.0952632) 
(0.9, -0.0910032) 
(0.905, -0.0867171) 
(0.91, -0.0824047) 
(0.915, -0.0780656) 
(0.92, -0.0736997) 
(0.925, -0.0693068) 
(0.93, -0.0648866) 
(0.935, -0.060439) 
(0.94, -0.0559636) 
(0.945, -0.0514603) 
(0.95, -0.0469288) 
(0.955, -0.0423687) 
(0.96, -0.0377799) 
(0.965, -0.0331621) 
(0.97, -0.028515) 
(0.975, -0.0238383) 
(0.98, -0.0191316) 
(0.985, -0.0143948) 
(0.99, -0.00962748) 
(0.995, -0.00482931) 
(1, 0) 
};}

{\addplot [ultra thick, red, dashed, forget plot]
coordinates {
(0.1, 0.312184) 
(1, 0) 
(0.1, -0.312184) 
(0.1, 0.312184) 
};}

\end{axis}
\end{tikzpicture}
 }
\caption{approximation of large deformations over curved boundaries using a multi-layer ($\ell=3$) compositional map.}
\label{fig:multilayer_map}
\end{figure}
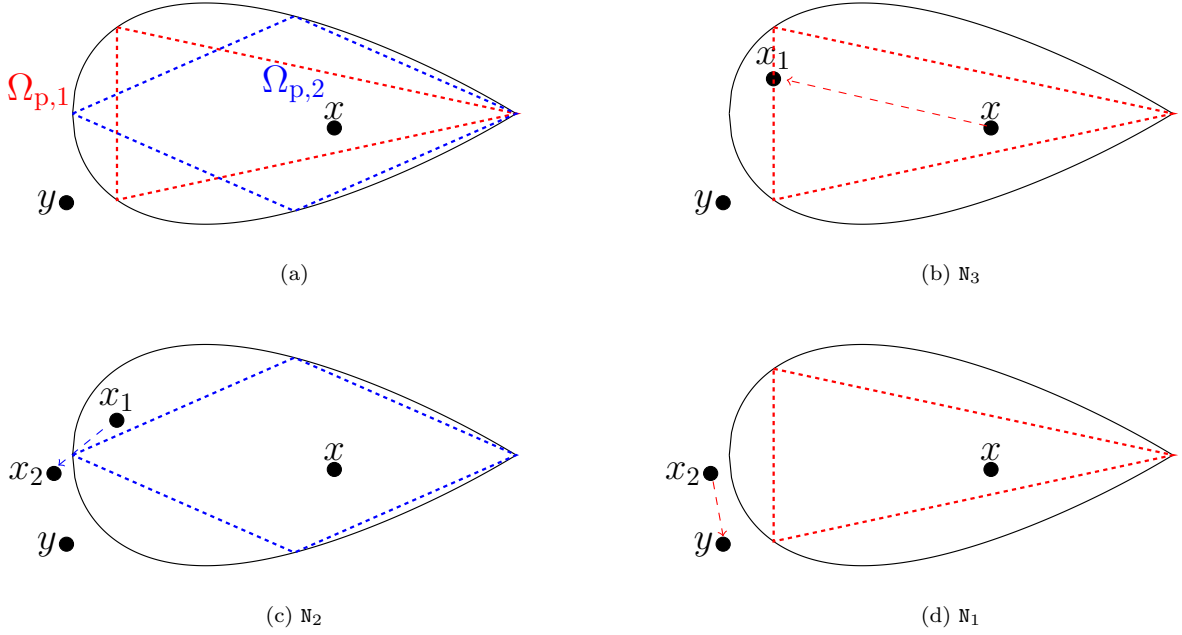

\subsection{Discussion}
\label{sec:discussion_curved}

The analysis of this section shows that compositional maps \eqref{eq:nonlinear_ansatz} and their multi-layer generalization \eqref{eq:nonlinear_ansatz_generalized} can be employed for registration in curved domains.
Since $\Psi$ in \eqref{eq:nonlinear_ansatz} and
$\Psi_1,\ldots,\Psi_{\ell}$ in \eqref{eq:nonlinear_ansatz_generalized} are independent of the coefficients $\mathbf{a}$, 
we can apply the penalty term \eqref{eq:f_pen_theory} to the polytopal map  $\texttt{N}_{\rm p}(\mathbf{a})$ in \eqref{eq:nonlinear_ansatz}, or
$\{ \texttt{N}_{{\rm p},i}(\mathbf{a}_i  \}_{i=1}^{\ell}$  in \eqref{eq:nonlinear_ansatz_generalized}; then,
we can exploit Proposition \ref{th:bijectivity} (see also the discussion in section 
 \ref{sec:ansatz_explained_polytopes}) to prove that 
\eqref{eq:nonlinear_ansatz} and  \eqref{eq:nonlinear_ansatz_generalized} satisfy \eqref{eq:desiderata_mapping_space}.

Lemmas \ref{th:density_composition} and \ref{th:density_multilayer_composition} investigate the approximation  properties of the ansätze 
\eqref{eq:nonlinear_ansatz} and  \eqref{eq:nonlinear_ansatz_generalized}: the analysis shows that multi-layer maps provide much more approximation power, even for moderate number of layers $\ell$.
We note, however, that multi-layer maps are considerably more challenging to implement and might also be significantly more expensive to evaluate: the solution to \eqref{eq:tractable_optimization_based_registration}   requires indeed many evaluations of the mapping $\texttt{N}(\mathbf{a})$ and its gradient, which involve the evaluation of the bijections $\{  \Psi_i \}_{i=1}^{\ell}$ and their inverses. In this work, we focus on the implementation of registration strategies based on the more elementary ansatz \eqref{eq:nonlinear_ansatz} and we refer to a future work for the implementation of registration methods based on the model class \eqref{eq:nonlinear_ansatz_generalized}.

\section{Methodology}
\label{sec:methods}
Given the domain $\Omega\subset \mathbb{R}^2$, we present a registration procedure based on  compositional mappings of the form \eqref{eq:nonlinear_ansatz}.
We recall (cf. \eqref{eq:tractable_optimization_based_registration})  that we consider the minimization problem
$$
\min_{   \mathbf{a}  \in \mathbb{R}^M}
 \mathfrak{f}_{\mu}^{\rm tg}(\texttt{N}( \mathbf{a} ))
\; +\xi  \; 
 \mathfrak{f}_{\rm pen}(\mathbf{a} ),
$$
where $\texttt{N}(\mathbf{a}) = \Psi \circ  \texttt{N}_{\rm p}(\mathbf{a}) \circ \Psi^{-1}$,
  $\Psi: \Omega_{\rm p} \to \Omega$ is a bijection from the polytope $\Omega_{\rm p}$ to $\Omega$ and 
$\texttt{N}_{\rm p}(\mathbf{a}): \Omega_{\rm p} \to    \Omega_{\rm p}$  satisfies
 $\texttt{N}_{\rm p}( \mathbf{a}) = 
 \texttt{id} + \sum_{i=1}^M (\mathbf{a})_i \varphi_i$  with
 $\varphi_i \cdot \mathbf{n} |_{\partial \Omega_{\rm p}} = 0$ for $i=1,\ldots,M$.
 In section \ref{sec:nonlinear_ansatz}, we discuss the construction of the polytope $\Omega_{\rm p}$ and  the mapping $\Psi$ in \eqref{eq:nonlinear_ansatz} based on a curved mesh of $\Omega$; 
in section \ref{sec:FEM}, 
we    introduce the 
FE space $\mathcal{U} = {\rm span} \{ \varphi_i \}_i$ for the displacement field in the polytope $\Omega_{\rm p}$ and we discuss the choice of the functional norm $\| \cdot \|$ for $\mathcal{U}$ that enters in the penalty $ \mathfrak{f}_{\rm pen}$; then, 
in section  \ref{sec:curved_mesh}, we discuss the definition of the curved mesh of $\Omega$.
In sections \ref{sec:penalty_discrete}, \ref{sec:target_function},
 and \ref{sec:parametric_map}, we review the choice of the penalty
$ \mathfrak{f}_{\rm pen}$ 
  and  target functions $ \mathfrak{f}_{\mu}^{\rm tg}$  
  proposed in \cite{taddei2020registration,taddei2021registration}
 and the extension to  parametric  problems.
 To simplify the presentation, we here assume that the domain $\Omega$ is parameter-independent.
  
We recall that our ultimate goal is to devise Lagrangian approximations of parametric fields. Given the parametric field $u:\Omega \times \mathcal{P} \to \mathbb{R}^{D_{\rm u}}$ , Lagrangian approximations read as
\begin{equation}
\label{eq:Lagrangian_approximations}
u_{\mu} \approx \widehat{u}_{\mu}:=\widetilde{u}_{\mu} \circ \Phi_{\mu}^{-1},
\end{equation}
where $\widetilde{u}_{\mu}$ is a low-rank approximation of the mapped field 
${u}_{\mu} \circ \Phi_{\mu}$. If we denote by $\mathcal{T}_{\rm pb}$ the high-fidelity (HF) mesh used to approximate the mapped field, 
it might be important 
(see, e.g., \cite{barral2023registration})
to ensure that the deformed mesh
$\Phi_{\mu}(\mathcal{T}_{\rm pb})$, which shares the  connectivity matrix with 
$\mathcal{T}_{\rm pb}$ but has deformed nodes according to $\Phi_{\mu}$,
 is well-behaved for all $\mu\in \mathcal{P}$: in section \ref{sec:penalty_discrete}, we discuss how to enforce this constraint in the registration procedure.
  
\subsection{Definition of the Lipschitz map $\Psi$ based on a curved mesh of $\Omega$}
\label{sec:nonlinear_ansatz}
Our point of departure is the definition of a   triangular
curved mesh $\mathcal{T}$ of degree $\kappa>0$ of the domain  
$\Omega$: the mesh $\mathcal{T}$ is uniquely identified by the mesh nodes 
and the connectivity matrix.
We define the reference (or master) element 
$\widehat{\texttt{D}} = \{ x\in (0,1)^2: \sum_{i=1}^2 (x)_i < 1 \}$ with 
 reference vertices   and nodes 
$\{ \tilde{x}_{i}^{\rm p} \}_{i=1}^{3} \subset 
\{ \tilde{x}_{i} \}_{i=1}^{n_{\rm lp}} \subset \overline{\widehat{\texttt{D}}}$, 
 and the associated Lagrangian bases 
 $\{ \ell_{i}^{\rm fe,p} \}_{i=1}^{3}$ and  $\{ \ell_{i}^{\rm fe} \}_{i=1}^{n_{\rm lp}}$  
 for polynomials of degree one,  $\mathbb{P}_{1}$, 
  and degree $\kappa$, $\mathbb{P}_{\kappa}$. 
We denote by $ \{  \texttt{D}_k \}_{k=1}^{N_{\rm e}}$ the elements of the mesh, by  $ \{  \texttt{F}_j \}_{j=1}^{N_{\rm f}}$ the facets of the mesh;
for each element $\texttt{D}_k$ of the mesh $\mathcal{T}$,
we  define the nodes 
$\{ x_{i,k}^{\rm hf} \}_{i,k}$ such that
$x_{i,k}^{\rm hf}$ is 
 the $i$-th node of the $k$-th element of  the mesh $\mathcal{T}$ for $i=1,\ldots,n_{\rm lp}$ and $k=1,\ldots,N_{\rm e}$; similarly, 
 we define the  vertices $\{  x_{j,k}^{\rm hf,v}   \}_{j,k}$ of the curved triangles for $j=1,\ldots,3$ and $k=1,\ldots,N_{\rm e}$.

We have now the elements to define the polytope $\Omega_{\rm p}$ and the mapping  $\Psi$.
We introduce 
the FE mappings
\begin{subequations}
\label{eq:geometric_mapping}
\begin{equation}
\label{eq:geometric_mapping_elemental_mapping}
\Psi_k(\tilde{x})
=
\sum_{i=1}^{n_{\rm lp}} 
x_{i,k}^{\rm hf} \ell_i^{\rm fe}(\tilde{x}),
\quad
\Psi_{k,\rm p}(\tilde{x})
=
\sum_{i=1}^{3} 
x_{i,k}^{\rm hf,v} \ell_i^{\rm fe,p}(\tilde{x}),
\quad
k=1,\ldots,N_{\rm e};
\end{equation}
for each $k=1,\ldots,N_{\rm e}$, $\Psi_k$ maps the reference element into the $k$-th element of the mesh, 
while $\Psi_{k,\rm p}$ is a linear map whose image defines the $k$-th linearized element of the mesh,
$\texttt{D}_{k,\rm p}:= \Psi_{k,\rm p}(\widehat{\texttt{D}})$.
Then, we define the polytope $\Omega_{\rm p}$ such that
\begin{equation}
\label{eq:geometric_mapping_polytope_coarsegrained}
\overline{\Omega}_{\rm p} 
:=\bigcup_{i=1}^{N_{\rm e}}
\overline{\texttt{D}}_{k,\rm p},
\end{equation} 
the linear triangular mesh 
$\mathcal{T}_{\rm p}$ with elements 
$\{  \texttt{D}_{k,\rm p}  = \Psi^{-1}(\texttt{D}_k) \}_{k=1}^{N_{\rm e}}$
  and   facets 
  $\{ \texttt{F}_{j,\rm p} = \Psi^{-1}(   \texttt{F}_{j}  )  \}_{j=1}^{N_{\rm f}}$, 
and the FE space
  \begin{equation}
  \label{eq:polynomial_space_fem}
    \mathcal{X}_{\rm hf,p} = \left\{
\phi \in C(\Omega_{\rm p}) : \phi|_{\texttt{D}_{k,\rm p}} \in \mathbb{P}_{\kappa}
\right\}.  
  \end{equation}
  Finally, we introduce the  field  $\Psi$ from $\Omega_{\rm p}$  to $\Omega$ as
\begin{equation}
\label{eq:geometric_mapping_Psi}
\Psi:\Omega_{\rm p} \to \Omega \quad {\rm s.t.} \quad
\Psi(x) \big|_{\texttt{D}_{k,\rm p}}
=
\Psi_k \left(
\Psi_{k,\rm p}^{-1}(x)
\right).
\end{equation}
\end{subequations}

Since 
$\Psi_{k,\rm p}$ is a linear map, 
$\Psi\in \mathcal{X}_{\rm hf,p}$.
If the local FE mappings \eqref{eq:geometric_mapping_elemental_mapping} are invertible with positive Jacobian determinant, we find that 
the restriction of $\Psi$ to 
$\texttt{D}_{k,\rm p}$ is a bijection from 
$\texttt{D}_{k,\rm p}$ to 
$\texttt{D}_{k}$ with positive Jacobian determinant (composition of bijections 
with positive Jacobian determinant), for $k=1,\ldots,N_{\rm e}$: therefore, 
$\Psi$ is a bijection from $\Omega_{\rm p}$ to $\Omega$.
We further observe that
by construction 
 the restriction of $\Psi$ to the vertices of $\mathcal{T}$
and to linear  facets of the mesh
  is the identity map: therefore, the pair $(\Omega_{\rm p},\Psi)$ satisfies Hypothesis \ref{hyp:regular_polytopes}.

Figure
\ref{fig:simple_example_curved_mesh}
 provides a graphical interpretation of our construction for a curved mesh with two elements.
Figure
\ref{fig:simple_example_curved_mesh}(a)  shows a curved mesh of the domain $\Omega$, while 
 Figure
\ref{fig:simple_example_curved_mesh}(b)  shows the corresponding linearized mesh.
The mapping $\Psi$ is the unique polynomial map of degree $\kappa$ that maps each element  $  \texttt{D}_{k,\rm p}  $ into the corresponding curved element $\texttt{D}_{k}$.
Notice that by construction the vertices of the mesh $\mathcal{T}$ are unchanged in 
$\mathcal{T}_{\rm p}$; notice also that linear elements of $\mathcal{T}$ are preserved by $\mathcal{T}_{\rm p}$.

\begin{figure}[h!]
\centering
\subfloat[]{
 \begin{tikzpicture}[scale=1.75]
\linethickness{0.3 mm}

\draw[ultra thick]  (-1.5,0)--(-1.5,0.8)--(1.5,0.8)--(1.5,0);

\draw[ultra thick]  (-1.5,0)--(1.5,0.8);

\draw[smooth, domain = 0:1.5, color=blue, very thick] plot (\x,{0.0625*exp(-25*\x^2)});

\draw[smooth, domain = 0:1.5, color=blue, very thick] plot (-\x,{0.0625*exp(-25*\x^2)});

\coordinate [label={above:  {\LARGE {$\Omega$}}}] (E) at (0.8, 0) ;

\coordinate [label={right:  {\LARGE {$\mathcal{T}$}}}] (E) at (1.5, 0.6) ;

\coordinate [label={right:  {\LARGE {$\texttt{D}_{1}$}}}] (E) at (-1.5, 0.3) ;

\coordinate [label={left:  {\LARGE {$\texttt{D}_{2}$}}}] (E) at (1.5, 0.3) ;
%
%
%

\end{tikzpicture}
}
~~
\subfloat[]{
\begin{tikzpicture}[scale=1.75]
\linethickness{0.3 mm}

\draw[ultra thick]  (-1.5,0)--(-1.5,0.8)--(1.5,0.8)--(1.5,0);

\draw[ultra thick]  (-1.5,0)--(1.5,0.8);

\draw[smooth, domain = 0:1.5, color=black, ultra thick] plot (\x,{0.*exp(-25*\x^2)});

\draw[smooth, domain = 0:1.5, color=black, ultra thick] plot (-\x,{0.0*exp(-25*\x^2)});

\coordinate [label={above:  {\LARGE {$\Omega_{\rm p}$}}}] (E) at (0.8, 0) ;

\coordinate [label={right:  {\LARGE {$\mathcal{T}_{\rm p}$}}}] (E) at (1.5, 0.6) ;

\coordinate [label={right:  {\LARGE {$\texttt{D}_{1,\rm p}$}}}] (E) at (-1.5, 0.3) ;

\coordinate [label={left:  {\LARGE {$\texttt{D}_{2,\rm p}$}}}] (E) at (1.5, 0.3) ;
%
%
%

\end{tikzpicture}
}
\caption{construction of the polytope $\Omega_{\rm p}$ and the mapping $\Psi$ for a curved mesh with two elements.}
\label{fig:simple_example_curved_mesh}
\end{figure}
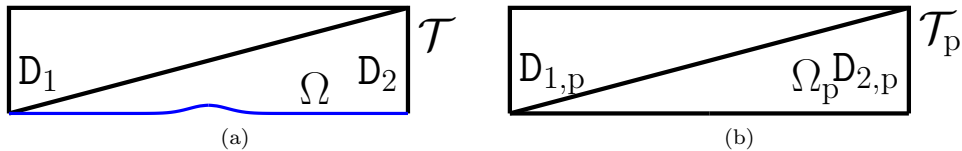 

The construction of $\Psi$ outlined above might fail for excessively coarse meshes, as shown in the example of   Figure \ref{fig:inadmissible_mesh}.
Given the domain $\Omega$ depicted in Figure \ref{fig:inadmissible_mesh}(a), we might be tempted to consider the coarse mesh in  Figure \ref{fig:inadmissible_mesh}(b): even if the curved mesh $\mathcal{T}$ associated with this partition is a proper mesh of $\Omega$, we find that the resulting polytope $\Omega_{\rm p}$ is not isomorphic to $\Omega$  up to the boundary (since $\overline{\Omega}_{\rm p}$ is a rectangle) and $\Psi$ is not continuous up to the boundary of ${\Omega}_{\rm p}$.
To fix this issue for this geometry, we should consider meshes with at least three points on the profile (cf. Figure \ref{fig:inadmissible_mesh}(c)).

We here rely on polynomial discretizations of very high-order (up to $\kappa = 10$): to ensure accurate and stable computations, we rely on a nodal-based  discretization that exploits Koornwinder polynomials to represent the local shape functions and to a tensor product of Gauss and Gauss–Radau quadratures (see, e.g., 
\cite{hesthaven2007nodal,karniadakis2005spectral}).
Instead, we consider regular (tensorized) nodes 
$\{ \tilde{x}_{i} \}_{i=1}^{n_{\rm lp}} \subset \overline{\widehat{\texttt{D}}}$;
alternative selections that improve the conditioning of the interpolation matrix can be found in
\cite{hesthaven2007nodal,karniadakis2005spectral}, 
see also 
\cite{hesthaven1998electrostatics}.

 In our framework,  the use of high-order discretizations is motivated by two independent considerations.
First,
coarse-grained meshes with relatively few elements enable rapid searches over the elements and hence guarantee rapid evaluations of the mapping $\Psi$ for new points in $\Omega_{\rm p}$ --- this feature is crucial in optimization-based registration as iterative algorithms for \eqref{eq:tractable_optimization_based_registration} require multiple evaluations of \eqref{eq:nonlinear_ansatz} for many values of $\mathbf{a}$. 
Second, since by construction
$\texttt{N}(x^{\rm v};  \mathbf{a}) = x^{\rm v}$ for all vertices $x^{\rm v}$ of $\Omega_{\rm p}$ and any $\mathbf{a}\in \mathbb{R}^M$ (cf. section \ref{sec:density_composition}), 
the   reduction of the  number of vertices on curved facets of $\partial \Omega$ 
 improves the expressiveness of the ansatz \eqref{eq:nonlinear_ansatz}.

\begin{figure}[h!]
\centering
\subfloat[]{
\includegraphics[width=.3\textwidth]{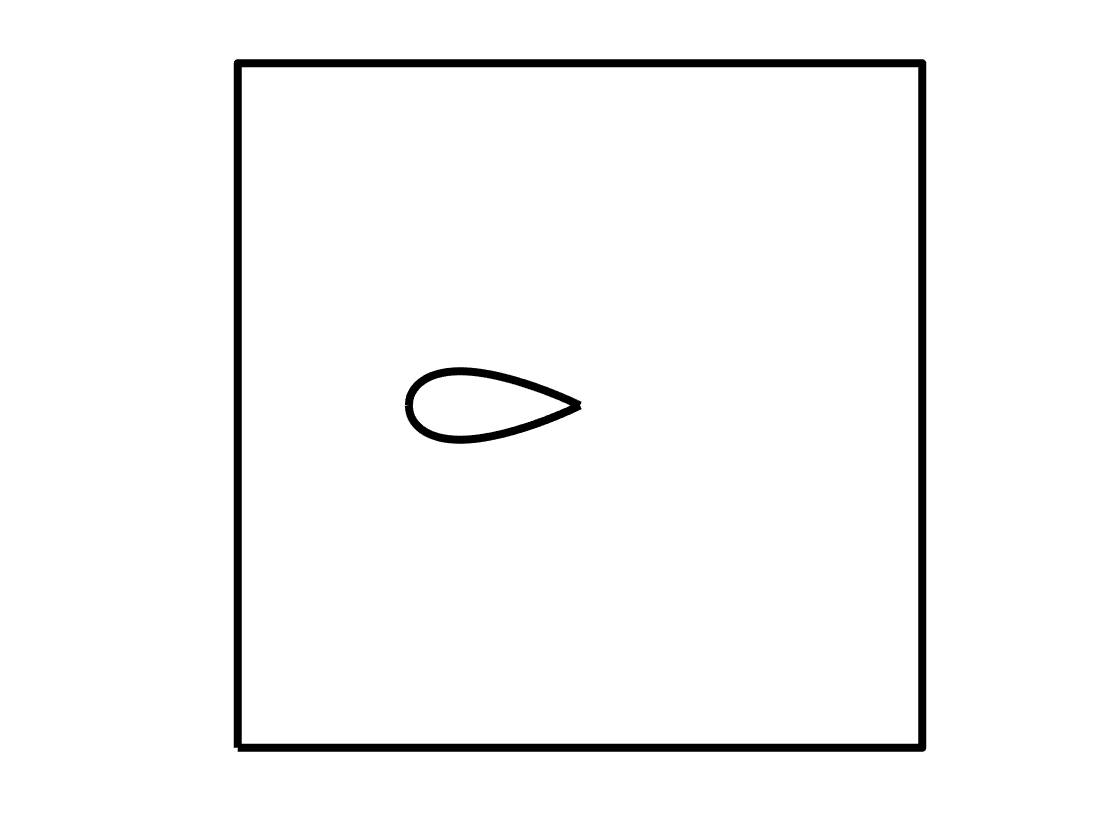}
}
\subfloat[]{ 
\includegraphics[width=.3\textwidth]{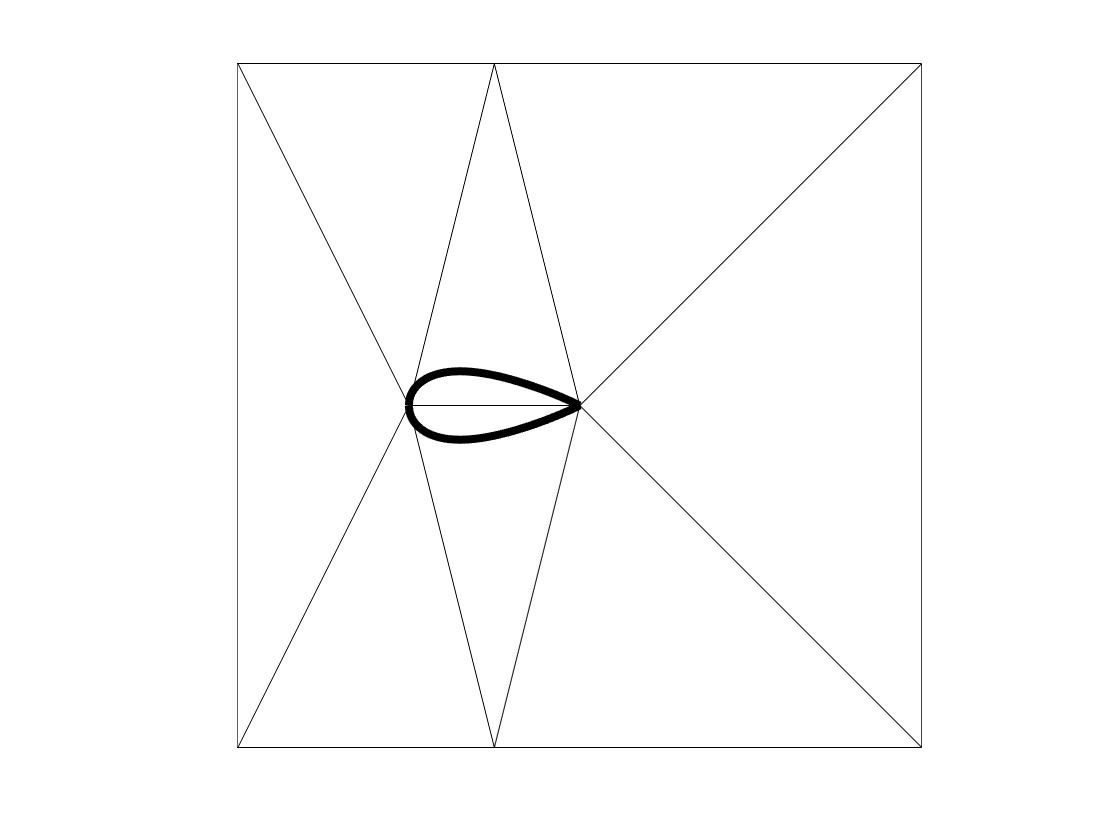}
}
~~~
\subfloat[]{
\includegraphics[width=.3\textwidth]{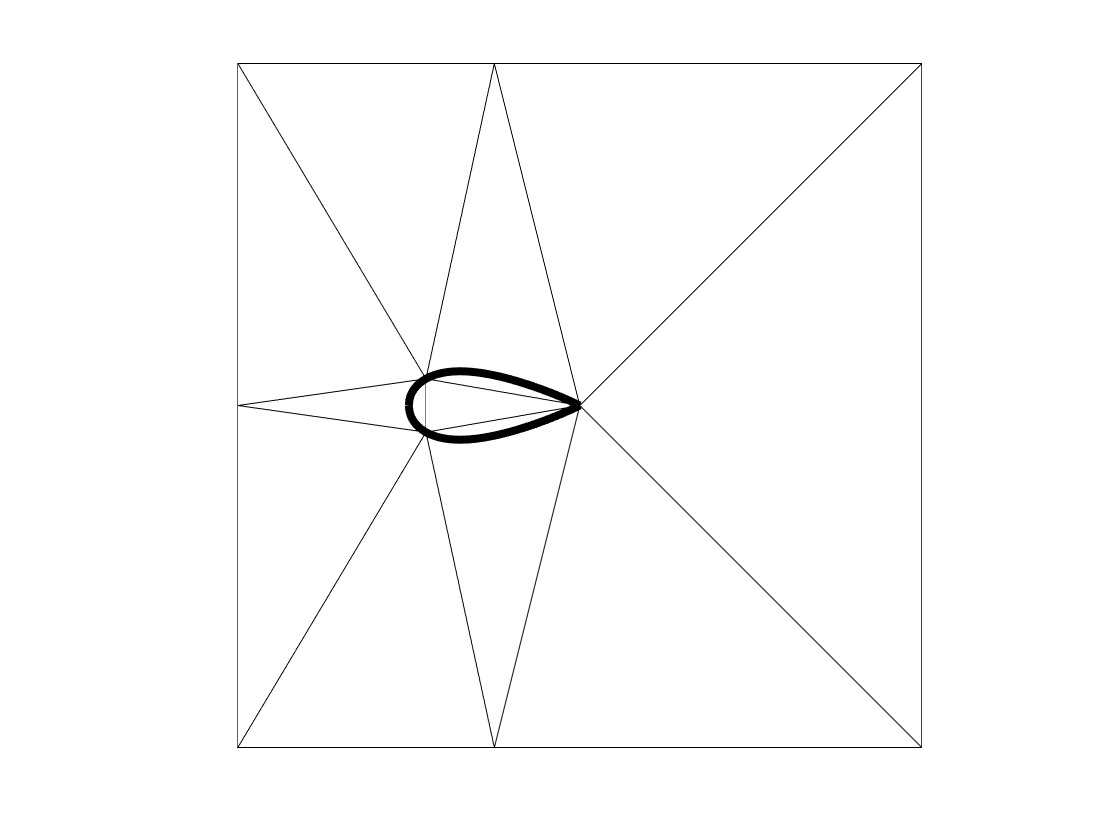}
}
\caption{isolated profile in rectangular domain.
(a) domain $\Omega$.
(b) inadmissible coarse-grained partition.
(c) admissible coarse-grained partition.
}
\label{fig:inadmissible_mesh}
\end{figure} 
 
 \subsection{Finite element displacement space on polytopes}
\label{sec:FEM}
Given the polytope $\Omega_{\rm p} \subset \mathbb{R}^2$
and  the linear triangular mesh $\mathcal{T}_{\rm p}$ defined in section \ref{sec:nonlinear_ansatz},
we  denote by 
$\texttt{I}_{\rm int} \subset \{1,\ldots,N_{\rm f}\}$ the indices of the internal facets and we define  the union of internal facets
$\mathcal{E}_{\rm p}^{\rm int} = \bigcup_{j\in \texttt{I}_{\rm int} } \texttt{F}_{j,\rm p}$.
We further introduce the 
 positive\footnote{The positive normal is chosen arbitrarily for internal facets and coincides with the outward normal to 
$\Omega_{\rm p}$ for boundary facets.} normal $\mathbf{n}^+: \mathcal{E}_{\rm p}^{\rm int}  \to \mathbb{S}^1$.
Given the field $\phi\in \mathcal{X}_{\rm hf,p}$ and  an arbitrary (possibly differential) operator $\mathfrak{L}$  (e.g.,
$\mathfrak{L}(\phi) =\phi$,
 $\mathfrak{L}(\phi) = \nabla \phi$, or 
$\mathfrak{L}(\phi) = H(\phi)$), 
 we define the limits
\begin{subequations}
\label{eq:avg_jump_operators}
\begin{equation}
\mathfrak{L}(\phi)^{\pm}(x)
=
\lim_{\epsilon\to 0^+} \mathfrak{L}(\phi)(x  \mp \epsilon \mathbf{n}^+(x)),
\quad
\forall \, x \in \mathcal{E}_{\rm p}^{\rm int},
\end{equation}
and the average and jump operators
(we omit dependence on $x$)
\begin{equation}
\left\{ \mathfrak{L}(\phi) \right\} 
=\frac{\mathfrak{L}(\phi)^+ + \mathfrak{L}(\phi)^-}{2},
\quad
\llbracket   \mathfrak{L}(\phi)   \rrbracket  
= \left( \mathfrak{L}(\phi)^+  - \mathfrak{L}(\phi)^-  \right) \cdot  \mathbf{n}^+,
\end{equation}
for all $x \in \mathcal{E}_{\rm p}^{\rm int}$.
We clarify that if $\mathfrak{L}(\phi)$ is a matrix-valued quantity the jump operator involves a matrix-vector product.
\end{subequations}

We introduce the space $\mathcal{U}$ for the displacement field (cf.  section \ref{sec:affine_maps_polytopes}):
\begin{equation}
\label{eq:mapping_space}
\mathcal{U} = 
\left\{
\varphi \in [\mathcal{X}_{\rm hf,p} ]^2 : \varphi \cdot \mathbf{n} \big|_{\partial \Omega_{\rm p}} = 0
\right\}.
\end{equation}
We observe that
$\mathcal{U}$ 
\eqref{eq:mapping_space}
  is contained in the space of Lipschitz functions in $\Omega_{\rm p}$ but it is not contained in $C^1(\Omega_{\rm p})$.

In view of the definition of the penalty function in \eqref{eq:tractable_optimization_based_registration} in section \ref{sec:penalty_discrete}, we define discrete counterparts of the $H^2$ norm and seminorm.
Since the FE space $\mathcal{X}_{\rm hf,p}$  is not $H^2$-conforming,   the discrete approximation of the $H^2$ inner product,
should involve a direct 
control of the jump of the mapping derivatives across elements.
We hence equip $\mathcal{U}$ with the inner product
\begin{equation}
\label{eq:mapping_inner_product}
\begin{array}{rl}
\displaystyle{
(w,v) =}
&
\displaystyle{\sum_{k=1}^{N_{\rm e}}
 \int_{ \texttt{D}_{k,\rm p}   } \left( H(w) : H(v) \, + \, w \cdot v \right) \, dx
}
\\[3mm]
+ &
\displaystyle{\sum_{j \in \texttt{I}_{\rm int}}  \int_{ \texttt{F}_{j,\rm p}   }
\left(
\beta_j 
\llbracket   \nabla w  \rrbracket \cdot 
\llbracket   \nabla v  \rrbracket
+ 
\frac{1}{\beta_j}
\{  H(w)  \} : \{  H(v)  \}  \right) \, ds,
}
\\
\end{array}
\end{equation}
where  
$H(w) : H(v) = \sum_{i,j,k=1}^2 \partial_{i,j} w_k \partial_{i,j} v_k$,
and
$\beta_j = \sigma_{\beta} \kappa^2 | \texttt{F}_{j,\rm p}   |^{-1}$ with 
$ \sigma_{\beta} =10$ and $j\in \texttt{I}_{\rm int}$;
we further define the induced norm
$\|\cdot \| = \sqrt{(\cdot,\cdot)}$.
The choice of the inner product is inspired by the analysis in 
\cite[section 4]{mozolevski2007hp} for the biharmonic equation: 
the authors of \cite{mozolevski2007hp} consider the norm 
$$
\begin{array}{rl}
\vertiii{w}^2 \, = &
\displaystyle{\Big(
\sum_{k=1}^{N_{\rm e}} \| \Delta w   \|_{L^2( \texttt{D}_{k,\rm p}   )}^2
\, + \,
\sum_{j \in \texttt{I}_{\rm int}   }
 \|  \sqrt{\alpha}  \llbracket   w \rrbracket   \|_{L^2( \texttt{F}_{j,\rm p}   )}^2
}
\\[3mm]
+ &
\displaystyle{
\|  \sqrt{\beta}  \llbracket  \nabla  w \rrbracket   \|_{L^2( \texttt{F}_{j,\rm p}   )}^2 +
\|  1/\sqrt{\beta} \{   \Delta  w  \}   \|_{L^2( \texttt{F}_{j,\rm p}   )}^2 +
\|  1/\sqrt{\alpha} \{  \nabla  \Delta  w  \}   \|_{L^2( \texttt{F}_{j,\rm p}   )}^2 \Big),
}
\\
\end{array}
$$
for proper choices of $\alpha$ and $\beta$.
Since we here consider continuous discretizations, we omit in \eqref{eq:mapping_inner_product} the first and fourth terms in the facet integral and we consider the same expression for $\beta$ as in \cite{mozolevski2007hp}; 
furthermore, since we wish to consider  the full $H^2$ norm, we replace $\Delta w$ with $H(w)$ in the volumetric and surface integral.

We also define the seminorm
\begin{equation}
\label{eq:penalty_term}
\begin{array}{rl}
\mathfrak{P} \left( \varphi   \right)
=
&
\displaystyle{
\sum_{k=1}^{N_{\rm e}}
 \int_{ \texttt{D}_{k,\rm p}   } \left( H( \varphi) : H( \varphi)   \right) \, dx
}
\\[3mm]
\,+\,
&
\displaystyle{
\sum_{j \in \texttt{I}_{\rm int}}  \int_{ \texttt{F}_{j,\rm p}   }
\left(
\beta_j 
\llbracket   \nabla  \varphi  \rrbracket \cdot 
\llbracket   \nabla   \varphi  \rrbracket
+ 
\frac{1}{\beta_j}
\{  H( \varphi)  \} : \{  H( \varphi)  \}  \right) \, ds.
}
\\
\end{array}
\end{equation}
  Note that  $\mathfrak{P} \left( \Phi   \right) = 0$ for any linear map $\Phi = \mathbf{b}+ \mathbf{A} x$;
in particular,  
  $\mathfrak{P} \left( \texttt{id}  \right) = 0$.
  We also define the penalty
  \begin{equation}
\label{eq:penalty_term_broken}
\mathfrak{P}_{\rm brkn} \left( \varphi   \right)
=
\sum_{k=1}^{N_{\rm e}}
 \int_{ \texttt{D}_{k,\rm p}   } \left( H( \varphi) : H( \varphi)   \right) \, dx,
\end{equation}
which does not penalize discontinuities of the displacement derivative at elements' interfaces.

 \subsection{Construction of the curved mesh of $\Omega$}
 \label{sec:curved_mesh}
 The   polytope $\Omega_{\rm p}$   and  the map $\Psi$ are explicitly linked to  the high-order mesh $\mathcal{T}$: 
the problem of determining $\Omega_{\rm p}$   and  the map $\Psi$ can be recast as   the problem of generating a (coarse) high-order mesh $\mathcal{T}$ of the domain of interest.
We propose to 
first  prescribe the 
polytope $\Omega_{\rm p}$,
then determine a 
linear triangulation $\mathcal{T}_{\rm p}$ of $\Omega_{\rm p}$ and  finally   compute the map $\Psi$ by solving 
a suitable optimization problem.
We introduce 
the parameterizations 
$\{\gamma_{k}:(0,1)\to \texttt{F}_k \}_{k=1}^{N_{\rm b}}$
of the curved edges
$\{ \texttt{F}_k \}_{k=1}^{N_{\rm b}}$
 and the Gauss-Lobatto points $\{t_i^{\rm gl}\}_{i=1}^{\kappa+1}\subset [0,1]$. Then, we introduce the 
reference and deformed points
$$
x_{j_{i,k}}
=
\gamma_{k}(0) (1 - t_i^{\rm gl})
+
\gamma_{k}(1) \, t_i^{\rm gl},
\quad
y_{j_{i,k}}
=
\gamma_{k}(t_i^{\rm gl}),
\quad
{\rm where} \;\; 
 j_{i,k}=i +(k-1)N_{\rm b},
$$
for  $i=1,\ldots,\kappa+1$, and
$k=1,\ldots,N_{\rm b}$, and the affine space
$\widetilde{\mathcal{W}}$ with 
\begin{equation}\label{eq:displacement_geometric_registration}
\widetilde{\mathcal{W}}
=\left\{
\texttt{id} +\varphi  \in [\mathcal{X}_{\rm hf,p} ]^2 : \varphi \cdot \mathbf{n} \big|_{\partial \Omega_{\rm p} \cap \partial \Omega} = 0
\right\};
\end{equation}
the set $\partial \Omega_{\rm p} \cap \partial \Omega$ 
corresponds to the portion of 
$\partial \Omega$ that is the union of linear facets and that is hence shared with 
$\partial \Omega_{\rm p}$ (see Hypothesis \ref{hyp:regular_polytopes}).
Note  that the affine space $\widetilde{\mathcal{W}}$ enables deformation of curved facets.

 In conclusion, we consider the
constrained optimization statement:
\begin{equation}
\label{eq:optimization_georeg}
\min_{\Psi \in    \widetilde{\mathcal{W}} } \;\;
\mathfrak{f}_{\rm jac}(\Psi)
\, + \,
 \mathfrak{P}_{\rm brkn} \left( \Psi   \right)
\quad
{\rm s.t} \;\; \|  \Psi(x_j) - y_j \|_{\infty} \leq \delta, \;\; j=1,\ldots,N=(\kappa+1) N_{\rm b},
\end{equation}
where $\mathfrak{f}_{\rm jac}$  is introduced in \eqref{eq:f_jac_tmp} and is 
designed to enforce the bijectivity of the mapping,  while 
$\mathfrak{P}_{\rm brkn} $ is the quadratic penalty term \eqref{eq:penalty_term_broken}. The penalty \eqref{eq:penalty_term_broken} was found superior to 
 \eqref{eq:penalty_term} 
 ---
 that is,  it led to more regular meshes and to more rapid convergence of the optimizer ---
 for the model problem of section \ref{sec:numerics}.

The optimization statement 
\eqref{eq:optimization_georeg} reads as a  nonlinear non-convex optimization problem with linear inequality constraints; in our numerical implementation, we resort to the Matlab function \texttt{fmincon} to solve \eqref{eq:optimization_georeg} based on an interior-point method.
 In the numerical experiments, we consider $\delta=10^{-6}$. 

\subsection{Penalty function}
\label{sec:penalty_discrete}
We denote by $\mathcal{T}_{\rm pb}$ the mesh used for HF calculations:
  $\mathcal{T}_{\rm pb}$ is  independent of the FE mesh 
$\mathcal{T}_{\rm p}$
used in sections
 \ref{sec:nonlinear_ansatz} and \ref{sec:FEM} 
to define the mapping space; it  might also be associated to a different   (spectral element, finite volume) discretization method.
For model reduction tasks, the registration procedure should ensure that the 
deformed mesh $\Phi(\mathcal{T}_{\rm pb})$  is well-behaved (see, e.g.,  \cite{barral2023registration}).
The penalty $\mathfrak{f}_{\rm pen}^{\rm th}$   \eqref{eq:f_pen_theory} 
is not tailored to the discrete representation of the FE fields 
 $\{ \varphi_i \}_{i=1}^M$ that are not $H^2$-conforming; furthermore, the penalty does not ensure that the deformed mesh $\Phi(\mathcal{T}_{\rm pb})$  is well-behaved. Below, we hence derive an alternative penalty function that is well-suited for model reduction tasks.
We denote by 
$\{ \texttt{D}_k^{\rm pb}  \}_{k=1}^{N_{\rm e}^{\rm pb}}$
and 
$\{ \Psi_k^{\rm pb}  \}_{k=1}^{N_{\rm e}^{\rm pb}}$
 the elements  and the
  elemental mappings 
of 
$\mathcal{T}_{\rm pb}$; 
we use notation 
$\{ \Psi_{\Phi,k}^{\rm pb}  \}_{k=1}^{N_{\rm e}^{\rm pb}  }$ to refer to the elemental mappings of the deformed mesh
(see, e.g.,
\cite[Eq. (9)]{taddei2021registration}).

We hence propose the penalty:
\begin{subequations}
\label{eq:penalty_function}
\begin{equation}
\label{eq:penalty_function_a}
\mathfrak{f}_{\rm pen}(\mathbf{a} )
=
\mathfrak{f}_{\rm jac} (  \texttt{N}_{\rm p}(\mathbf{a} )    )
+
\mathfrak{f}_{\rm msh} (\texttt{N}(\mathbf{a} )   )
+
\mathfrak{P} \left(  \texttt{N}_{\rm p}(\mathbf{a} )    \right)
\end{equation}
where  $\mathfrak{P}$ (c.f. \eqref{eq:penalty_term}) promotes smoothness, 
$\mathfrak{f}_{\rm jac}$ (cf. \eqref{eq:f_jac_tmp})  enforces  local bijectivity,
and the function
$\mathfrak{f}_{\rm msh}$ 
controls the quality of the deformed mesh,
\begin{equation}
\label{eq:f_msh}
\mathfrak{f}_{\rm msh}({\Phi}   )
=
\frac{1}{|\Omega|}
\sum_{k=1}^{N_{\rm e}^{\rm pb}  } \; \; 
|  \texttt{D}_k^{\rm pb}  |
\int_{  \widehat{\texttt{D}}  } 
{\rm exp} 
\left(
\frac{q_{k}^{\rm msh}(\Phi)}{q_{k}^{\rm msh}(\texttt{id})}
\, - \, 
\kappa_{\rm msh}
\right)
 \; dx,
\end{equation}
where
\begin{equation}
\label{eq:f_msh_b}
 q_{k}^{\rm msh}(\Phi) : = 
 \frac{1}{d^2}
\left(
\frac{ \|  \nabla \Psi_{\Phi,k}^{\rm pb} \|_{\rm F}^2  }{( {\rm det} ( \nabla  \Psi_{\Phi,k}^{\rm pb}   )    )_+^{2/d}}\right)^2,
\quad
d=2.
\end{equation}
We observe that the ratio
$ q_{k}^{\rm msh}$ measures the degree of anisotropy of the mesh and it can thus be interpreted as a measure of the quality of the deformed mesh ---
we recall that the indicator
$ q_{k}^{\rm msh}$  was  used in \cite{zahr2020implicit} in the framework of DG methods.
As discussed in \cite{taddei2022optimization}, 
the decision to activate 
the penalty terms
$\mathfrak{f}_{\rm jac}$ and/or 
$\mathfrak{f}_{\rm msh}$ depends on the particular way we treat parameterized geometries --- map-then-discretize or discretize-then-map: in the numerical experiments, we activate only $\mathfrak{f}_{\rm msh}$ with $\kappa_{\rm msh}=10$.
\end{subequations}
In \eqref{eq:penalty_function}, we employ an $L^2$ penalization of the Hessian rather than an $L^{\infty}$ penalization as in \eqref{eq:f_pen_theory}: this choice simplifies computations and is justified by the fact that the ansatz $\texttt{N}_{\rm p}$ consists of a finite expansion  of modes.

\subsection{Target function}
\label{sec:target_function}
The target $\mathfrak{f}_{\mu}^{\rm tg}$ measures the degree of similarity between the available estimate of the solution field $u_{\mu}$ to the problem of interest and a suitable template solution 
 or template reduced space; 
$\mathfrak{f}_{\mu}^{\rm tg}$ relies on the introduction of a sensor $s_{\mu}:=\mathfrak{s}_{\mu}(u_{\mu})$ which should highlight the coherent structures we wish to track. We might distinguish between \emph{point-set sensors} and \emph{distributed sensors}.
The choice  of the sensor   might reflect the connection between distributed sensors and the shock sensors that are used to activate artificial viscosities in high-order discretization \cite{nicoud1999subgrid,persson2006sub}.
We anticipate that in the numerical experiments we 
combine  point-set and distributed sensors
to take into account different sources of information and  to robustify the greedy procedure outlined in section \ref{sec:parametric_map}
(cf. \eqref{eq:target_blade}).

\subsubsection*{Target function based on point-set sensors}
Point-set sensors are based on the introduction of a scalar testing function  that selects the   points of the mesh $\{ x_{\mu,j}^{\star} \}_{j=1}^{N_{\mu}^{\star}} $ where the solution $u_{\mu}$ satisfies a suitable user-defined criterion (cf. \cite{iollo2022mapping}),
for any $\mu\in \mathcal{P}$.
Given the template point cloud 
$\{ \bar{x}_{j}^{\star} \}_{j=1}^{\overline{N}^{\star}}$ --- which can be prescribed \emph{a priori} or be  chosen based on one specific value of the parameters
--- and the parameter  value $\mu\in \mathcal{P}$,
we first rely on a standard point-set registration (PSR) procedure
that takes as inputs the point clouds
$\{ \bar{x}_{j}^{\star} \}_{j=1}^{\overline{N}^{\star}}$ and
$\{ x_{\mu,j}^{\star} \}_{j=1}^{N_{\mu}^{\star}}$,
to determine the deformed point cloud
$\{ \widehat{x}_{j,\mu}^{\star} \}_{j=1}^{\overline{N}^{\star}}$  that  approximates (in a suitable sense)  the target  cloud 
$\{ x_{\mu,j}^{\star} \}_{j=1}^{N_{\mu}^{\star}}$
(see, e.g., \cite{myronenko2010point} for further details);
then, we solve the optimization problem \eqref{eq:tractable_optimization_based_registration} with 
\begin{equation}
\label{eq:target_pointset}
\mathfrak{f}_{\mu}^{\rm tg}
\left(
\Phi
\right)
=
\frac{1}{  \overline{N}^{\star}  }
\sum_{i=1}^{ \overline{N}^{\star}   }
\|
\Phi(  \bar{x}_{i}^{\star}   )
-
\widehat{x}_{i,\mu}^{\star} 
\|_2^2.
\end{equation}
As discussed in \cite{iollo2022mapping} and also
\cite{taddei2022optimization},
we can interpret the solution to \eqref{eq:tractable_optimization_based_registration} with objective \eqref{eq:penalty_function}-\eqref{eq:target_pointset}  as an approximate projection of the mapping returned by the PSR procedure
--- which is neither guaranteed to map the  boundary of the domain $\Omega$ in itself nor to be globally  bijective  --- 
 onto the space of admissible bijections in $\Omega$.

In order to speed up calculations, 
it is important to 
exploit the structure of the mapping in \eqref{eq:nonlinear_ansatz}
and 
precompute 
$ \Psi^{-1}$ where needed before calling the optimizer. 
In more detail, in order to speed up the evaluation of 
$\mathfrak{f}_{\mu}^{\rm tg}$ in \eqref{eq:target_pointset},
we first compute and store
$\bar{z}_j^{\star} = \Psi^{-1}(\bar{x}_j^{\star})$ for $j=1,\ldots, \overline{N}^{\star}$ and then we evaluate \eqref{eq:target_pointset} as 
$$
\mathfrak{f}_{\mu}^{\rm tg}
\left(
\Phi
\right)
=
\frac{1}{  \overline{N}^{\star}  }
\sum_{i=1}^{ \overline{N}^{\star}   }
\|
\Psi \left( \Phi_{\rm p}(\bar{z}_{i}^{\star} ) \right)
-
\widehat{x}_{i,\mu}^{\star} 
\|_2^2.
$$ 
Similar reasoning can be exploited for the evaluation of  \eqref{eq:f_msh}.

\subsubsection*{Target function based on distributed sensors}
 A distributed sensor is a function of the solution field $u_{\mu}$ that highlights the features of $u_{\mu}$ that we wish to track.
 Given the reduced $n$-dimensional space $\mathcal{Z}_n$ embedded in a Hilbert space $\mathcal{X}$ defined over $\Omega$,  recalling \eqref{eq:Lagrangian_approximations}, we find that  the parametric mapping $\Phi$ should minimize the target
\begin{equation}
\label{eq:dream_target}
 \mathfrak{f}_{\mu}^{\rm tg,opt}
\left(
\Phi
\right)
:=
\min_{\zeta\in \mathcal{Z}_n}
\int_{\Omega}
\|
u_{\mu} \circ \Phi
-\zeta
\|_2^2 \, dx,
 \end{equation}
 over all values of $\mu$ in $\mathcal{P}$:
 $ \mathfrak{f}_{\mu}^{\rm tg,opt}
\left(
\Phi
\right)$ is the projection error in the mapped configuration. Note that
 $\mathfrak{f}_{\mu}^{\rm tg,opt}$  depends on the choice of the reduced space
 $\mathcal{Z}_n$ whose selection is inherently coupled with the problem of finding $\Phi$: we postpone the procedure for the construction of  the reduced space to the next section.

The explicit use of the solution $u_{\mu}$ in the optimization statement \eqref{eq:tractable_optimization_based_registration}
 is  computationally unfeasible and prone to instabilities: 
 first, the solution $u_{\mu}$ is typically defined over an unstructured grid for which function evaluation at 
 the deformed quadrature points $\{  \Phi(x_q^{\rm qd}) \}_{q=1}^{N_{\rm qd}}$ is extremely expensive; second, we might exploit prior knowledge about the problem of interest to identify a scalar function of $u_{\mu}$ that better isolates 
 the features we wish to track (e.g., shocks) using registration,
  from 
  the features we expect to be able to approximate through a linear expansion of (mapped) snapshots.

Exploiting the form of $\Phi$, and the change-of-variable  $x= \Psi(\xi)$, we find
 $$
 \int_{\Omega}
\|
u_{\mu} \circ \Phi
-\zeta
\|_2^2 \, dx
=
 \int_{\Omega_{\rm p}}
\|
u_{\mu} \circ \Psi \circ \Phi_{\rm p}
-\zeta\circ \Psi^{-1}
\|_2^2 \, J(\Psi) d \xi.
$$
If we replace $u_{\mu} \circ \Psi $ with a scalar sensor $s_{\mu}$ defined over the domain $\Omega_{\rm p}$ and the reduced space $\mathcal{Z}_n$ for the solution with a  space (dubbed \emph{template space}) $\mathcal{S}_n\subset L^2(\Omega_{\rm p})$ for the sensor, we obtain 
\begin{equation}
 \label{eq:distributed_target}
 \mathfrak{f}_{\mu}^{\rm tg}
\left(
\Phi
\right)
:=
\min_{\nu \in \mathcal{S}_n}
\int_{\Omega_{\rm p}}
\big| 
s_{\mu} \circ \Phi_{\rm p}
-\nu
\big|^2  J(\Psi)  \, dx,
 \end{equation}
  which is the target function employed in the numerical experiments.

From
the definition \eqref{eq:distributed_target}, we deduce that
 computation of $ \mathfrak{f}_{\mu}^{\rm tg}(\Phi)$ requires to evaluate 
 $s_{\mu}$ in the deformed quadrature points
 $\{\Phi_{\rm p}(x_q^{\rm qd,p}) \}_q$ of the mesh $\mathcal{T}_{\rm p}$ at each iteration of the optimization algorithm for \eqref{eq:tractable_optimization_based_registration}.
Since the solution is discontinuous, we consider a P1 discretization of the sensor $s_{\mu}$ over a linear mesh $\mathcal{T}_{\rm p,s}$ of $\Omega_{\rm p}$, which is generated independently of $\mathcal{T}_{\rm p}$.
Following  
 \cite{luke2012fast}, we rely on  KD-trees
to speed up mesh interpolation 
 (cf. Matlab function \texttt{KDTreeSearcher}): since $\mathcal{T}_{\rm p}$ is a linear mesh with a modest number of elements, the evaluation of the sensor
 $s_{\mu}$ in the deformed quadrature points is still affordable. Notice that the same fast mesh interpolation routine should also be used to evaluate the FE mapping $\Psi$ in \eqref{eq:f_msh}.
 
\subsection{Construction of the parametric map}
\label{sec:parametric_map}
The target function \eqref{eq:distributed_target} depends on the template space $\mathcal{S}_n\subset L^2(\Omega_{\rm p})$. 
Following \cite{taddei2021registration,taddei2021space}, 
given a set of sensor snapshots
$\{s_{\mu} :\mu\in \mathcal{P}_{\rm train}\}$ with $\mathcal{P}_{\rm train}=
\{\mu^k\}_{k=1}^{n_{\rm train}}$, 
we propose  an iterative procedure that  
performs registration over the entire training set and then exploits the results to
 update  the template space $\mathcal{S}_n$ in a greedy fashion.
 Algorithm \ref{alg:registration} summarizes the computational procedure.
 
\begin{algorithm}[H]                      
\caption{Registration algorithm}     
\label{alg:registration}     

 \small
\begin{flushleft}
\emph{Inputs:}  $\{  s_{\mu} : \mu\in 
\mathcal{P}_{\rm train} \}$ snapshot set, 
$\mathcal{S}_{n_0} = {\rm span} \{ 
 s_{\mu^{\star,(i)}} 
\}_{i=1}^{n_0}$ initial template space;
$\mathcal{T}_{\rm pb}$ mesh for HF computations.
\smallskip

\emph{Outputs:} 
${\mathcal{S}}_n $ template space, 
$\mathbf{W} \in \mathbf{R}^{M\times m}$ mapping space,
$\{    \mathbf{a}_{\mu^k}^{\star}  \}_k \subset \mathbb{R}^m$ optimal  mapping coefficients.
\end{flushleft}                      

 \normalsize 

\begin{algorithmic}[1]
\State
Initialization: 
$\mathcal{S}_{n=n_0} = \mathcal{S}_{n_0}$,
$\Xi_{\star} = \{\mu^{\star,(i)} \}_{i=1}^{n_0}$,
$\mathbf{W} =\mathbbm{1}_M$.
\vspace{3pt}

\For {$n=n_0, \ldots, n_{\rm max}-1$ }

\State
$
\left[\mathbf{a}_{\mu}^{\star}, 
\mathfrak{f}_{\mu}^{\star}  \right]
\, = \,
\texttt{registration} \left(
s_{\mu}, \,   \mathcal{S}_n, \,  \mathbf{W}, \, \mathcal{T}_{\rm pb}, \, 
\Psi, 
\mathbf{a}_{\mu}^{0}
\right)
$ for all $\mu \in \mathcal{P}_{\rm train}$

\hfill
\emph{see Remark \ref{remark:practical_implementation} for definition of $\mathbf{a}_{\mu}^{0}$}
\vspace{3pt}

\State
$[\mathbf{W}, \; \{  
\mathbf{a}_{\mu}^{\rm proj}
  \}_{\mu\in \mathcal{P}_{\rm train}} ]  =
\texttt{POD} \left( 
\{  \mathbf{a}_{\mu}^{\star}  \}_{\mu\in \mathcal{P}_{\rm train}}, 
tol_{\rm pod}  ,
(\cdot, \cdot)_2
 \right),$
\vspace{3pt}

\If{$\max_{\mu\in \mathcal{P}_{\rm train}}  \mathfrak{f}_{\mu}^{\star}   < \texttt{tol}$}
\State
\texttt{break}

\Else
\State
$\Xi_{\star} = 
\Xi_{\star} \cup 
\{ \mu^{\star,(n+1)} \}$
with $\mu^{\star, (n+1)}= {\rm arg} 
\max_{\mu\in \mathcal{P}_{\rm train}}  
\mathfrak{f}_{\mu}^{\star}$.
\vspace{3pt}

\State
$\mathcal{S}_{n+1} =
 {\rm span}  \{ s_{\mu^{\star,(i)}}  \circ \texttt{N} _{\rm p}( 
\mathbf{a}_{\mu^{\star,(i)}}^{\rm proj}) \}_{i=1}^{n+1}$.
 \vspace{3pt}
 
\EndIf
\EndFor
\end{algorithmic}
\end{algorithm}

 Given the orthonormal basis  $\{ \varphi_i \}_i$ of $\mathcal{U}$ and the operator $\texttt{N}$, we define the orthogonal 
matrix $\mathbf{W}\in \mathbb{R}^{M\times m}$ and the low-rank map $\widehat{\texttt{N}}(\mathbf{a}): = \texttt{N} ( \mathbf{W} \mathbf{a}  )$; by construction, we have
 $\| \mathbf{a}  \|_2 = \|  \sum_i  ( \mathbf{W}   \mathbf{a} ) \varphi_i \|$ for any $\mathbf{a}\in \mathbb{R}^m$. Then, we introduce notation:
 $$
\left[\mathbf{a}_{\mu}^{\star}, 
\mathfrak{f}_{\mu}^{\star}  \right]
\, = \,
\texttt{registration} \left(
s_{\mu}, \,   \mathcal{S}_n, \,  \mathbf{W}, \, \mathcal{T}_{\rm pb}, \, 
\Psi, \, 
\mathbf{a}_{\mu}^{0}
\right)
$$
to refer to the function that takes as inputs 
(i) the target sensor $s_{\mu}:\Omega_{\rm p}\to \mathbb{R}$, 
(ii) the template space $\mathcal{S}_n$,
(iii) the  orthogonal matrix   $\mathbf{W}$, 
(iv) the mesh $\mathcal{T}_{\rm pb}$, 
(v) the geometric mapping $\Psi:\Omega_{\rm p}\to \Omega$ and 
(vi) the initial guess  $ \mathbf{a}_{\mu}^{0} \in \mathbb{R}^m$
for the optimizer,
and returns 
(I) the mapping coefficients $\mathbf{a}_{\mu}^{\star}$ that minimize
the (reduced) objective 
$\mathbf{a} \mapsto \mathfrak{f}_{\mu}^{\rm obj}(\mathbf{W} \mathbf{a}_{\mu}^{\star})$, and 
(II) the value of the target function
$\mathfrak{f}_{\mu}^{\star}  = 
\mathfrak{f}_{\mu}^{\rm tg}(\mathbf{W} \mathbf{a})
$.
Note that the objective  $\mathfrak{f}_{\mu}^{\rm obj}: \mathbb{R}^M \to \mathbb{R}_+$  of
\eqref{eq:tractable_optimization_based_registration} 
   depends on the mesh  $\mathcal{T}_{\rm pb}$  through  the term 
$\mathfrak{f}_{\rm msh}$ in \eqref{eq:f_msh}.
We also introduce the function
$$
[ \mathbf{W}_{\rm new},  \; \{  
\mathbf{a}_{\mu}^{\rm proj}
  \}_{\mu\in \mathcal{P}_{\rm train}} ]  =
\texttt{POD} \left( 
\{  \mathbf{W}_{\rm  old} \mathbf{a}_{\mu}^{\star}  \}_{\mu\in \mathcal{P}_{\rm train}}, 
tol_{\rm pod}  ,
(\cdot, \cdot)_2
 \right),
$$
which implements POD
based on the method of snapshots with Euclidean inner product $(\cdot,\cdot)_2$:
the tolerance 
$tol_{\rm pod}>0$ drives 
the selection of the number of modes $m$ based on the energy criterion 
\begin{equation}
\label{eq:POD_cardinality_selection}
m := \min \left\{
m': \, \sum_{j=1}^{m'} \lambda_{j} \geq  \left(1 - tol_{\rm pod} \right) 
\sum_{i=1}^{n_{\rm train}} \lambda_i
\right\},
\end{equation} 
where $ \lambda_1\geq \ldots \geq \lambda_{n_{\rm train}}\geq 0$ are the eigenvalues of the Gramian matrix $\mathbf{C}\in \mathbb{R}^{n_{\rm train}\times n_{\rm train}}$ such that
$(\mathbf{C})_{k,k'}
=
\mathbf{a}_{\mu^k}^{\star} \cdot 
\mathbf{a}_{\mu^{k'}}^{\star} 
$. 
The function \texttt{POD} returns also the mapping coefficients associated with the 
projected displacements
$ \mathbf{a}_{\mu}^{\rm proj} =\mathbf{W}_{\rm new}^\top  
\mathbf{W}_{\rm old}  \mathbf{a}_{\mu}^{\star}$;
the latter  are used to initialize the iterative method for the optimization problem for the subsequent iterations.

\begin{remark}
\label{remark:practical_implementation}
\textbf{Further implementation details.}
Since the optimization problem is highly non-convex, the choice of the initial condition is extremely important to avoid convergence to unsatisfactory local minima. 
For $n=n_0+1,n_0+2,\ldots,n_{\rm max}-1$, we simply use 
$\mathbf{a}_{\mu}^{0}=\mathbf{a}_{\mu}^{\rm proj}$ (cf. Line 4).  On the other hand, for the first iteration, we first reorder the parameters in $\mathcal{P}_{\rm train}$ so that
$\mu^{(1)}
=
{\rm arg}\min_{\mu \in \mathcal{P}_{\rm train}}
\| \mu - \mu^{\star, (1)}  \|_2$ and
$$
\mu^{(k)} = {\rm arg} \min_{\mu \in \mathcal{P}_{\rm train} 
\setminus \{  \mu^{(i)} \}_{i=1}^{k-1}
 }
 \left(
 \min_{\mu' \in   \{  \mu^{(i)} \}_{i=1}^{k-1}}
 \| \mu - \mu'  \|_2
 \right),
 \quad
k=2,\ldots,n_{\rm train};
$$
then, we choose the initial condition as follows:
$$
\mathbf{a}_{\mu^{(1)}}^0 = 0,
\quad
\mathbf{a}_{\mu^{(k)}}^0
=
 {\mathbf{a}}_{\mu^{({\rm ne}_k)}}^\star
\;\;
{\rm with} \;\;
{\rm ne}_k
={\rm arg}\, \min_{j=1,\ldots,k-1} 
 \| \mu^{(j)} - \mu^{(k)} \|_2,
$$
for $k=2,\ldots,n_{\rm train}$.
In a previous implementation of the procedure, we also included box constraints in the optimization statement 
(cf. \cite[section 3.1.2]{taddei2020registration}) to control  the sensitivity of the mapping coefficients to parameter variations;
\begin{equation}
\label{eq:box_constraints}
\|
\mathbf{a}_{\mu^{(k)}}^{\star}
-
\mathbf{a}_{\mu^{(k)}}^0
\|_{\infty}
\leq
C_{\infty} \| \mu^{(k)} -  \mu^{({\rm ne}_k)}  \|_2,
\quad
{\rm with} \; \;\; C_{\infty}=10;
\end{equation}
however, in the numerical experiments of the present work, we found that the solution to the unconstrained problem satisfied the constraints   \eqref{eq:box_constraints} for all the experiments considered.
We further observe that
Algorithm \ref{alg:registration}
 depends on several hyper-parameters. In our tests, we set  
 $\mathcal{S}_{n_0=1} = {\rm span}\{ s_{\bar{\mu}}  \}$,
 where $\bar{\mu}$ is the centroid of $\mathcal{P}$; furthermore, we set $n_{\rm max}=6$,  $tol_{\rm pod}=5\cdot 10^{-3}$ and  
$\texttt{tol}=10^{-4}$.
\end{remark}

\begin{remark}
\label{remark:generalization}
\textbf{Generalization.}
Given the dataset $\{ (\mu^k, \mathbf{a}_{\mu^k}^{\star})  \}_{k=1}^{n_{\rm train}}$ as provided by Algorithm \ref{alg:registration}, 
we resort to a multi-target regression algorithm to learn a regressor $\mu \mapsto \widehat{\mathbf{a}}_{\mu}$
for the mapping coefficients,
 and ultimately define the  parametric mapping
\begin{equation}
\label{eq:parametric_mapping_Phi}
 {\Phi} : \Omega \times \mathcal{P}  \to \Omega,
\quad
 {\Phi}_{\mu} : =
 \texttt{N} \left( \mathbf{W}    \widehat{\mathbf{a}}_{\mu} \right).
\end{equation}
We here resort to radial basis function (RBF, \cite{wendland2004scattered}) approximation: other regression algorithms could also be considered. 
Similarly to \cite{taddei2020registration,taddei2021space},
to  avoid overfitting, we verify the statistical significance of the RBF estimators. We randomly split the dataset 
$\{ (\mu^k, \mathbf{a}_{\mu^k}^{\star})  \}_{k=1}^{n_{\rm train}}$ into the learning and test sets
$\{ (\mu^k, \mathbf{a}_{\mu^k}^{\star})  \}_{k\in D_{\rm learn}}$ and
$\{ (\mu^j, \mathbf{a}_{\mu^j}^{\star})  \}_{j\in D_{\rm test}}$
with $D_{\rm learn} \cap D_{\rm test} = \emptyset$ and
$D_{\rm learn} \cup D_{\rm test} =
\{1, \ldots, n_{\rm train} \}$
(we here consider a $80\%$-$20\%$ learning/test split);
we compute the RBF approximation 
$\widehat{\mathbf{a}}: \mathcal{P} \to \mathbb{R}^m$ based on the learning set and we compute the out-of-sample R-squared coefficient for each component:
\begin{equation}
\label{eq:Rsquared}
\texttt{R}_i^2 =  
1 -
\frac{
\sum_{j\in D_{\rm test} }
\left(  
\mathbf{a}_{\mu^j}^{\star}
-    \widehat{\mathbf{a}}_{\mu^j}   \right)_i^2}{
\sum_{j\in D_{\rm test}}
\left(  
\mathbf{a}_{\mu^j}^{\star}
-    \bar{\mathbf{a}}  \right)_i^2
}
,
\quad
\bar{\mathbf{a}}  = \frac{1}{|D_{\rm learn}|} \sum_{k\in D_{\rm learn}} \,   \mathbf{a}_{\mu^k}^{\star},
\quad
i=1,\ldots,m.
\end{equation}
Then, we retain exclusively modes for which $\texttt{R}_i^2$ is above the threshold $R_{\rm min}=0.70$. 
\end{remark}

\begin{remark}
\label{remark:pointsetregistration}
\textbf{Parametric registration based on point-set sensors.}
The greedy procedure in Algorithm \ref{alg:registration} is motivated by the need to construct the template space $\mathcal{S}_n$. If we rely on point-set sensors, we do not have to perform multiple iterations. However, we empirically found that performing two iterations of the for loop in Algorithm \ref{alg:registration} does not hinder  computational efficiency
--- since the cost is dominated by the high-dimensional registration problems solved during the first iteration --- and has a beneficial effect on generalization outside the training set.
\end{remark}

\section{Application to an inviscid flow past an array of LS89 turbine blades}
\label{sec:numerics}
We consider the problem of estimating the solution to the two-dimensional Euler equations past an array of LS89 turbine blades. 

\subsection{Model problem}
We consider the computational domain depicted in Figure 
\ref{fig:meshes}(a); we prescribe total temperature, total pressure and flow direction at the inflow, static pressure at the outflow, non-penetration condition on the blade and periodic boundary conditions on the lower and upper  boundaries.
We study the sensitivity of the solution with respect to two parameters: the free-stream Mach number ${\rm Ma}_{\infty}$ and the height of the channel $H$,  i.e.~ $\mu=[ H, {\rm Ma}_{\infty}]$. We consider the parameter domain 
$\mathcal{P}=[0.95,1.05]\times [0.9,0.95]$. 
We refer to 
\cite{barral2023registration} for a detailed presentation of the 
employed 
nondimensionalization, HF  DG  formulation and  pseudo-transient  continuation strategy.

Figure \ref{fig:vis_model_problem} shows the distribution of the Mach field for $\mu_{\rm min}=[0.95,0.9]$ and 
$\mu_{\rm max}=[1.05,0.95]$; 
Figure \ref{fig:vis_model_problem_slices}(a) shows the behavior of the Mach number on the upper side of the blade for four parameter values, while 
Figure \ref{fig:vis_model_problem_slices}(b) shows the behavior of the entropy profile ${\rm E}=\log (p)  - \gamma \log (\rho)$  where $p$ is the pressure field, $\gamma=1.4$ is the ratio of specific heats, and $\rho$ is the density field.
The solution develops a normal shock on the upper side of the blade for sufficiently large values of ${\rm Ma}_{\infty}$ and $H$; furthermore, the entropy ${\rm E}$ exhibits several peaks that correspond to the blade wakes.
The solution    develops two shocks at the trailing edge, which are highly undesirable for turbomachinery applications:
we expect that at the trailing edge viscous effects might not be negligible; for this reason, 
a more thorough investigation should rely on a    model that  accounts for viscous effects.
We observe that the shock location and the entropy peaks are sensitive to the value of the parameter: this justifies the application of registration procedures.

\begin{figure}[h!]
\centering
\subfloat[]{ 
\includegraphics[width=.47\textwidth]{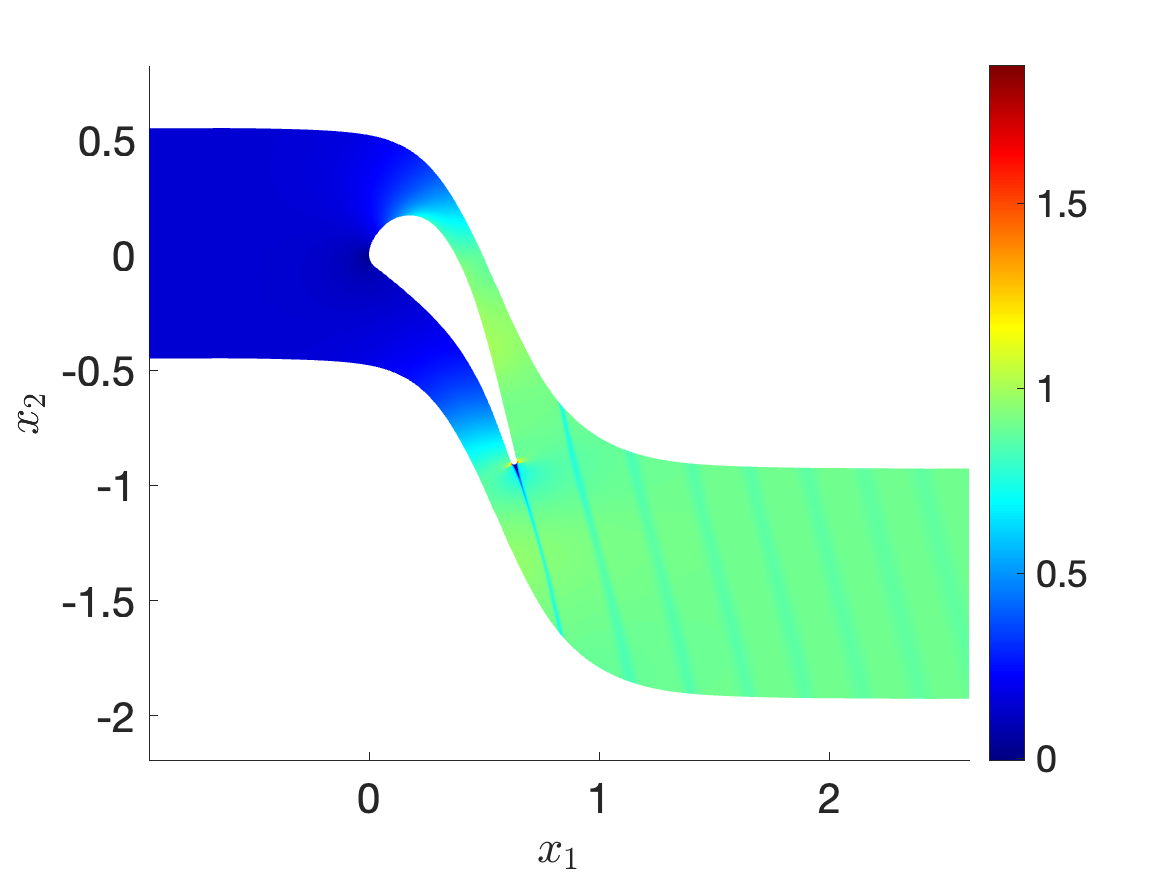}
}
~~
\subfloat[]{
\includegraphics[width=.47\textwidth]{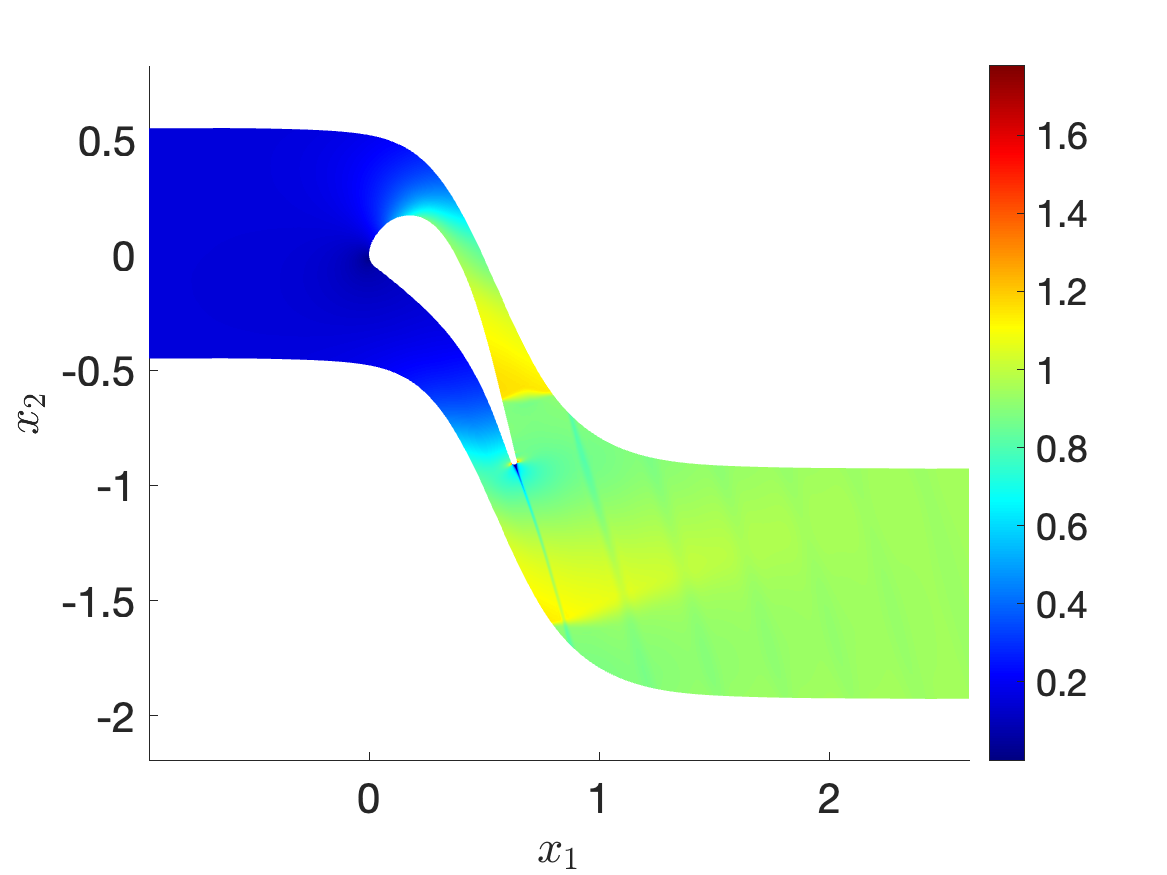}
}


\caption{inviscid flow past an array of LS89 turbine blades.
(a)-(b)   Mach field for $\mu=[0.95,0.9]$ and $\mu=[1.05,0.95]$.
}
\label{fig:vis_model_problem}
\end{figure} 

\begin{figure}[h!]
\centering
\subfloat[]{ 
\includegraphics[width=.47\textwidth]{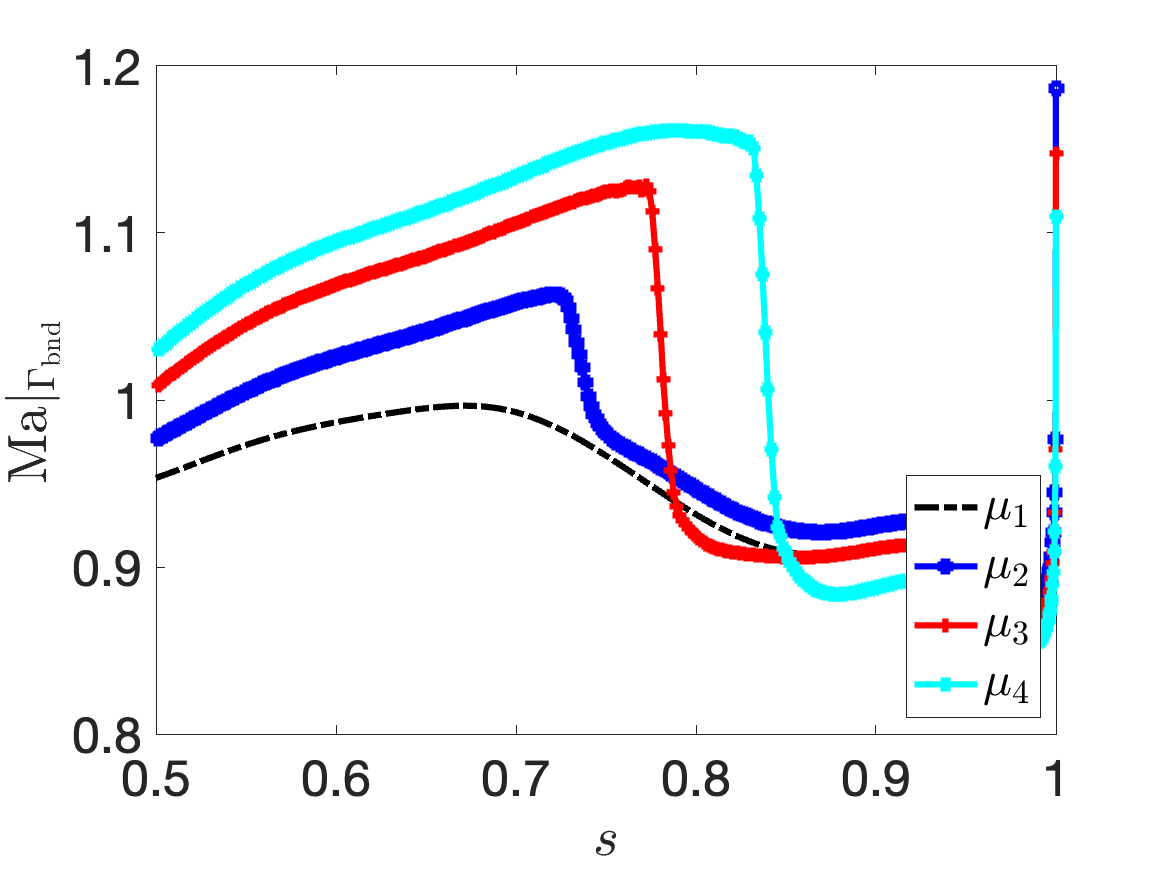}
}
~~
\subfloat[]{
\includegraphics[width=.47\textwidth]{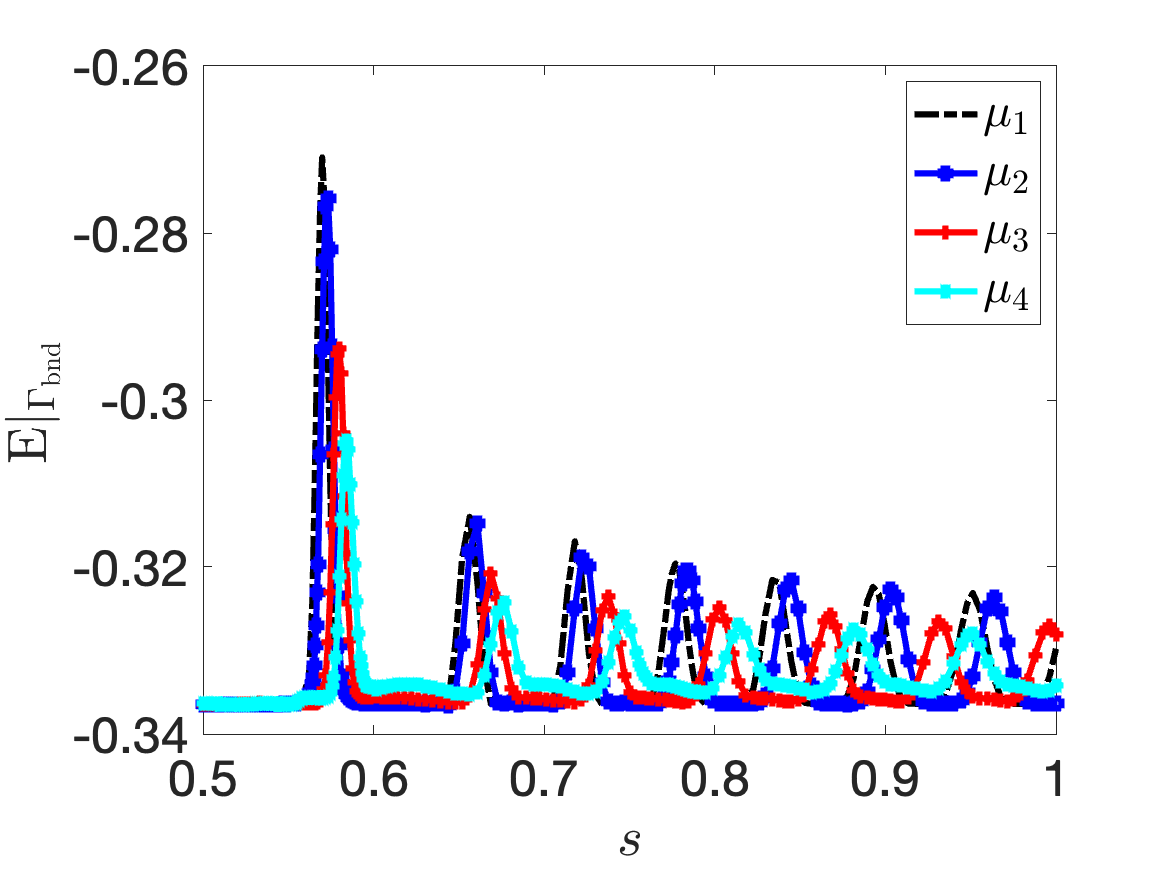}
}

\caption{inviscid flow past an array of LS89 turbine blades.
(a)  Mach  profile over the upper side of the blade for four values of the parameter in $\mathcal{P}$.
(b)  entropy  profile on the bottom boundary   for four values of the parameter in $\mathcal{P}$.
The parameter $s$ denotes the normalized curvilinear coordinate,  $s\in [0,1]$.
}
\label{fig:vis_model_problem_slices}
\end{figure} 

\subsection{Definition of the mapping ansatz}
We deal with geometry variations through a piecewise-smooth mapping associated with the partition in 
Figure \ref{fig:meshes}(b). 
We set $H_{\rm ref}=1$ and we define the curve $x_1\mapsto f_{\rm btm}(x_1)$ that describes the lower boundary $\Gamma_{\rm btm}$ of the domain  $\Omega=\Omega(H=H_\mathrm{ref})$; then, we define $\widetilde{H}>0$ such that 
$x_1\mapsto f_{\rm btm}(x_1)+\widetilde{H}$ and
$x_1\mapsto f_{\rm btm}(x_1)+H-\widetilde{H}$ do not intersect the blade for any $H\in [0.95,1.05]$; finally, we define the geometric mapping
\begin{subequations}
\label{eq:geometric_mapping_turbine}
\begin{equation}
\Psi_H^{\rm geo}(x=[x_1,x_2])
=
\left[
\begin{array}{l}
x_1 \\
\psi_H^{\rm geo}(x)
\\
\end{array}
\right],
\quad
x\in \Omega,
\end{equation}
where 
\begin{equation}
\psi_H^{\rm geo}(x)
=
\left\{
\begin{array}{ll}
o_1(x_1)+ C(H) \left(x_2 - o_1(x_1) \right) & x_2 < o_1(x_1),\\
o_2(x_1)+ C(H) \left(x_2 - o_2(x_1) \right) & x_2 > o_2(x_1),\\
x_2 & {\rm otherwise},
\\
\end{array}
\right.
\end{equation}
with 
$o_1(x_1)=f_{\rm btm}(x_1)+\widetilde{H}$,
$o_2(x_1)=f_{\rm btm}(x_1)+H_{\rm ref} - \widetilde{H}$ and
$C(H) = \frac{H-H_{\rm ref}}{2\widetilde{H}}+1$.
\end{subequations}
Then, we consider computational maps \eqref{eq:nonlinear_ansatz} from the domain $\Omega=\Omega(H=H_{\rm ref})$ to $\Omega_{\mu}=\Omega(H)$ such that
\begin{equation}
\label{eq:nonlinear_ansatz_turbine}
\texttt{N}(x, \mathbf{a}, \mu) =
\Psi_H^{\rm geo} \circ \Psi \circ \texttt{N}_{\rm p}(\mathbf{a})
\circ \Psi^{-1} (x),
\end{equation}
for a proper choice of the mapping $\Psi$ and the polytope $\Omega_{\rm p}$.
Note that $\widetilde{\texttt{N}}(\mathbf{a}) =
  \Psi \circ \texttt{N}_{\rm p}(\mathbf{a})
\circ \Psi^{-1}$ defines a bijection in  $\Omega$.

\begin{figure}[h!]
\centering
\subfloat[]{ 
\includegraphics[width=.45\textwidth]{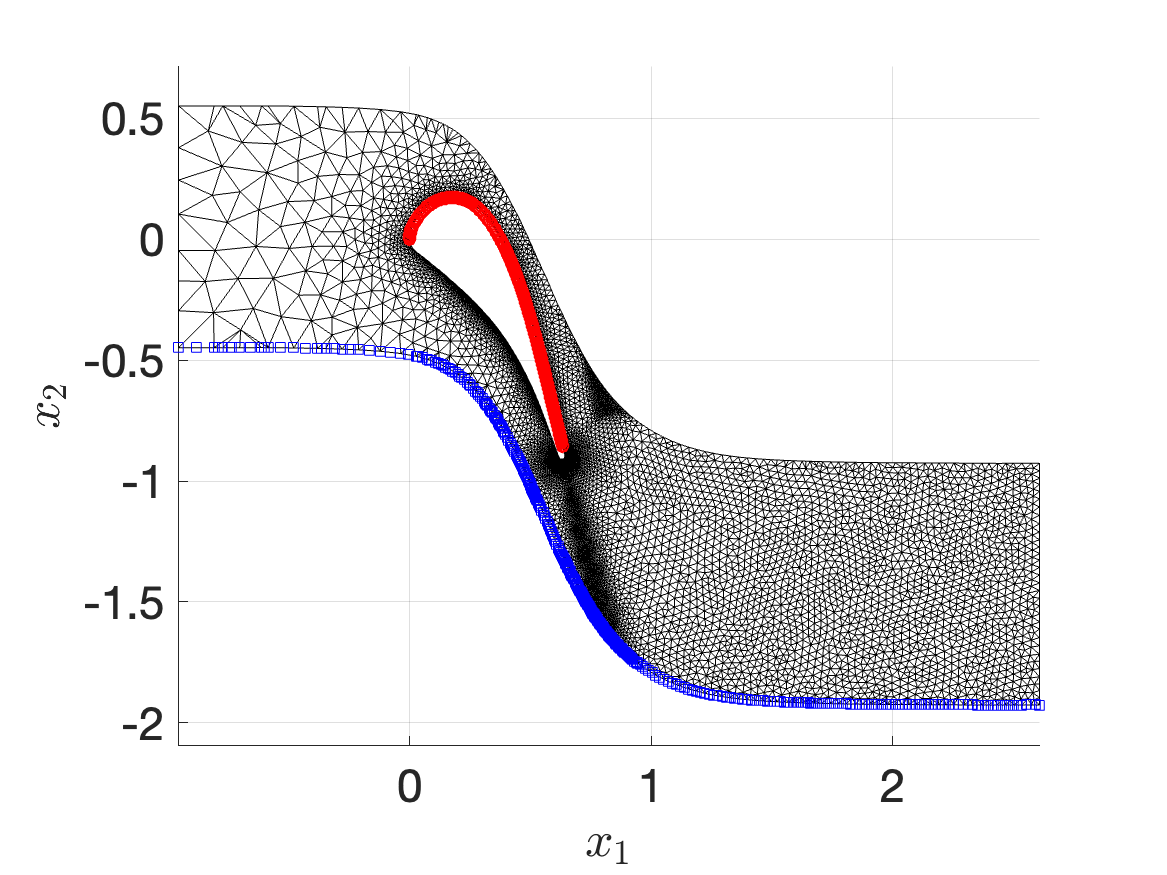}
}
~~
\subfloat[]{ 
\includegraphics[width=.45\textwidth]{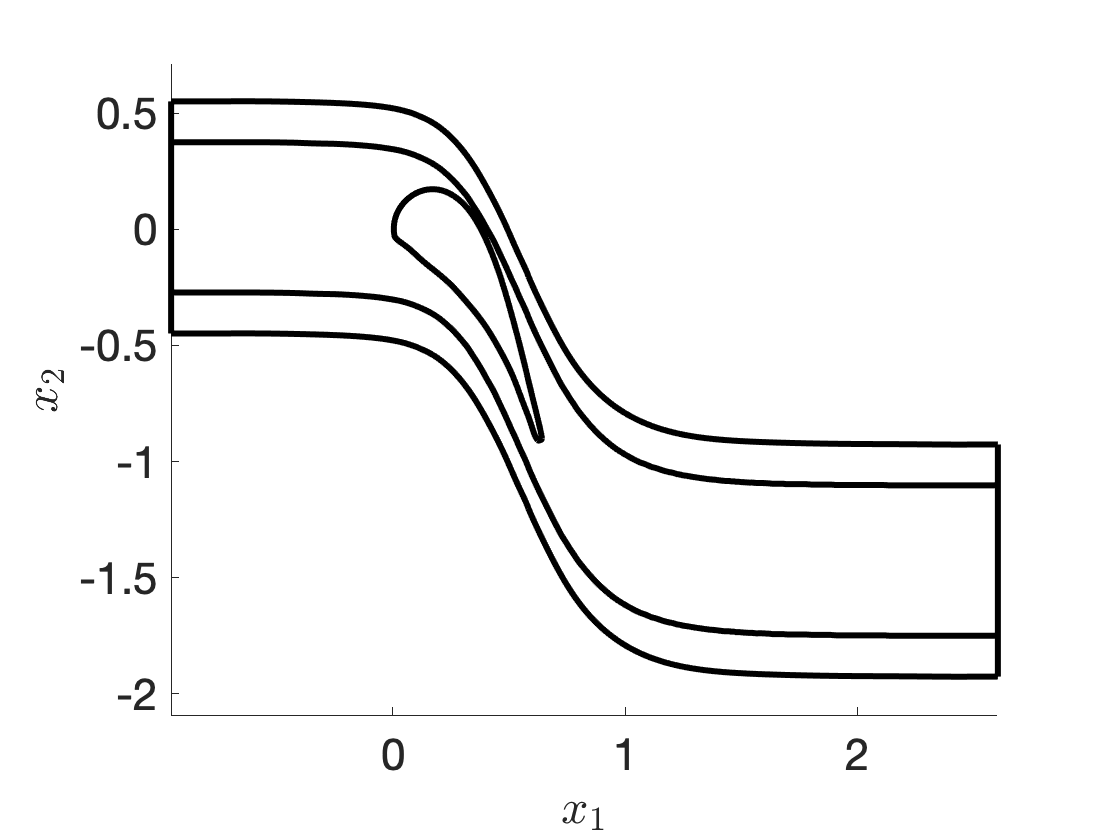}
}

\subfloat[]{
\includegraphics[width=.45\textwidth]{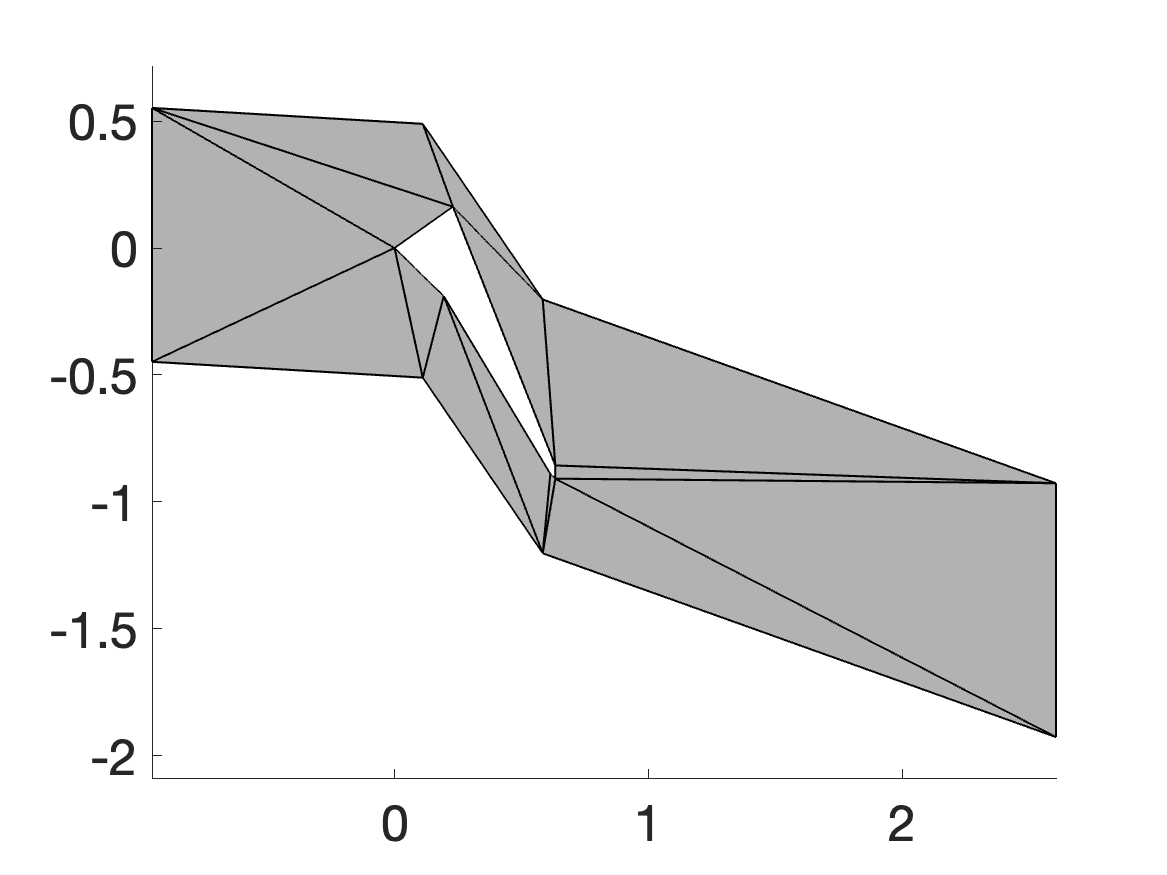}
}
~~
\subfloat[]{
\includegraphics[width=.45\textwidth]{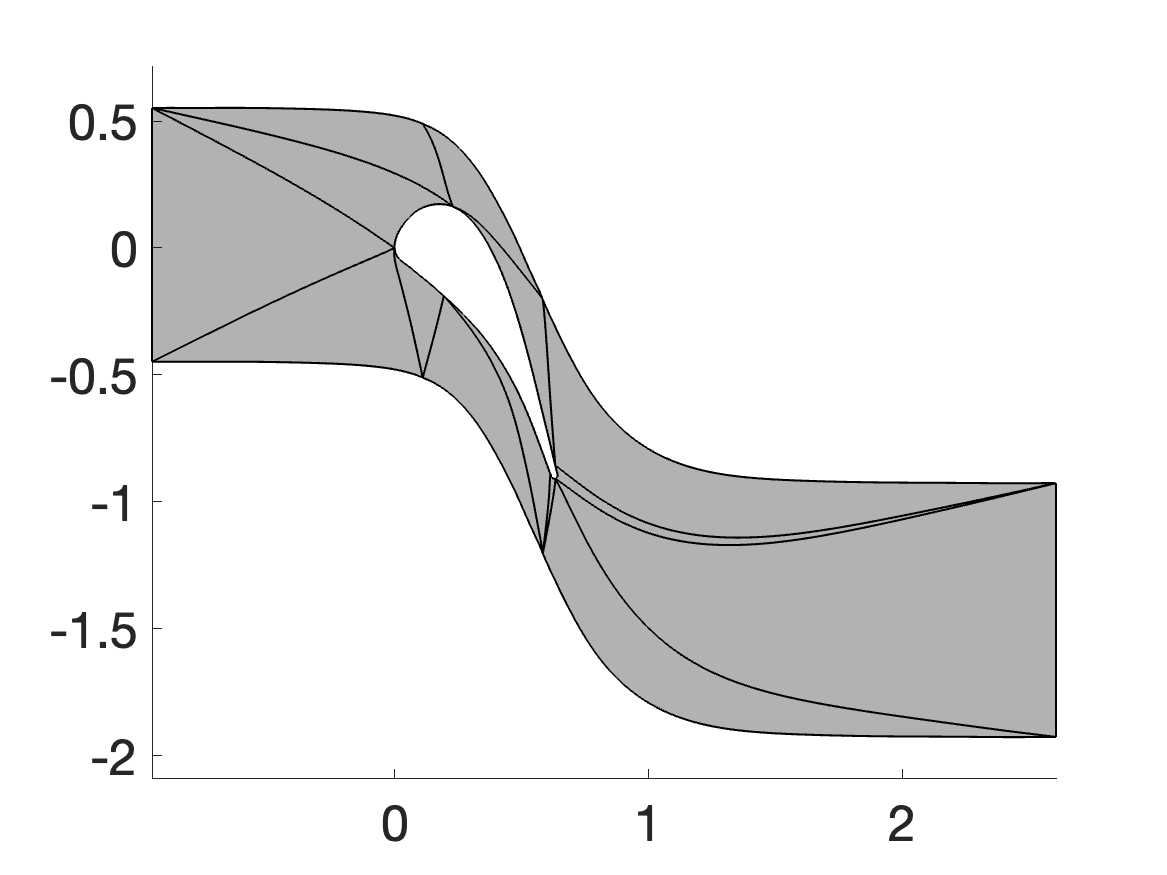}
}

\caption{inviscid flow past an array of LS89 turbine blades.
(a) computational mesh $\mathcal{T}_{\rm pb}$ for DG calculations;
red dots denote the nodes on the upper blade side;
blue squares  denote the nodes on the lower boundary.
(b) partition associated with the geometric map \eqref{eq:geometric_mapping_turbine}.
(c)-(d) coarse-grained linear    mesh 
 $\mathcal{T}_{\rm p}$ and
 curved mesh $\mathcal{T}$ 
 for registration.
}
\label{fig:meshes}
\end{figure} 

The polytope $\Omega_{\rm p}$ and the mapping 
$\Psi:\Omega_{\rm p} \to \Omega$ should be designed to ensure the periodicity constraint
\begin{equation}
\label{eq:periodicity_constraint}
\texttt{N}(x+[0,H_{\rm ref}], \mathbf{a}, \mu) =
\texttt{N}(x, \mathbf{a}, \mu) +[0,H],
\quad
\forall \, x\in \Gamma_{\rm btm}.
\end{equation}
The geometric mapping $\Psi^{\rm geo}$ satisfies the periodicity constraint; it is thus sufficient to enforce that 
$\widetilde{\texttt{N}}(x+[0,H_{\rm ref}], \mathbf{a}, \mu) =
\widetilde{\texttt{N}}(x, \mathbf{a}, \mu) +[0,H_{\rm ref}]$ for all
$x\in \Gamma_{\rm btm}$. Towards this end, 
(i)
we define the polytope $\Omega_{\rm p}$ 
with vertices $V$ such that
$V\cap \Gamma_{\rm btm, p} = V\cap \Gamma_{\rm top, p} +[0,H_{\rm ref}]$, where 
$\Gamma_{\rm btm, p},\Gamma_{\rm top, p}$ denote the lower and upper boundaries of the polytope
(i.e. for every vertex on the lower boundary there exists a corresponding vertex on the upper boundary); 
(ii)
we define  the linear mesh $\mathcal{T}_{\rm p}$ in Figure \ref{fig:meshes}(c) with matching nodes on the periodic boundaries\footnote{The same condition is enforced in the HF mesh.};
(iii) we define the mapping $\Psi$ by solving
\eqref{eq:optimization_georeg} in the affine space
$\widetilde{W} = 
\left\{
\texttt{id} + 
\varphi \in [\mathcal{X}_{\rm hf,p} ]^2 : \varphi \cdot \mathbf{n} \big|_{\partial \Omega_{\rm p} \cap \partial \Omega} = 0,
\;\;
\varphi |_{\Gamma_{\rm btm,p}}
=
\varphi |_{\Gamma_{\rm top,p}}
\right\}$
 and we set
${\mathcal{U}}$ (cf. \eqref{eq:mapping_space}) as
${\mathcal{U}}
=\left\{
\varphi \in [\mathcal{X}_{\rm hf,p} ]^2 : \varphi \cdot \mathbf{n} \big|_{\partial \Omega_{\rm p}} = 0,
\;\;
\varphi |_{\Gamma_{\rm btm,p}}
=
\varphi |_{\Gamma_{\rm top,p}}
\right\}$.
The resulting curved mesh is provided in Figure \ref{fig:meshes}(d); in the experiments, we consider polynomials of degree ten, which implies $M=1359$.
Note that,  since $\widetilde{\mathcal{W}}$ and ${\mathcal{U}}$ are spaces of piecewise polynomials, it suffices to enforce the   periodicity constraint  at mesh nodes.

\subsection{Definition of the sensor}
We consider a target function $\mathfrak{f}^{\rm tg}$ that combines a point-set sensor and a distributed sensor
\begin{equation}
\label{eq:target_blade}
\mathfrak{f}_{\mu}^{\rm tg}
=
\frac{1}{4}\sum_{i=1}^4 \|  \texttt{N}(\bar{x}_i, \mathbf{a}, \mu)  - \widehat{x}_{i,\mu} \|_2^2
\,+\,
\min_{\nu \in \mathcal{S}_n} \int_{\Omega_{\rm p}}
\big|
s_{\mu}\circ \texttt{N}_{\rm p}(\mathbf{a})
-
\nu
\big|^2 J(\Psi) \, dx,
\end{equation}
where the distributed sensor $s_{\mu}$ is obtained from the Mach field, 
the template space $\mathcal{S}_n$ is built adaptively using Algorithm \ref{alg:registration}, 
 and the salient points
$\{  \widehat{x}_{i,\mu}   \}_i$ correspond to the first three peaks of the entropy profile on the lower boundary (cf. Figure \ref{fig:vis_model_problem_slices}(b)) and the shock location on the upper boundary (cf. 
Figure \ref{fig:vis_model_problem_slices}(a)).

We identify the entropy peaks by computing the local maxima of the entropy profile on $\Gamma_{\rm btm,\mu}$. 
We detect the shock location on the upper blade side using the following procedure:
first, we compute the mean values $\{  \overline{{\rm Ma}}_j  \}_j$ and $\{  \overline{d{\rm Ma}}_j  \}_j$ of the Mach number and its tangential derivative on each facet  $\{\texttt{F}_j\}_j$ of the select boundary;
second, we find the index $j^{\star}$ such that
$\overline{{\rm Ma}}_{j^{\star}}>1>\overline{{\rm Ma}}_{j^{\star}+1}$ and 
$ \overline{d{\rm Ma}}_{j^{\star}} 
< - \frac{10^{-2}}{| \texttt{F}_{j^{\star}} |}$ and we return the estimate
$ {x}_{\rm shk,\mu}^{\rm raw}$ equal to the midpoint of the selected facet.

Given the raw estimates $\{  {x}_{i,\mu}^{\rm raw}: i=1,\ldots,4, \mu\in \mathcal{P}_{\rm train} \}$ of the point sensors, we apply radial basis function (RBF) regression to find smoother estimates that facilitate the generalization step (cf. Remark \ref{remark:generalization}); in order to ensure that the points are on the boundary, we apply RBF to the curvilinear coordinates of the points. As shown in Figure \ref{fig:vis_model_problem_slices}(a), the solution does not exhibit any shock for several parameter values; we do not consider these points to train the RBF surrogate,   and we rely on the surrogate itself to find a fictitious estimate of $\widehat{x}_{\rm shk,\mu}$ for all $\mu\in \mathcal{P}$.


\subsection{Performance of the registration procedure}
\label{sec:performance_registration}

We apply Algorithm \ref{alg:registration} based on a regular
$11\times 6$  grid of parameters $\mathcal{P}_{\rm train}$.
Figure \ref{fig:sensor_model_problem_salientpoints} compares the location of the sensors
$\{ \widetilde{x}_{i,\mu} = \left( \Psi_H^{\rm geo}  \right)^{-1}(\widehat{x}_{i,\mu} ) : i=1,\ldots,4, \mu\in \mathcal{P}_{\rm train}\}$ 
and their mapped counterparts
$\{ \widetilde{x}_{i,\mu}' =  \Phi_{\mu}^{-1}(\widehat{x}_{i,\mu} ) : i=1,\ldots,4, \mu\in \mathcal{P}_{\rm train}\}$;
similarly, Figure
\ref{fig:vis_model_problem_slices_ref} reproduces the results in Figure \ref{fig:vis_model_problem_slices} for the mapped solution field. 
We observe that the peaks of the mapped entropy field on $\Gamma_{\rm btm}$ and the shock on the upper side of the blade are nearly insensitive to the parameter value.
For all in-sample and out-of-sample configurations considered, the quality of the deformed mesh --- which is measured by \eqref{eq:f_msh_b} --- is comparable with the one of the original mesh.

A  thorough analysis of the computational cost of the procedure is beyond the scope of the present paper. We observe that the execution time of Algorithm \ref{alg:registration} is $54$ minutes   on a commodity Linux workstation; for comparison, the generation of the training set of HF simulations ($66$ simulations) requires $7$ hours and $33$ minutes on the same machine.

\begin{figure}[h!]
\centering
\subfloat[]{ 
\includegraphics[width=.45\textwidth]{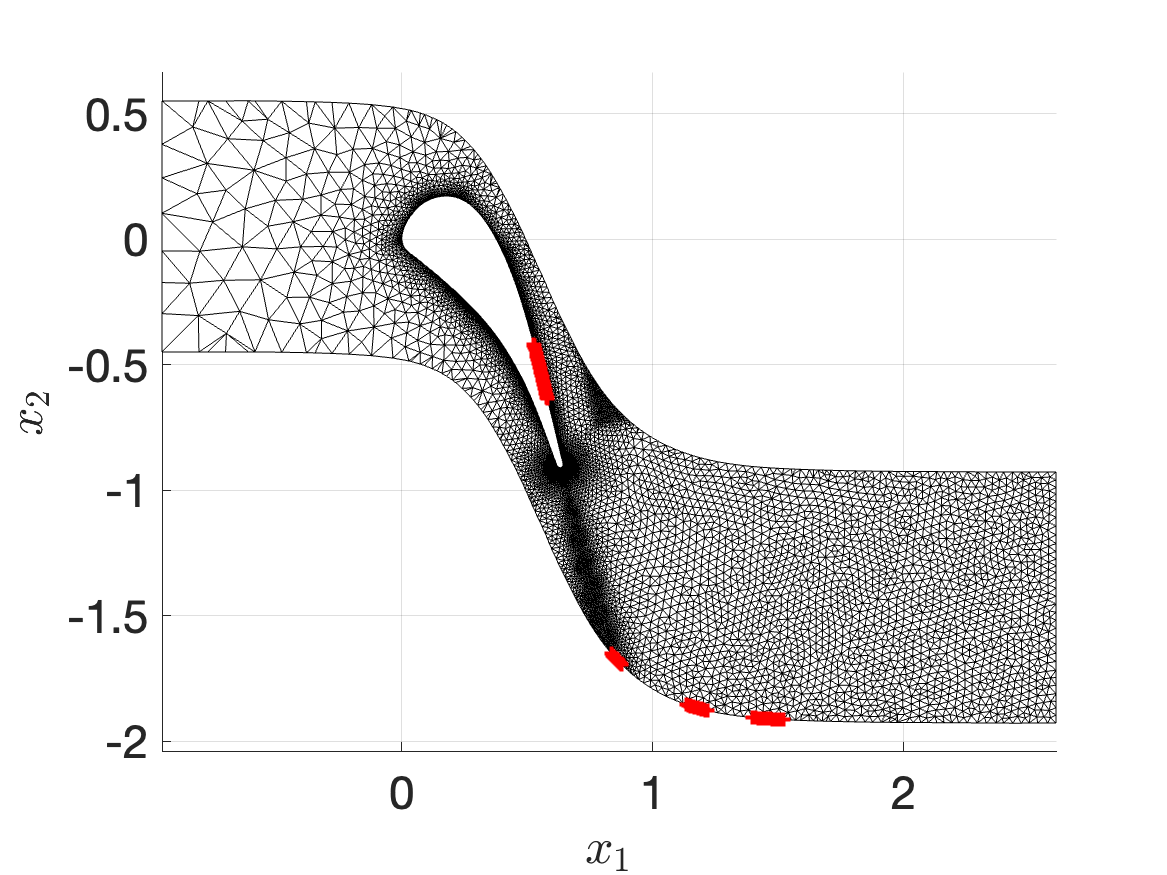}
}
~~
\subfloat[]{
\includegraphics[width=.45\textwidth]{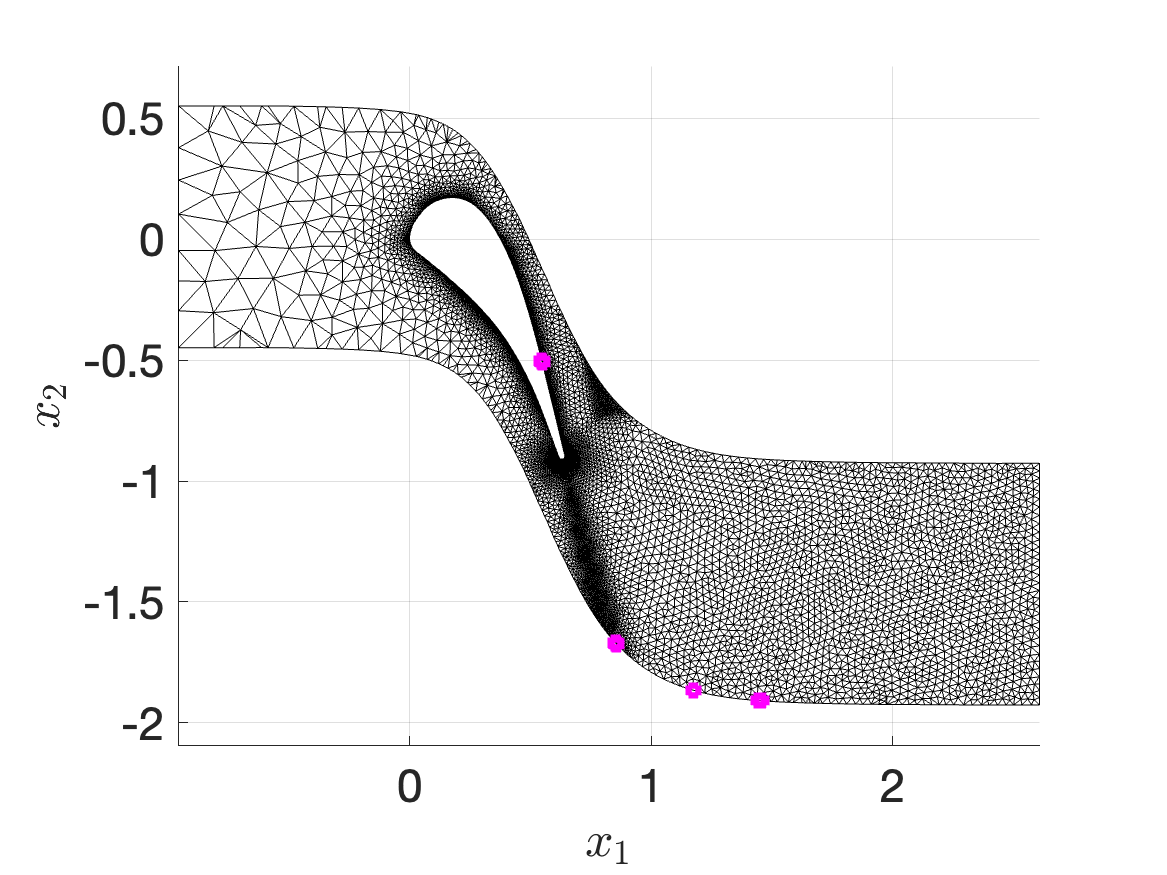}
}

\caption{inviscid flow past an array of LS89 turbine blades.
(a) sensor points used in the registration procedure for $n_{\rm train}=66$ configurations.
(b) mapped sensor points.
}
\label{fig:sensor_model_problem_salientpoints}
\end{figure} 

\begin{figure}[h!]
\centering
\subfloat[]{ 
\includegraphics[width=.47\textwidth]{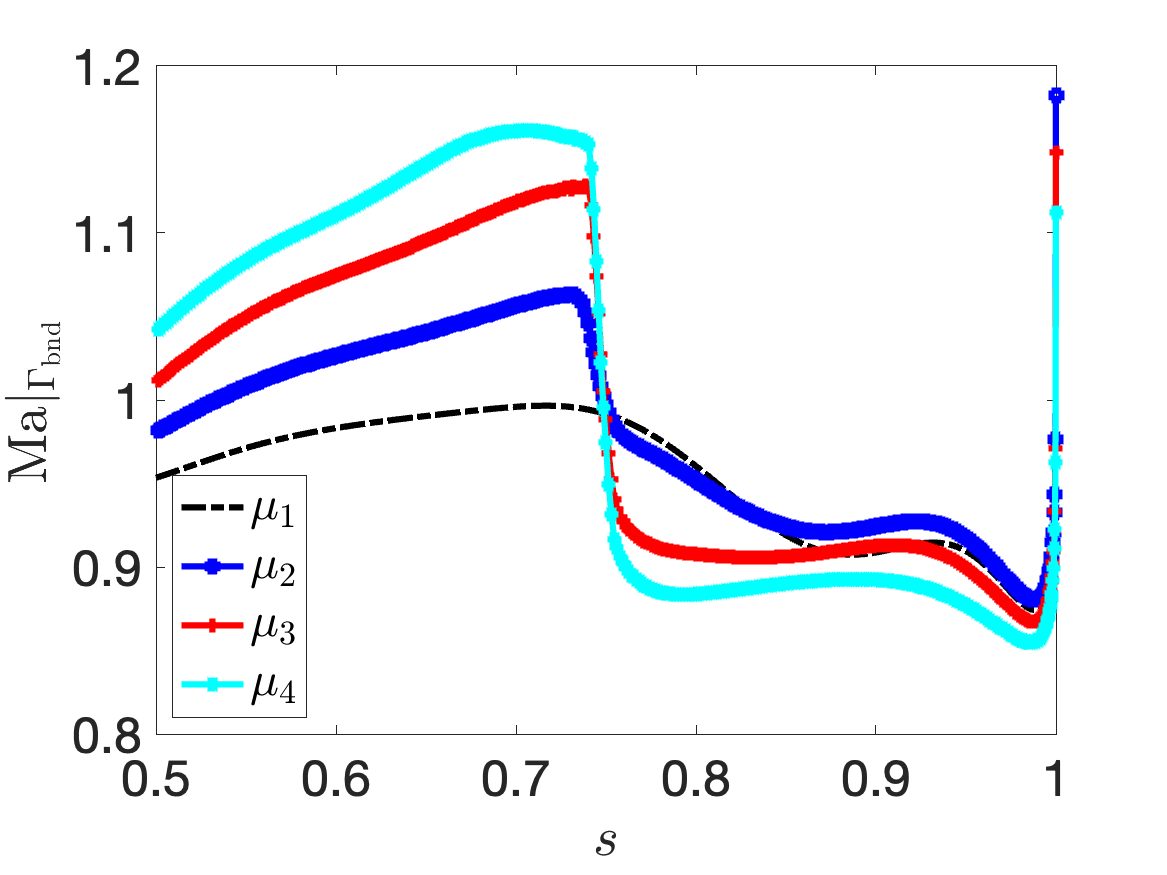}
}
~~
\subfloat[]{
\includegraphics[width=.47\textwidth]{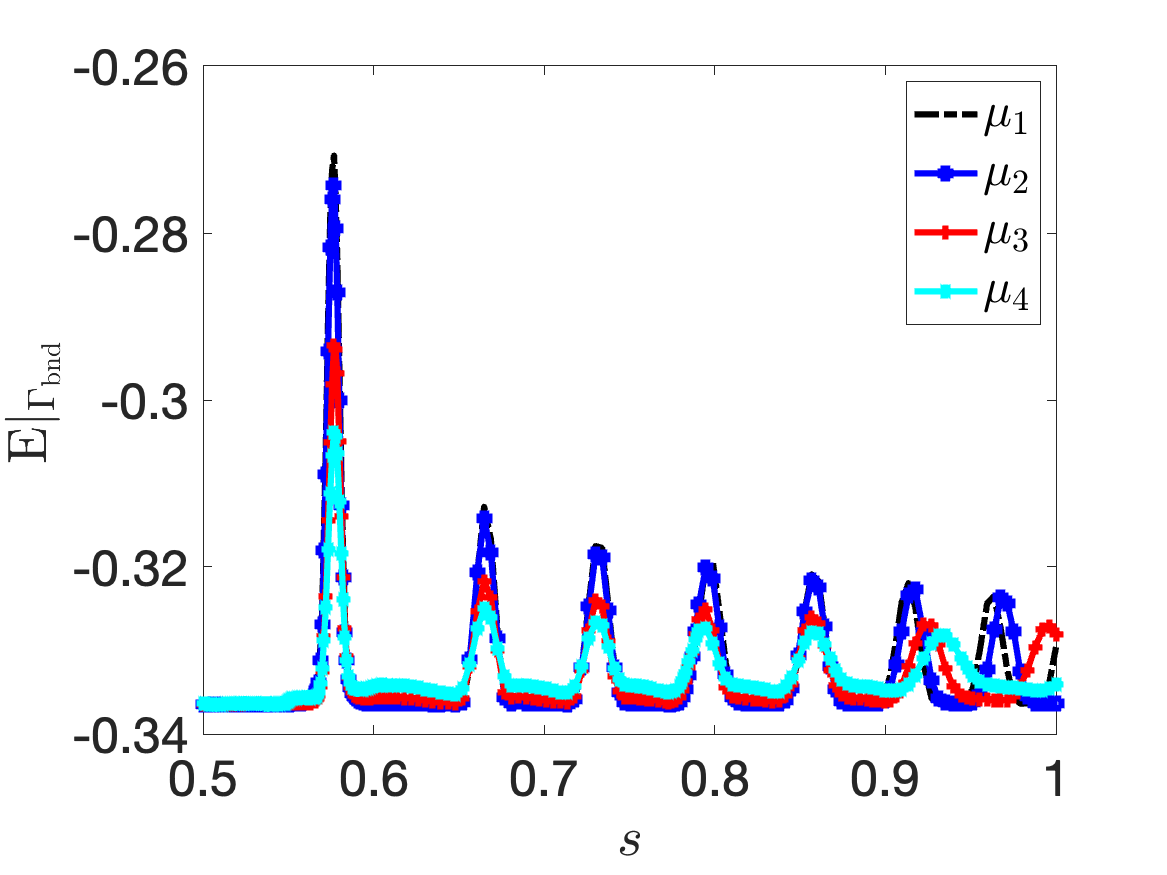}
}

\caption{inviscid flow past an array of LS89 turbine blades.
(a)  mapped Mach  profile over the upper side of the blade for four values of the parameter in $\mathcal{P}$.
(b)  mapped entropy  profile on the bottom boundary   for four values of the parameter in $\mathcal{P}$.
Unmapped profiles are provided in Figure \ref{fig:vis_model_problem_slices}.
}
\label{fig:vis_model_problem_slices_ref}
\end{figure} 

\subsection{Model reduction}
We rely on POD+regression to estimate the solution field. POD solution coefficients are estimated using the same RBF regression method that is employed for the mapping coefficients (cf. Remark \ref{remark:generalization}). 
We denote by 
$u_{\mu}^{\rm hf}$ the HF solution, and by $\widehat{u}_{\mu}$ the POD+RBF estimate. To assess performance, we consider a dataset of $n_{\rm test}=20$ randomly-selected parameters $\mathcal{P}_{\rm test}$ and we compute the $L^2$ maximum relative errors
\begin{equation}
\label{eq:Emax}
E_{\rm max} = 
\max_{\mu\in \mathcal{P}_{\rm test}} \frac{
\| \widehat{u}_{\mu} -   {u}_{\mu}^{\rm hf} \|_{L^2(\Omega_{\mu})}}{\|   {u}_{\mu}^{\rm hf} \|_{L^2(\Omega_{\mu})}},
\quad
E_{\rm max}^{\rm bnd} = 
\max_{\mu\in \mathcal{P}_{\rm test}} 
\frac{
\| \widehat{u}_{\mu} -   {u}_{\mu}^{\rm hf} \|_{L^2(\Gamma_{\rm bld})}}{\|   {u}_{\mu}^{\rm hf} \|_{L^2(\Gamma_{\rm bld})}},
\end{equation}
where $\Gamma_{\rm bld}$ denotes the boundary of the blade.
Figure \ref{fig:performance_nonlinear} shows the results for the ``linear'' ROM and the ``registered'' ROM --- both approaches involve the application of a geometric mapping; the difference is that for the linear ROM the mapping is chosen a priori and corresponds to $\Psi^{\rm geo}$ \eqref{eq:geometric_mapping_turbine}, while for the registered ROM the mapping is chosen through the registration procedure discussed in the paper.
Both linear and registered ROMs reach a plateau for $n\approx 10$ due to the limited amount of datapoints.
We observe that   registration improves performance by a factor  $2.04$ for the global error and a factor $8.04$ for the error on the blade, for the same amount of HF data.

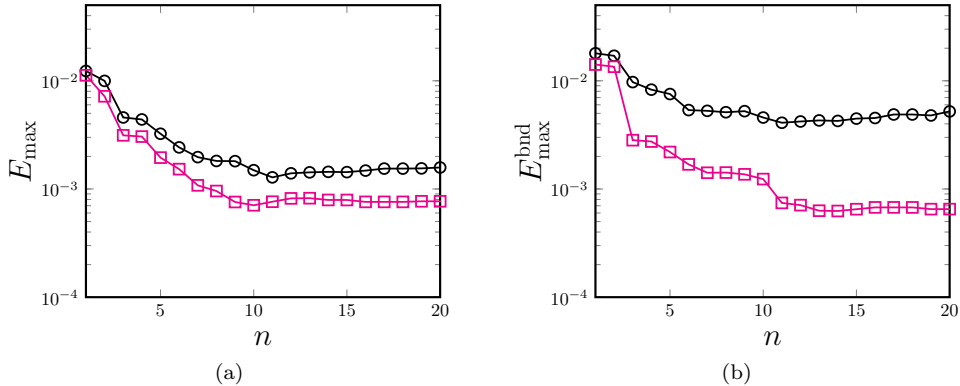
\begin{figure}[h!]
\centering

\subfloat[] {
\begin{tikzpicture}[scale=0.68]
\begin{semilogyaxis}[
xlabel = {\LARGE {$n$}},
  ylabel = {\LARGE {$E_{\rm max}$}},
  line width=1.2pt,
  mark size=3.0pt,
  ymin=0.0001,   ymax=0.05,
xmin=1,   xmax=20,
  ]
  
\addplot[line width=1.pt,color=black,mark=o] table {data/ntrain66/linear.dat};
\label{ROMperformance:linear}
     \addplot[line width=1.pt,color=magenta,mark=square]  table {data/ntrain66/registered.dat};
    \label{ROMperformance:reg};
    
\end{semilogyaxis}

\end{tikzpicture}
}
~~~
\subfloat[] {
\begin{tikzpicture}[scale=0.68]
\begin{semilogyaxis}[
xlabel = {\LARGE {$n$}},
ylabel = {\LARGE {$E_{\rm max}^{\rm bnd}$}},
  line width=1.2pt,
  mark size=3.0pt,
  ymin=0.0001,   ymax=0.05,
xmin=1,   xmax=20,
  ]
  
\addplot[line width=1.pt,color=black,mark=o] table {data/ntrain66/linear_bnd.dat};
\addplot[line width=1.pt,color=magenta,mark=square]  table {data/ntrain66/registered_bnd.dat};
\end{semilogyaxis}

\end{tikzpicture}
}

\bgroup
\sbox0{\ref{data}}%
\pgfmathparse{\ht0/1ex}%
\xdef\refsize{\pgfmathresult ex}%
\egroup

\caption[Caption in ToC]{
 inviscid flow past an array of LS89 turbine blades; model reduction.
 Behavior of the out-of-sample errors
 \eqref{eq:Emax} for the 
linear ROM \tikzref{ROMperformance:linear} and the 
registered ROM \tikzref{ROMperformance:reg} for $n_{\rm train}=66$ and several values of $n$.
}
 \label{fig:performance_nonlinear}
\end{figure}

\section{Summary and discussion}
\label{sec:conclusions}

We addressed the problem of parametric registration for manifolds associated to the solution to parametric PDEs: registration is designed to track coherent structures of the solution field, and ultimately simplify the task of  MOR.
We proposed and analyzed
a new class of compositional maps \eqref{eq:nonlinear_ansatz} for registration in two-dimensional domains. We provided   
sufficient conditions   to ensure the bijectivity of the mapping $\Phi$ (cf. Proposition \ref{th:bijectivity} and Corollary \ref{th:easy_bijectivity_consequence}) and we studied the approximation properties: we found that the ansatz \eqref{eq:nonlinear_ansatz} is dense in a meaningful subspace of diffeomorphisms only for polytopal domains
(Corollary \ref{th:easy_density_consequence}, Lemma \ref{th:density_composition}); on the other hand, the multi-layer generalization \eqref{eq:nonlinear_ansatz_generalized} is dense for arbitrary domains under the assumption of small deformation (cf. Lemma \ref{th:density_multilayer_composition}).
We discussed an actionable procedure to determine the ansatz \eqref{eq:nonlinear_ansatz} for arbitrary domains, which exploits a coarse-grained high-order FE mesh;
finally, we illustrated the performance of the method for a parametric compressible flow past a cascade of turbine blades.

We aim to extend our approach in several directions.
First, we aim to extend Proposition \ref{th:bijectivity} to (a class of) Lipschitz maps. Second, we wish to automate the construction of the polytope $\Omega_{\rm p}$; we also wish to devise effective implementations of the ansatz \eqref{eq:nonlinear_ansatz_generalized}.
Third, we wish to extend our method to three-dimensional domains: this extension   likely requires the development of specialized optimization routines to cope with high-dimensional mapping spaces (i.e., large $M$), and efficient mesh interpolation routines to speed up the evaluation of the objective function in \eqref{eq:tractable_optimization_based_registration}.

As extensively discussed in section \ref{sec:discussion_curved}, our choice of the ansatz   \eqref{eq:nonlinear_ansatz} is a compromise between approximation power and simplicity of implementation. 
In the future, we aim to investigate generalizations of 
\eqref{eq:nonlinear_ansatz}  to improve the approximation power of our model class; in this respect, we plan to develop registration techniques based on the ansatz   \eqref{eq:nonlinear_ansatz_generalized}.

Finally, we remark that registration procedures constitute a key elements of MOR techniques based on Lagrangian approaches.
The results of this paper contribute to motivate the interest in this class of methods for problems with coherent localized structures. Nevertheless, their inadequacy to cope with topology changes (e.g., two shocks that merge into one shock
\cite[Appendix E.1]{taddei2021space}) might  justify the use of different nonlinear proposals
(see, e.g., \cite{barnett2023neural,black2021efficient}).
It is hence of paramount importance to perform thorough comparative studies to study the strengths and the weaknesses of Lagrangian methods and to compare performance with other nonlinear approaches.
 
\section*{Acknowledgments}
The author  acknowledges the support by European Union’s Horizon 2020 research and innovation programme under the Marie Sklodowska-Curie Actions, grant agreement 872442 (ARIA), and the support provided by Inria under the programme 
Exploratory Actions (project AM2OR).
The author  also  thanks  
Prof. Angelo Iollo and 
Dr. Pierre Mounoud (Univ. Bordeaux),  
Dr. Alberto Remigi  (Safran Tech), and Prof. Masayuki Yano (University of Toronto) for fruitful discussions on various aspects of the method.
The author also thanks three anonymous reviewers for their many comments that help improve the quality of the manuscript.

 \appendix

\section{Proof of  Proposition \ref{th:bijectivity}}
\label{sec:proofs}
The key challenge of the proof is that the domain $\Omega_{\rm p}$ is not in general  simply connected: it is hence not possible to directly employ the result in 
 \cite{taddei2020registration} based on Hadamard's global inverse function theorem. To address this issue,
(i) we first prove a variant of 
 Proposition 2.1 in \cite{taddei2020registration} (cf. 
 Proposition \ref{th:mapping_general}) that addresses the case of simply connected domains (cf. Proposition \ref{th:mapping_general});
 (ii) we prove that the restriction of $\Phi$ to $\partial \Omega_{\rm p}$ is a diffeomorphism (cf. Lemma \ref{th:lemma_easy});
 (iii) we prove a technical result  to glue together $C^1$ maps with positive Jacobian determinants
 (cf. Lemma \ref{th:bijectivity_fricking_mess});
 (iv) we exploit the previous results and a well-known extension  
  theorem in differential geometry to prove the desired result. Below, we present a number of preliminary results that address (i), (ii) and (iii) and then we discuss the proof of Proposition \ref{th:bijectivity}.
 
\begin{proposition}
\label{th:mapping_general}
Let $\Omega_{\rm p} \subset \mathbb{R}^2$  be a polytope isomorphic to the unit ball.
Consider the vector-valued function
 ${\Phi}:  \Omega_{\rm p}  \to \mathbb{R}^2$   that satisfies 
(i)
$\Phi \in C^1 \left(  \overline{\Omega}_{\rm p} ; \mathbb{R}^2 \right)$;
(ii)
$| J({\Phi})(x) |= \big|  {\rm det} \left(  \nabla \Phi (x)  \right)  \big|  \geq \epsilon >0$  for all $x \in \Omega_{\rm p}$ and a given  $\epsilon>0$;
(iii)
${\Phi}(\partial  
 \Omega_{\rm p})  \subseteq  \partial \Omega_{\rm p}$.
Then, ${\Phi}$ is a bijection that maps  $ \Omega_{\rm p}$ into $\Omega_{\rm p}$.
\end{proposition}
 
\begin{proof}
 This result is a variant   of  Proposition 2.1 in \cite{taddei2020registration} for  polytopes $\Omega_{\rm p}$ isomorphic to the unit ball;
 the proof follows the exact same steps of the one in 
\cite{taddei2020registration}; details are omitted.
In \cite{taddei2020registration} we assumed that 
$\Phi$ is of class $C^1$ in a $\delta$-neighborhood of $\Omega_{\rm p}$; since
 $\Omega_{\rm p}$ is a Lipschitz domain,
 if $\Phi$ is of class $C^1$ up to the boundary, then there exists a $C^1$ extension of $\Phi$ to $\mathbb{R}^2$ (cf.
 \cite[section 2.5]{brudnyi2011methods}): we can thus replace the hypothesis (i)
of \cite[Proposition 2.1]{taddei2020registration} with the   hypothesis $ {\Phi} \in C^1 \left(  \overline{\Omega}_{\rm p} ; \mathbb{R}^2 \right)$. 
Furthermore, we replace 
$J({\Phi})(x)   \geq \epsilon >0$ for all $x\in \Omega_{\rm p}$  with the more general assumption
$| J({\Phi})(x) |  \geq \epsilon >0$:
since $ J({\Phi})$ is continuous, this condition means that 
$ J({\Phi})$ is either strictly positive or strictly negative in $\Omega_{\rm p}$.
 \end{proof}
 
\begin{lemma}
\label{th:lemma_easy}
Let $\Phi = \texttt{id} + \varphi$ be a
$C^1(\overline{\Omega}_{\rm p}; \mathbb{R}^2)$ field   such that  $\inf_{x \in \Omega_{\rm p}} J(\Phi) > 0$ and
$\varphi \cdot \mathbf{n}|_{\partial \Omega_{\rm p}} = 0$.
Then,  
 $\Phi (\partial \Omega_{\rm p}) = \partial \Omega_{\rm p}$,   and  $\Phi|_{\partial \Omega_{\rm p}}$ is a diffeomorphism in  $\partial \Omega_{\rm p}$.
\end{lemma}
\begin{proof} 
It suffices to prove that  each edge $F$ is mapped in itself  and  that the restriction of $\Phi$ to $F$
is a bijection.
Towards this end, we define the vertices $x_1^{\rm v},x_2^{\rm v}$ of $F$, the tangent and normal vectors $\mathbf{t}_{\rm f}, \mathbf{n}_{\rm f}$ and the parameterization $\gamma_{\rm f}:[0,1]\to F$ such that 
$\gamma_{\rm f}(s) = x_1^{\rm v} + s \|x_2^{\rm v} -x_1^{\rm v}   \|_2 \mathbf{t}_{\rm f}$. We also introduce the normal vectors $\mathbf{n}_{\rm f}^-, \mathbf{n}_{\rm f}^+$ of the neighboring edges that meet $F$ at $x_1^{\rm v},x_2^{\rm v}$, respectively.

 We first observe that $\varphi (x_1^{\rm v}) = \varphi(x_2^{\rm v})= 0$:
since $F$ meets with the neighboring edges, the pairs of normals 
  $\mathbf{n}_{\rm f}^-,\mathbf{n}_{\rm f}$  and 
   $\mathbf{n}_{\rm f}^+,\mathbf{n}_{\rm f}$ should be linearly independent (i.e., they should span $\mathbb{R}^2$);  then, exploiting the continuity of $\varphi$ up to the boundary of $ \Omega_{\rm p}$ --- and in particular on $\partial \Omega_{\rm p}$ --- we find
 $\varphi(x_1^{\rm v})\cdot \mathbf{n}_{\rm f}^- = \varphi(x_1^{\rm v})\cdot \mathbf{n}_{\rm f} = 0$
and
 $\varphi(x_2^{\rm v})\cdot \mathbf{n}_{\rm f}  = \varphi(x_2^{\rm v})\cdot \mathbf{n}_{\rm f}^+ = 0$, which imply 
  $\varphi(x_1^{\rm v}) = \varphi(x_2^{\rm v})= 0$.
 
 By contradiction, $\Phi(F) \not\subset F$. 
Given $x\in F$, we find
 $$
 \varphi(x) = (\varphi(x)\cdot  \mathbf{n}_{\rm f}) \mathbf{n}_{\rm f} + (\varphi(x)\cdot  \mathbf{t}_{\rm f}) \mathbf{t}_{\rm f} =  (\varphi(x)\cdot  \mathbf{t}_{\rm f}) \mathbf{t}_{\rm f}
 $$
  and thus 
$
\Phi(\gamma_{\rm f}(s)) = x_1^{\rm v} + \alpha(s) 
 \mathbf{t}_{\rm f}$ with 
$\alpha(s) = 
(\Phi(\gamma_{\rm f}(s))  -x_1^{\rm v} )\cdot \mathbf{t}_{\rm f}.$
Since $\varphi(x_1^{\rm v}) = \varphi(x_2^{\rm v})= 0$, we find $\alpha(0) = 0, \alpha(1)=
\| x_2^{\rm v}  - x_1^{\rm v}   \|_2$:
 exploiting the intermediate value theorem, we   find that 
$[0,1]  \subset \alpha([0,1]) $ and thus
 $F \subset \Phi(F)$. 
The condition
 $\Phi(F) \not\subset F$ implies  that $\alpha(s)\notin [0,1]$ for some $s\in [0,1]$ and thus 
  $\alpha'(s^{\star}) = 
  \lVert x_2^\mathrm{v}-x_1^\mathrm{v}\rVert_2
  \mathbf{t}_{\rm f}^T \nabla \Phi(\gamma_{\rm f}(s^{\star})) \mathbf{t}_{\rm f} = 0$ for some $s^{\star}\in [0,1]$. 
  Note that $\Phi(\gamma_{\rm f}(s)) \cdot \mathbf{n}_{\rm f} = x_1^{\rm v} \cdot  \mathbf{n}_{\rm f}$ for any $s\in [0,1]$; if we differentiate left-  and right-hand sides of the previous identity with respect to $s$ we obtain $ 
  \lVert x_2^\mathrm{v}-x_1^\mathrm{v}\rVert_2  
  \mathbf{n}_{\rm f}^\top \nabla \Phi(\gamma_{\rm f}(s) )  \mathbf{t}_{\rm f} = 0$ for all $s\in [0,1]$. Therefore,  if we set $x^{\star} = \gamma_{\rm f}(s^{\star})$, we find
  $
  \nabla \Phi( x^{\star} ) \mathbf{t}_{\rm f} 
=  
\left(  \mathbf{t}_{\rm f}^\top   \nabla 
\Phi( x^{\star} ) \mathbf{t}_{\rm f} \right) \mathbf{t}_{\rm f}
+
\left(  \mathbf{n}_{\rm f}^T   \nabla 
\Phi(  x^{\star} ) \mathbf{t}_{\rm f} \right) \mathbf{n}_{\rm f}
= 0, 
$
which implies that  $\nabla \Phi$ is singular at 
$x^{\star}$
and thus
$\min_{x\in \overline{\Omega}_{\rm p}} J(\Phi) \leq  J(\Phi)(x^{\star})=0$. Contradiction.

So far, we proved that the $C^1$ function $\alpha:[0,1]\to\mathbb{R}$ is surjective in  $[0,1]$ and is monotonic increasing. Therefore, $\alpha$ is bijective in  $[0,1]$ and thus $\Phi|_F$ is  a bijection in $F$.
\end{proof}

\begin{lemma}
\label{th:bijectivity_fricking_mess}
Let  $\Omega = \Omega_{\rm ext} \setminus \bigcup_{i=1}^N \Omega_{{\rm int},i}$ be a polytope that satisfies the conditions of Definition 
\ref{def:regular_polytopes}.
We define $\Omega_{\rm int}=\bigcup_{i=1}^N \Omega_{{\rm int},i}$.
Let $\Phi^+: \Omega\to \mathbb{R}^2$ and $\Phi^-: \Omega_{\rm int} \to \mathbb{R}^2$ be  $C^1$ maps up to the boundary that coincide on $\Gamma:= \partial \Omega_{\rm int}$ and satisfy 
$J(\Phi^+) (x) \geq \epsilon>0$ for all $x\in \Omega$,
$J(\Phi^-) (x) \geq \epsilon>0$ for all $x\in \Omega_{\rm int}$. 
We define
$\Phi :\Omega_{\rm ext} \to \mathbb{R}^2$ 
that 
is equal to $\Phi^+$ in $\Omega$ and is equal to
$\Phi^-$ in $\Omega_{\rm int}$.
Then, there exists a sequence 
 $\{ \Phi_j   \}_j \subset C^1(\overline{\Omega}; \mathbb{R}^2)$ such that
 (i) $\Phi_j = \Phi$ on  $\partial \Omega_{\rm ext}$,
 (ii) $\lim_{j\to \infty} \| \Phi_j - \Phi  \|_{L^{\infty}(\Omega_{\rm ext})} = 0$,
 (iii) $\min_{x\in \overline{\Omega}} J(\Phi_j) \geq \frac{1}{2}\epsilon$ for $j=1,2,\ldots$.
 \end{lemma}
\begin{proof}
We assume that $\Omega_{\rm int}$ is isomorphic to the unit ball (that is,
$\Omega_{{\rm int}}=\Omega_{{\rm int},1}$ and $N=1$); the proof can be extended to $N>1$  through an inductive process, iteratively incorporating one subdomain at each step.
 We define the piecewise-linear 
 parameterization $\boldsymbol{\gamma}  :(0,1) \to \Gamma$;
 we 
 denote by $\mathbf{t}$ the tangent vector to $\Gamma$,  by $\mathbf{n}$ the normal vector that points towards $\Omega$, and by
 $V_{\rm int}$ the vertices of $\Omega$ on $\Gamma$;
 ${\rm dist}_{\rm H}(A,B)$ is the Hausdorff distance between the sets $A$ and $B$.
  Given $\delta>0$, we define the set
  $V_{\rm int,\delta} = \bigcup_{x\in V_{\rm int}} \mathcal{B}_{\delta}(x)$, and    
  the $C^1$ approximation $\Gamma_{\delta}$ of $\Gamma$ such that
 ${\rm dist}_{\rm H}(\Gamma_{\delta} ,\Gamma ) \leq \delta$ and
  ${\rm dist}_{\rm H}(\Gamma_{\delta} \setminus  V_{\rm int,\delta} ,\Gamma \setminus  V_{\rm int,\delta}) = 0$. Then, we  
introduce the $C^1$  parameterization $\boldsymbol{\gamma}_{\delta} :(0,1) \to \Gamma_{\delta} $, the corresponding unit normal
$\mathbf{n}_{\delta}$ and the set
$
\Gamma_{\eta, \delta} =
\left\{
\boldsymbol{\gamma}_{\delta}(u) + \eta v \mathbf{n}_{\delta}(\boldsymbol{\gamma}_{\delta}(u) ) \,:\,
v\in \left(
- f(u), 1 - f(u) \right),
\;\;
u\in (0,1)
\right\}$,
 where $f:(0,1) \to [0,1]$ is  a suitable 
smooth
 function that will be introduced below.
Figure \ref{fig:messy_lemma} provides a graphic explanation of the quantities introduced so far. 
 
 Exploiting the previous definitions, we define
$\psi_{\eta,\delta}: \mathbb{R}^2\to [0,1]$ such that
(i)
 $\psi_{\eta,\delta} (  
 \boldsymbol{\gamma}_{\delta}(u) + \eta v \mathbf{n}_{\delta}      )
 =
 \phi( v + f(u) )$
 for all   $u \in (0,1),
 v \in 
 \left(
- f(u), 1 - f(u) 
\right)
$,
 with $\phi(t) =  - 2 t^3 + 3 t^2$, 
 (ii)
  $ \psi_{\eta,\delta}|_{\Omega_{\rm ext} \setminus \Gamma_{\eta,\delta}} = 1$, and
(iii)
 $ \psi_{\eta,\delta}|_{\Omega_{\rm int} \setminus \Gamma_{\eta,\delta}} = 0$.
 It is straightforward to verify  that 
 $\psi_{\eta,\delta}$ is of class $C^1$, and that
 $$
 \nabla \psi_{\eta,\delta}(x)
=
\left\{
\begin{array}{ll}
\displaystyle{
\frac{1}{\eta} \left(
 \alpha_{\psi,\delta}(x)  \mathbf{n} + o(1)
\right) 
}
&
\displaystyle{
\forall \, x\in 
\Gamma_{\eta,\delta} 
 \setminus V_{
 \rm int, \delta},
}
\\[3mm]
\displaystyle{
\mathcal{O} \left(
\frac{ \alpha_{\psi,\delta}(x) }{\eta}
 \mathbf{n}_{\delta} + 
\frac{ \beta_{\psi,\delta}(x) }{\delta}
 \mathbf{t}_{\delta} 
\right) 
}
&
\displaystyle{
\forall \, x\in 
  V_{
 \rm int, \delta},
}
\\
\end{array}
\right.
 $$
where  $\alpha_{\psi,\delta}(x) \in (0,\max_{v\in \mathbb{R}} \phi'(v) )= (0,3/2)$ and $\beta_{\psi,\delta}(x)$ belongs to a bounded interval that is independent of  $\delta>0$.  In the previous expression, we used the fact that
 $\textbf{n}_\delta = \textbf{n}$ for 
 $x\in 
\Gamma_{\eta,\delta} 
 \setminus V_{
 \rm int, \delta}$, since ${\rm dist}_{\rm H}(\Gamma_{\delta} \setminus  V_{\rm int,\delta} ,\Gamma \setminus  V_{\rm int,\delta}) = 0$.
 
 We define 
 the positive vanishing sequences $\{   \eta_j \}_j, \{   \delta_j \}_j \subset \mathbb{R}_+$; for $j=1,2,\ldots$, we introduce
the function  $\psi_j =  \psi_{\eta_j ,\delta_j }$ 
 and 
  the field
 $ \Phi_j = \left( \Phi_{\rm e}^+ - \Phi_{\rm e}^- \right)\psi_j + \Phi_{\rm e}^-$, where  
 $ \Phi_{\rm e}^+, \Phi_{\rm e}^-$ are $C^1$ extensions
 of   $ \Phi^+, \Phi^-$ to $\mathbb{R}^2$
 (cf. \cite[section 2.5]{brudnyi2011methods}).
Clearly,  $\{ \Phi_j \}_j$ satisfies the conditions (i)
 and (ii);
furthermore,  $J( \Phi_j) \geq \epsilon$ in $\Omega_{\rm ext}\setminus \Gamma_{\eta_j,\delta_j} $;
   in the remainder, we prove that  
  $J( \Phi_j) \geq \epsilon/2$ in $\Gamma_{\eta_j,\delta_j} $.
 
 Let $x\in \Gamma_{\eta_j,\delta_j} \setminus V_{\rm int,\delta_j}$.
 Since $\Phi^+ = \Phi^-$ on $\Gamma$, we have 
$\left( \nabla (\Phi_{\rm e}^+ - \Phi_{\rm e}^-) \right) \mathbf{t} = 0$ for all 
$x\in \Gamma$. We define
$\mathbf{t}_{\Phi} = 
\frac{\nabla  \Phi_{\rm e}^{+}  \mathbf{t} }{\|  \nabla  \Phi_{\rm e}^{+} \mathbf{t}  \|_2}$ and 
$\mathbf{n}_{\Phi} = [(\mathbf{t}_{\Phi})_2, - (\mathbf{t}_{\Phi})_1]$. 
Given $x\in \Gamma_{\eta_j,\delta_j}$, $x=\boldsymbol{\gamma}(u) + \eta_j v \mathbf{n}$, we define
$\bar{x} = \boldsymbol{\gamma}(u) $ and
$\bar{y} = \Phi_{\rm e}^+( \bar{x}    ) = \Phi_{\rm e}^-(\bar{x} )$.
Then,
$$
\begin{array}{rl}
\Phi_{\rm e}^{\pm}(x) \approx &
\displaystyle{
\bar{y}  + \nabla \Phi_{\rm e}^{\pm}
( \bar{x})
 (x-\bar{x})
}
\\[3mm]
=
&
\displaystyle{
\bar{y}  + \nabla \Phi_{\rm e}^{\pm} ( \bar{x})
\mathbf{t} ( \bar{x}) \mathbf{t}^{\top}( \bar{x})
(x-\bar{x}) 
+ \nabla \Phi_{\rm e}^{\pm} ( \bar{x})
\mathbf{n}( \bar{x}) \mathbf{n}^{\top}( \bar{x})
(x-\bar{x})
}
\\[3mm]
=
& 
\displaystyle{
\bar{y}  + 
\left(
\alpha ( \bar{x})
\mathbf{t}_{\Phi}( \bar{x}) \mathbf{t}^{\top}( \bar{x})
+
\beta^{\pm}( \bar{x})
\mathbf{t}_{\Phi} ( \bar{x}) \mathbf{n}^{\top}( \bar{x})
+
\gamma^{\pm}
\mathbf{n}_{\Phi} ( \bar{x}) \mathbf{n}^{\top}( \bar{x})
\right)
(x-\bar{x})
},
\\
\end{array}
$$
which implies 
  $ \nabla \Phi_{\rm e}^{\pm} = 
  \alpha 
\mathbf{t}_{\Phi} \mathbf{t}^{T}
+
\beta^{\pm}
\mathbf{t}_{\Phi} \mathbf{n}^{T}
+
\gamma^{\pm}
\mathbf{n}_{\Phi} \mathbf{n}^{T}$, where
$\alpha,\beta^{\pm}, \gamma^{\pm}, \mathbf{t}_{\Phi}$
depend on $\bar{x}$.
Exploiting the latter, we find
$$
\begin{array}{l}
\displaystyle{
\nabla
\Phi_{j}(x)
=
\left( \nabla \Phi_{\rm e}^+(x) - \nabla  \Phi_{\rm e}^-(x) \right) \psi_{j}(x)
+
\left(   \Phi_{\rm e}^+(x)  -    \Phi_{\rm e}^-(x) \right) ( \nabla \psi_{j}(x) )^\top
+
\nabla  \Phi_{\rm e}^-(x)
}
\\[3mm]
\approx
\displaystyle{
\alpha  (\bar{x})
\mathbf{t}_{\Phi}  (\bar{x})   \mathbf{t}^{\top} (\bar{x})
+
\left(
\beta^- (\bar{x})  +
\llbracket   \beta  (\bar{x})  \rrbracket
\psi_{j}
\right)
\mathbf{t}_{\Phi} (\bar{x})
 \mathbf{n}^{\top} (\bar{x}) +
 }
\\[3mm]
\displaystyle{
\left(
\gamma^-(\bar{x})
+
\llbracket   \gamma  (\bar{x})  \rrbracket
\psi_{j}(\bar{x})
\right)
\mathbf{n}_{\Phi}(\bar{x}) \mathbf{n}^{\top}(\bar{x})
+
v \alpha_{\psi} (\bar{x})
\left(
\llbracket   \beta (\bar{x})  \rrbracket  \mathbf{t}_{\Phi} (\bar{x})
+
\llbracket   \gamma (\bar{x})  \rrbracket  \mathbf{n}_{\Phi} (\bar{x})
\right)
\mathbf{n}^\top (\bar{x})
},
\\
\end{array}
$$
for all $x\in \Gamma_{\eta}$. Since 
$\{ \mathbf{t}, \mathbf{n} \}$ and 
$\{ \mathbf{t}_{\Phi}, \mathbf{n}_{\Phi} \}$ are orthonormal bases of $\mathbb{R}^2$, we must have
$$
J(\Phi_{j})
=
\alpha
\left(
\gamma^-
+
\llbracket   \gamma  \rrbracket  
\left(
\psi_{j} + 
v \alpha_{\psi}
\right)
\right) + o(1).
$$
Let  $\alpha$ be positive;
the case $\alpha<0$ is analogous.
If $\llbracket   \gamma  \rrbracket>0$, we find
$$
\gamma^-
+
\llbracket   \gamma  \rrbracket  
\Big(
\underbrace{\psi_{\eta}}_{\geq 0}
 + 
 \underbrace{v}_{\geq -f(u)}
  \underbrace{\alpha_{\psi}}_{\leq 3/2}
\big)
\geq
\gamma - \frac{3}{2}  \llbracket   \gamma  \rrbracket   f(u),
$$
which exceeds $\gamma^-/2$ if $f(u) \leq \frac{\gamma^-}{3  \llbracket   \gamma  \rrbracket  }$.
If $\llbracket   \gamma  \rrbracket < 0$, we 
find that 
$J(\Phi_{\eta})\geq \alpha \frac{\gamma_+}{2} $ provided that 
$f(u) \geq 1 -  \frac{\gamma^+}{3 | \llbracket   \gamma  \rrbracket |  }$. Note that for $ \llbracket   \gamma  \rrbracket =0$ the bounds reduce to $-\infty \leq f(u) \leq \infty$: we can hence find a smooth function $f$ such that 
$J(\Phi_{j})
\geq 
 \frac{\alpha}{2} \min\{ \gamma^+, \gamma^-  \}
 \geq 
 \frac{\epsilon}{2}$.
 
 Let $x\in \Gamma_{\eta_j,\delta_j} \cap V_{\rm int,\delta_j}$;
 we assume  that $\eta_j =\mathcal{O}(\delta_j)$ for $j=1,2,\ldots$.
 Exploiting the same argument as before, we find
 $\nabla \Phi_{\rm e}^+=\nabla \Phi_{\rm e}^-$ for $\bar{x}\in  V_{\rm int}$,
 which implies that 
 $  \Phi_{\rm e}^+(x)=  \Phi_{\rm e}^-(x) + o(\delta)$ for any $x\in V_{\rm int}$. Therefore, by considering a Taylor expansion centered in 
 $x\in  V_{\rm int}$, we find
 $$
 \begin{array}{l}
  \nabla
\Phi_{j}(x)
=
\displaystyle{
\left( \nabla \Phi_{\rm e}^+(x) - \nabla  \Phi_{\rm e}^-(x) \right) \psi_{j}(x)
+
\left(   \Phi_{\rm e}^+(x)  -    \Phi_{\rm e}^-(x) \right) ( \nabla \psi_{j} (x)  )^\top
+
\nabla  \Phi_{\rm e}^- (x)
}
\\[3mm]
=
\displaystyle{
o(1)  \mathcal{O}(1) 
+
o(\delta) 
\mathcal{O} \left(
\frac{ \alpha_{\psi,\delta}(x) }{\eta}
 \mathbf{n}_{\delta}(x) + 
\frac{ \beta_{\psi,\delta}(x) }{\delta}
 \mathbf{t}_{\delta} (x)
\right) 
+
\nabla  \Phi_{\rm e}^- (\bar{x})
\overset{\rm (i)}{=}
\nabla  \Phi_{\rm e}^- (\bar{x})  + o(1)
}\\
 \end{array}
  $$
 which implies $J(  \Phi_{j})(x) \gtrsim \epsilon$.
 In (i), we exploited the hypothesis $\eta_j =\mathcal{O}(\delta_j)$ for $j=1,2,\ldots$.
\end{proof}

\begin{figure}[h!]
\centering
\subfloat[]{ 
\includegraphics[width=.47\textwidth]{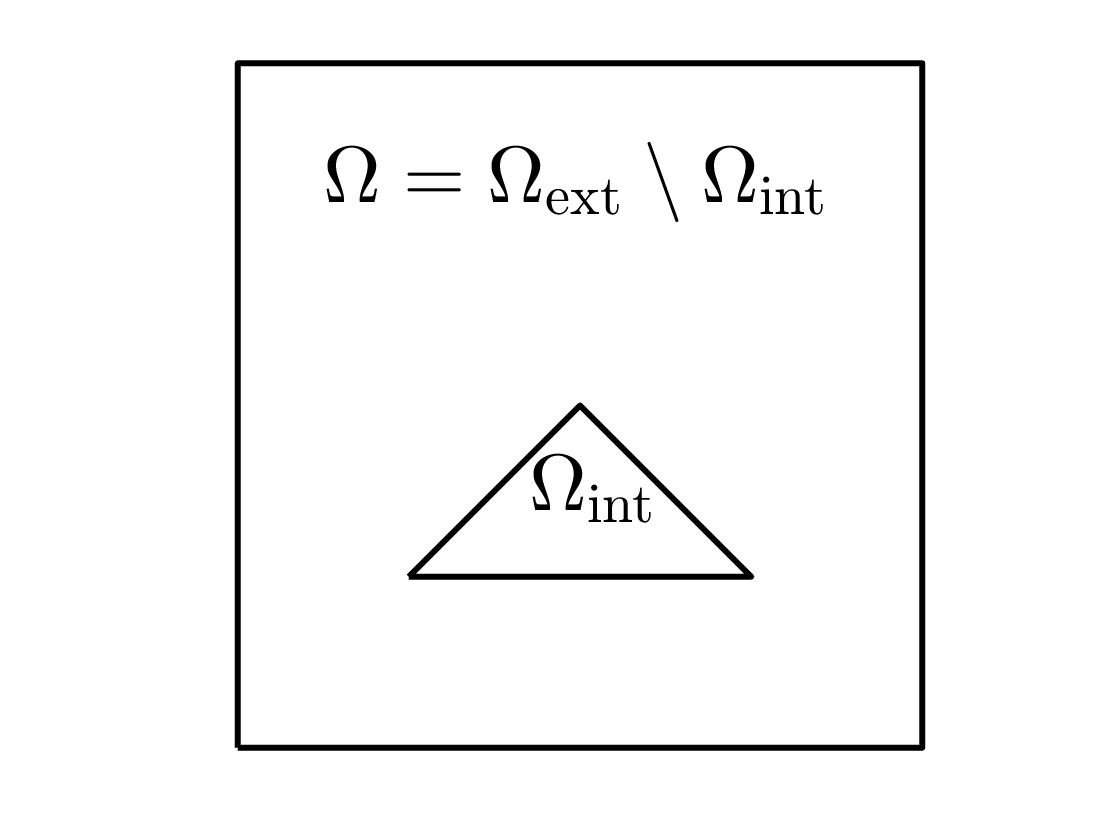}
}
~~
\subfloat[]{
\includegraphics[width=.47\textwidth]{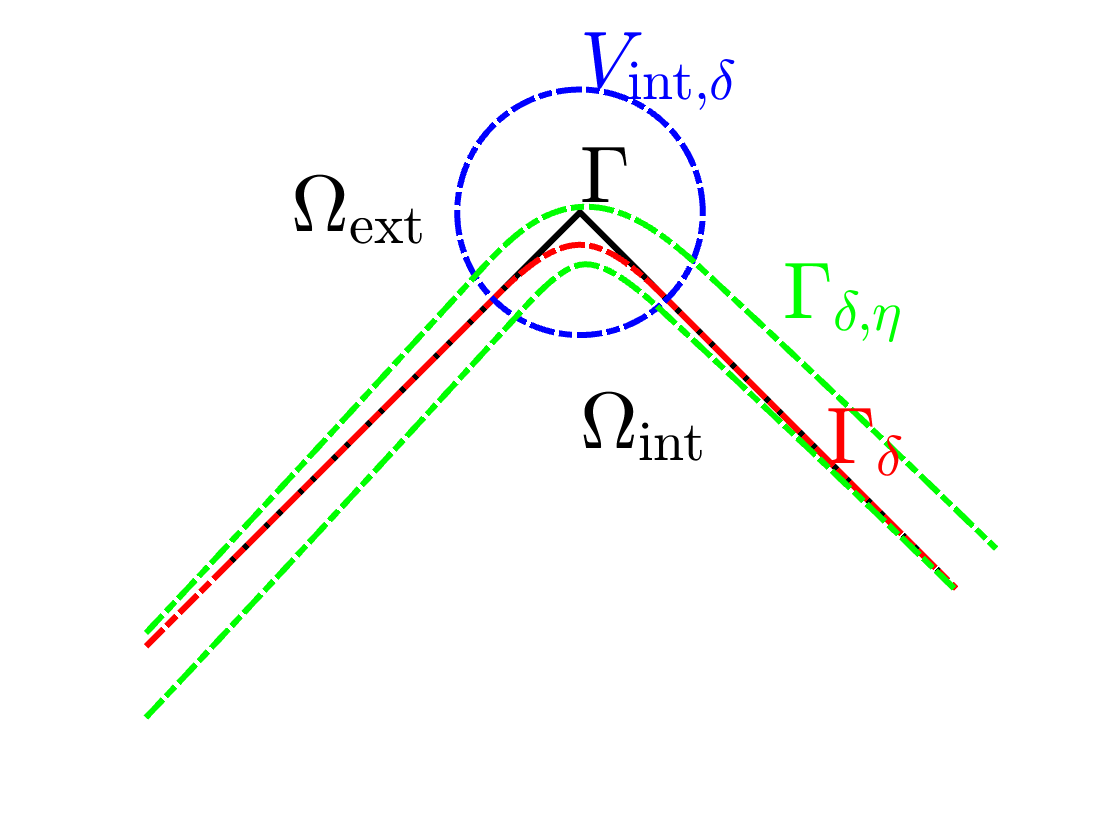}
}

\caption{proof of Lemma \ref{th:bijectivity_fricking_mess}. 
(a) example of regular polytope.
(b) visualization of $\Gamma$, 
$\Gamma_{\delta}$, 
$\Gamma_{\delta,\eta}$,
$V_{\rm int,\eta}$ in the proximity of one vertex.
}
\label{fig:messy_lemma}
\end{figure}

\begin{proof}
(\emph{Proposition \ref{th:bijectivity}}).
Exploiting Lemma \ref{th:lemma_easy}, we find that
$\Phi|_{\partial \Omega_{\rm int}}$ is a diffeomorphism in 
$\partial \Omega_{\rm int}$.
Therefore, exploiting Cerf's  theorem (see, e.g., \cite{laudenbach2023diffeomorphisms}),
 we find that, 
for $i=1,\ldots,N$,
 there exists a diffeomorphism
$\Psi_i:\Omega_{{\rm int},i} \to \Omega_{{\rm int},i}$ such that $\Psi_i=\Phi$ on $\partial \Omega_{{\rm int},i}$ and $| J(\Psi_i) |$ does not vanish in 
$\overline{\Omega}_{{\rm int},i}$.
If we define 
$\Omega_{\rm int}
= \bigcup_{i=1}^N
\Omega_{{\rm int},i}$ and $\Psi: \Omega_{\rm int} \to \Omega_{\rm int}$ such that
$\Psi|_{\partial \Omega_{{\rm int},i}} = \Psi_i$, we hence find that $\Psi$ is a diffeomorphism in $\Omega_{\rm int}$.
Exploiting the same argument as in Lemma \ref{th:bijectivity_fricking_mess}, we find that 
$\nabla \Psi = \nabla \Phi$ on the vertices of $\partial \Omega_{\rm int}$; therefore, 
$J(\Psi)$ is strictly positive in $\overline{\Omega}_{\rm int}$.
We set $\epsilon = \{ \inf_{x\in \Omega_{\rm p}} J(\Phi) ,
\inf_{x\in \Omega_{\rm int}} J(\Psi)   \}$.
We now define the extended map $\Phi_{\rm ext}: \Omega_{\rm ext}  \to \mathbb{R}^2$ that is equal to $\Phi$ in $\Omega_{\rm p}$ and is equal to $\Psi$ in $\Omega_{\rm int}$. Below, we prove that 
$\Phi_{\rm ext}$ is a bijection (injective and surjective) in $ \Omega_{\rm ext} $: since 
$\Psi$ is bijective in $\overline{\Omega}_{\rm int}$, this implies that 
$\Phi$ is bijective in  $\Omega_{\rm p}$.
Towards this end, we introduce a sequence  
$\{ \Phi_j   \}_j \subset C^1(\overline{\Omega}_{\rm ext}; \mathbb{R}^2)$ such that
 (i) $\Phi_j = \Phi$ on  $\partial \Omega_{\rm ext}$,
 (ii) $\lim_{j\to \infty} 
 \delta_j = 0$  with   
 $\delta_j =  \| \Phi_j - \Phi_{\rm ext}  \|_{L^{\infty}(\Omega_{\rm ext})}   =  0$,
 (iii) $\min_{x\in \overline{\Omega}_{\rm ext}} J(\Phi_j) \geq \frac{1}{2}\epsilon$ for $j=1,2,\ldots$
 (cf. Lemma \ref{th:bijectivity_fricking_mess}).
 For any $j>0$, $\Phi_j $ satisfies the hypotheses of Proposition \ref{th:mapping_general}: it is hence bijective in 
${\Omega}_{\rm ext}$.

\textbf{Injectivity of $\Phi_{\rm ext}$.}
By contradiction, assume that $\Phi_{\rm ext}$ is not injective, that is, there exist $x,y\in  {\Omega}_{\rm ext}$ such that
 $\Phi_{\rm ext}(x) = \Phi_{\rm ext}(y)$. 
 We distinguish between the case in which $x,y\in \partial \Omega_{\rm int}$ and 
the case in which $x \notin \partial \Omega_{\rm int}$ (the case $y \notin \partial \Omega_{\rm int}$ is analogous to the latter):
since $\Phi|_{\partial \Omega_{\rm int}}$ is a diffeomorphism in  
$\partial \Omega_{\rm int}$ (cf. Lemma \ref{th:lemma_easy}), the first scenario is impossible; we hence focus on the second case.

We denote  by $B_x$ the ball centered in $x$ of radius $\eta>0$.
Since $x \notin \partial \Omega_{\rm int}$,
there exists $\eta>0$ such that
$y \notin B_x$ and $\partial \Omega_{\rm int} \cap B_x = \emptyset$. 
Therefore, recalling notation introduced in Lemma \ref{th:bijectivity_fricking_mess},  there exists $j_0\in \mathbb{N}$ such that
$\Gamma_{\eta_j,\delta_j} \cap B_x = \emptyset$ for all $j>j_0$: this implies that, for any $j>j_0$, $\Phi_j(x')=\Phi(x')$ for any  $x' \in  B_x$;
furthermore, since $\Phi$ is a local diffeomorphism, there exist $c>0$ such that
$\| \Phi_j(x') - \Phi_j(x)  \|_2 = \| \Phi (x') - \Phi (x)  \|_2 \geq c \eta$ for 
any $x' \in \partial B_x$;
finally, since $\Phi_j$ is a global diffeomorphism, we also have 
$\| \Phi_j(x') - \Phi_j(x)  \|_2   \geq c \eta$
 for  any $x' \notin  B_x$.

Exploiting the previous identities and the reverse triangle inequality, we find
$$
\begin{array}{rl}
0 =  &
\| \Phi(y) - \Phi(x)  \|_2
=
 \| \Phi(y) - \Phi_j(x)  \|_2
=
 \| \Phi(y) - \Phi_j(y)  +\Phi_j(y) -  \Phi_j(x)  \|_2
\\[3mm]
\geq
&
 \| \Phi(y) - \Phi_j(y) \|_2 
-  
\| \Phi_j(y) -  \Phi_j(x)  \|_2
\\[3mm]
\Rightarrow &
\| \Phi_j(y) -  \Phi_j(x)  \|_2 \leq 
\| \Phi(y) - \Phi_j(y) \|_2. \\
\end{array}
$$
Since $\| \Phi_j(y) -  \Phi_j(x)  \|_2 \geq c \eta$ and
$\| \Phi(y) - \Phi_j(y) \|_2 \leq \delta_j$, we  find $c \eta \leq  \delta_j$ for any $j>j_0$,
 and thus $c \eta = 0$. Contradiction.

\textbf{Surjectivity of $\Phi_{\rm ext}$.}
In order to prove that 
$\Phi_{\rm ext}$ is surjective  in ${\Omega}_{\rm ext}$, 
that is, 
$\Phi_{\rm ext}({\Omega}_{\rm ext}) = {\Omega}_{\rm ext}$, 
 we first observe that 
$\Phi_{\rm ext}(\overline{\Omega}_{\rm ext})  \subset \overline{\Omega}_{\rm ext}$.
We have indeed that, for any $y\in  \overline{\Omega}_{\rm ext}$,
$$
{\rm dist} \left(
y,\Phi_{\rm ext}( {\Omega}_{\rm ext}   )
\right)
=
{\rm dist} \left(
\Phi_j(x_j),\Phi_{\rm ext}( {\Omega}_{\rm ext}   )
\right)
\leq
{\rm dist} \left(
\Phi_j(x_j),\Phi_{\rm ext} (x_j)
\right)
\leq\delta_j,
$$
where $x_j\in  \overline{\Omega}_{\rm ext}$ is the preimage of $y$ through $\Phi_j$; since  the latter holds for any $j$, we must have
${\rm dist} \left( y,\Phi_{\rm ext}( \overline{\Omega}_{\rm ext}   )
\right) =0$.
Second, we observe that 
$ \overline{\Omega}_{\rm ext} \subset 
\Phi_{\rm ext}(\overline{\Omega}_{\rm ext})$. We find indeed that
for any $x \in  \overline{\Omega}_{\rm ext}$
$$
{\rm dist} \left(
\Phi_{\rm ext}(x), \overline{\Omega}_{\rm ext}   
\right)
=
{\rm dist} 
\left(
\Phi_{\rm ext}(x), \Phi_j(\overline{\Omega}_{\rm ext}   )
\right)
\leq 
\delta_j,
$$
for $j=1,2,\ldots$, which implies
${\rm dist} \left(
\Phi_{\rm ext}(x), \overline{\Omega}_{\rm ext} \right) = 0$.
In conclusion, we find 
$\Phi_{\rm ext}(\overline{\Omega}_{\rm ext})  = \overline{\Omega}_{\rm ext}$: since 
$\Phi_{\rm ext}( \partial \Omega_{\rm ext})  = \partial \Omega_{\rm ext}$, we must have
$\Phi_{\rm ext}(   \Omega_{\rm ext})  =   \Omega_{\rm ext}$.
\end{proof}

\section{Notation}
\label{sec:notation}

To ease the presentation, we distinguish between basic definitions, 
definitions associated with the parametric problem of interest,
definitions associated with the optimization statement, and definitions associated with the finite element approximation of the compositional map. Below, $\omega$ denotes a generic domain in $\mathbb{R}^2$,
which is either $\Omega$ or $\Omega_{\rm p}$ in the main body of the manuscript.
To shorten notation, the explicit dependence on the domain $\omega$  is omitted if not strictly necessary.

\noindent \textbf{Basic definitions.}
\begin{itemize}
\item
$\Omega\subset \mathbb{R}^2$ Lipschitz domain.
\item
 $\texttt{id}: \mathbb{R}^2 \to \mathbb{R}^2$ identity map.
 \item
 $\mathfrak{B}(\omega)$  set of Lipschitz bijections from the domain $\omega$ in itself.
  \item
 $\mathfrak{D}(\omega)$  set of diffeomorphisms from the domain $\omega$ in itself.
 \item
 $J(\Phi) = {\rm det} (\nabla \Phi)$ Jacobian determinant of the vector-valued field $\Phi: \omega \to \mathbb{R}^2$;
  $H(\Phi)$ Hessian of $\Phi$.
 \item
 $\mathbf{n}_{\omega}: \partial \omega \to \mathbb{S}^1=\{x\in \mathbb{R}^2: \|x\|_2 = 1\}$ outward normal to the domain $\Omega$.
 \item
${\mathfrak{U}}_0(\omega) = \{ \varphi \in C^1(\overline{\omega}; \mathbb{R}^2) \,: \, \varphi \cdot \mathbf{n} |_{\partial \omega} = 0 \}$.
\item
$\Omega_{\rm p}\subset \mathbb{R}^2$ polytope isomorphic to the domain $\Omega$.
\item
$V = \{ x_i^{\rm v} \}_{i=1}^{N_{\rm v}}$
minimal set of  vertices of 
 $\Omega_{\rm p}$.
\item
$\Psi: \Omega_{\rm p} \to \Omega$  bijection from the polytope $\Omega_{\rm p}$ to $\Omega$.
\item
 $\{F_j^{\rm p} \}_{j=1}^{N_{\rm f,p}}$   facets of $\Omega_{\rm p}$; 
$\{F_j = \Psi(F_j^{\rm p}) \}_{j=1}^{N_{\rm f,p}}$ mapped facets.
\item
$V_{\rm ang}(\omega)$  angular points of $\partial \omega$ (discontinuities of the normal $\mathbf{n}_{\omega}$).
\item
$\overline{\omega}$ closure of $\omega$.
\item
${\rm dist}_{\rm H}(\omega,\omega')$ Hausdorff distance  between $\omega$ and $\omega'$.
\item
${\rm col}(\mathbf{W})$  subspace of $\mathbb{R}^M$ spanned by the columns of the matrix $\mathbf{W}\in \mathbb{R}^{M\times m}$.
\end{itemize}

\noindent \textbf{Parametric problem.}
\begin{itemize}
\item
$\mu$ vector of parameters in the 
parameter domain $\mathcal{P} \subset \mathbb{R}^P$.
\item
$u:\Omega \times \mathcal{P} \to \mathbb{R}^{D_{\rm u}}$ parametric field of interest, $u_{\mu}(\cdot) := u(\cdot,\mu)$.
\item
$\mathcal{M}=\{ u_{\mu} : \mu\in \mathcal{P}\}$ solution manifold.
\item
$\widetilde{\mathcal{M}}=\{ 
{u}_{\mu}  \circ \Phi_{\mu}^{-1}: \mu\in \mathcal{P}\}$ mapped solution manifold.
\item
$\mathcal{T}_{\rm pb}$ high-fidelity mesh  with elements
$\{ \texttt{D}_k^{\rm pb}  \}_{k=1}^{N_{\rm e}^{\rm pb}}$
and facets
$\{ \Psi_k^{\rm pb}  \}_{k=1}^{N_{\rm e}^{\rm pb}}$ employed to approximate the elements of 
$\widetilde{\mathcal{M}}$.
\end{itemize}

\noindent 
\textbf{Optimization statement  (I): general  definitions and analysis.}
\begin{itemize}
\item
Minimization statement: \eqref{eq:tractable_optimization_based_registration}.
\item
$\mathfrak{f}^{\rm tg}:
\mathfrak{B}(\Omega) \times \mathcal{P} \to \mathbb{R}$
target function.
\item
$\mathfrak{f}_{\rm pen}:
\mathbb{R}^M \to \mathbb{R}$
penalty function.  
\item
$\texttt{N}: \mathbb{R}^M \to {\rm Lip}(\Omega; \mathbb{R}^2)$
mapping operator.
\item
$A$ (cf.
 \eqref{eq:admissible_maps_jacobian})
admissible set for the operator $\texttt{N}$.
\item
$A_{\rm jac}$
 (cf.
 \eqref{eq:admissible_maps_jacobian})
maps with strictly positive Jacobian determinant. 
\item
\eqref{eq:nonlinear_ansatz} 
compositional maps.
The pair $(\Omega_{\rm p},\Psi)$ satisfies Hypothesis
\ref{hyp:regular_polytopes}.
 \item
 $\mathfrak{f}_{\rm pen}^{\rm th}$  
(cf. \eqref{eq:f_pen_theory})
penalty function used for the analysis of section \ref{sec:affine_maps_polytopes}
 \item
\eqref{eq:nonlinear_ansatz_generalized}
multi-layer compositional maps.
The pairs $(\Omega_{{\rm p},1},\Psi_1),\ldots,(\Omega_{{\rm p},\ell},\Psi_{\ell})$ satisfy Hypothesis
\ref{hyp:regular_polytopes}.
\end{itemize}

\noindent 
\textbf{Finite-element discretization of the mapping operator} (cf. section \ref{sec:methods}).
 
\begin{itemize}
\item
$\mathbb{P}_{\kappa}(\mathbb{R}^2)$ space of two-variable polynomials of total degree less or equal to $\kappa$.
\item
$\widehat{\texttt{D}} = \{ x\in (0,1)^2: \sum_{i=1}^2 (x)_i < 1 \}$
reference (or master) element.
$\{ \ell_{i}^{\rm fe} \}_{i=1}^{n_{\rm lp}}$ 
Lagrangian basis
 of 
$\mathbb{P}_{\kappa}(\mathbb{R}^2)$ associated with the nodes
$\{ \tilde{x}_{i} \}_{i=1}^{n_{\rm lp}} \subset \overline{\widehat{\texttt{D}}}$;
the nodes include the vertices $\{ \tilde{x}_{i}^{\rm p} \}_{i=1}^{3}$ of the triangle  $\widehat{\texttt{D}}$.
\item
 $\{ \ell_{i}^{\rm fe,p} \}_{i=1}^{3}$ Lagrangian basis of 
 $\mathbb{P}_{1}$ associated with the 
  vertices $\{ \tilde{x}_{i}^{\rm p} \}_{i=1}^{3}$ of   $\widehat{\texttt{D}}$.
 \item
 $\mathcal{T}$ curved mesh employed to define the polytope $\Omega_{\rm p}$ (cf. \eqref{eq:geometric_mapping_polytope_coarsegrained}) and the mapping $\Psi$
(cf. \eqref{eq:geometric_mapping_Psi});
 $ \{  \texttt{D}_k \}_{k=1}^{N_{\rm e}}$   elements of the mesh, 
 $ \{  \texttt{F}_j \}_{j=1}^{N_{\rm f}}$   facets of the mesh, 
 $\{ x_{i,k}^{\rm hf} \}_{i,k}$ nodes of the mesh ($x_{i,k}^{\rm hf}$ is 
 the $i$-th node of the $k$-th element of  the mesh).
 \item
 $\Psi_k $ FE mapping from $\widehat{\texttt{D}}$ to the  $k$-th element of the mesh, $k=1, \ldots, N_{\rm e}$ 
 (cf. \eqref{eq:geometric_mapping_elemental_mapping}).
 \item
 $\Psi_{k,\rm p}$ FE mapping from $\widehat{\texttt{D}}$ to the  $k$-th linearized element of the mesh, $k=1, \ldots, N_{\rm e}$ (cf. \eqref{eq:geometric_mapping_elemental_mapping}).
 \item
 $\mathcal{T}_{\rm p}$ linear mesh with elements 
  $ \{  \texttt{D}_{k,\rm p} \}_{k=1}^{N_{\rm e}}$ and facets
    $ \{  \texttt{F}_{j,\rm p} \}_{k=1}^{N_{\rm f}}$. 
\item
$\mathcal{E}_{\rm p}^{\rm int} = \bigcup_{j\in \texttt{I}_{\rm int} } \texttt{F}_{j,\rm p}$ internal facets of 
$\mathcal{T}_{\rm p}$.
\item
$\mathbf{n}^+: \mathcal{E}_{\rm p}^{\rm int}  \to \mathbb{S}^1$ positive normal.
\item
$\mathcal{U}$ 
\eqref{eq:mapping_space}
FE displacement  space 
equipped with the inner product
\eqref{eq:mapping_inner_product}.  
\end{itemize} 
 
\noindent 
\textbf{Optimization statement  (II): practical definitions.}

\begin{itemize}
\item
$\mathfrak{f}_{\rm pen}$ penalty function
\eqref{eq:penalty_function_a}: 
$\mathfrak{P}$ \eqref{eq:penalty_term}
smoothing term;
$\mathfrak{f}_{\rm jac}$ 
\eqref{eq:f_jac_tmp}
bijectivity-enforcing penalty term;
$\mathfrak{f}_{\rm msh}$ 
\eqref{eq:f_msh}
discrete bijectivity-enforcing penalty term;
\item
\eqref{eq:optimization_georeg}
optimization problem for the construction of 
$\Psi$:
$\widetilde{\mathcal{W}}$ 
\eqref{eq:displacement_geometric_registration}
search space;
 $\mathfrak{P}_{\rm brkn}$ smoothing term \eqref{eq:penalty_term_broken}.
\item
 Target function based on 
 point-set sensor
 \eqref{eq:target_pointset};
  target function  based on distributed sensor \eqref{eq:distributed_target}, $\mathcal{S}_n\subset L^2(\Omega_{\rm p})$ template space for the distributed sensor
$s:\mathcal{P} \to L^2(\Omega_{\rm p})$.
\end{itemize}

\bibliographystyle{abbrv}	
\bibliography{all_refs}

\end{document}